\documentclass[11pt, a4paper,leqno]{amsart}
\usepackage{amsmath,amsthm,amscd,amssymb,amsfonts, amsbsy}
\usepackage{latexsym}
\usepackage{exscale}
\usepackage{mathtools}

\usepackage{pgf}
\calclayout
\allowdisplaybreaks

\newtheorem{theorem}{Theorem}[section]

\newtheorem{lemma}{Lemma}[section]

\theoremstyle{definition}
\newtheorem{definition}{Definition}
\newtheorem{remark}{Remark}[section]

\numberwithin{equation}{section}

\def\Xint#1{\mathchoice
{\XXint\displaystyle\textstyle{#1}}%
{\XXint\textstyle\scriptstyle{#1}}%
{\XXint\scriptstyle\scriptscriptstyle{#1}}%
{\XXint\scriptscriptstyle%
\scriptscriptstyle{#1}}%
\!\int}
\def\XXint#1#2#3{{\setbox0=\hbox{$#1{#2#3}{%
\int}$ }
\vcenter{\hbox{$#2#3$ }}\kern-.6\wd0}}
\def\barint{\,\Xint -}

\def\Yint#1{\mathchoice
    {\YYint\displaystyle\textstyle{#1}}%
    {\YYint\textstyle\scriptstyle{#1}}%
    {\YYint\scriptstyle\scriptscriptstyle{#1}}%
    {\YYint\scriptscriptstyle\scriptscriptstyle{#1}}%
      \!\iint}
\def\YYint#1#2#3{{\setbox0=\hbox{$#1{#2#3}{\iint}$}
    \vcenter{\hbox{$#2#3$}}\kern-.5\wd0}}
\def\longdash{{-}\mkern-3.5mu{-}}
\def\bariint{\, \Yint\longdash}

\def\Zint#1{\mathchoice
    {\ZZint\displaystyle\textstyle{#1}}%
    {\ZZint\textstyle\scriptstyle{#1}}%
    {\ZZint\scriptstyle\scriptscriptstyle{#1}}%
    {\ZZint\scriptscriptstyle\scriptscriptstyle{#1}}%
      \!\iiint}
\def\ZZint#1#2#3{{\setbox0=\hbox{$#1{#2#3}{\iiint}$}
    \vcenter{\hbox{$#2#3$}}\kern-.5\wd0}}
\def\loongdash{{-}\mkern-3mu{-}\mkern-3mu{-}}
\def\bariiint{\, \Zint\loongdash}

\newcommand{\beq}{\begin{equation}}
\newcommand{\bea}[1]{\begin{array}{#1} }
\newcommand{\eeq}{ \end{equation}}
\newcommand{\ea}{ \end{array}}

\def \d {{\delta}}

\def\mean#1{\mathchoice%
          {\mathop{\kern 0.2em\vrule width 0.6em height 0.69678ex depth -0.58065ex
                  \kern -0.8em \intop}\nolimits_{\kern -0.4em#1}}%
          {\mathop{\kern 0.1em\vrule width 0.5em height 0.69678ex depth -0.60387ex
                  \kern -0.6em \intop}\nolimits_{#1}}%
          {\mathop{\kern 0.1em\vrule width 0.5em height 0.69678ex
              depth -0.60387ex
                  \kern -0.6em \intop}\nolimits_{#1}}%
          {\mathop{\kern 0.1em\vrule width 0.5em height 0.69678ex depth -0.60387ex
                  \kern -0.6em \intop}\nolimits_{#1}}}

\def\vintslides_#1{\mathchoice%
          {\mathop{\kern 0.1em\vrule width 0.5em height 0.697ex depth -0.581ex
                  \kern -0.6em \intop}\nolimits_{\kern -0.4em#1}}%
          {\mathop{\kern 0.1em\vrule width 0.3em height 0.697ex depth -0.604ex
                  \kern -0.4em \intop}\nolimits_{#1}}%
          {\mathop{\kern 0.1em\vrule width 0.3em height 0.697ex depth -0.604ex
                  \kern -0.4em \intop}\nolimits_{#1}}%
          {\mathop{\kern 0.1em\vrule width 0.3em height 0.697ex depth -0.604ex
                  \kern -0.4em \intop}\nolimits_{#1}}}

\newcommand{\aveint}[2]{\mathchoice%
          {\mathop{\kern 0.2em\vrule width 0.6em height 0.69678ex depth -0.58065ex
                  \kern -0.8em \intop}\nolimits_{\kern -0.45em#1}^{#2}}%
          {\mathop{\kern 0.1em\vrule width 0.5em height 0.69678ex depth -0.60387ex
                  \kern -0.6em \intop}\nolimits_{#1}^{#2}}%
          {\mathop{\kern 0.1em\vrule width 0.5em height 0.69678ex depth -0.60387ex
                  \kern -0.6em \intop}\nolimits_{#1}^{#2}}%
          {\mathop{\kern 0.1em\vrule width 0.5em height 0.69678ex depth -0.60387ex
                  \kern -0.6em \intop}\nolimits_{#1}^{#2}}}

\def\eqn#1$$#2$${\begin{equation}\label#1#2\end{equation}}
\def\charfn_#1{{\raise1.2pt\hbox{$\chi
_{\kern-1pt\lower3pt\hbox{{$\scriptstyle#1$}}}$}}}
\def\diam{\operatorname{diam}}
\def\qq1{q_*}
\def\q2{q_{**}}

\newdimen\vintbar
\def\vint{-\kern-\vintbar\int}

\def\E{\mathcal E}

\def\L{\mathcal L}
\def\K{\mathcal K}

\def\P{\mathcal P}

\def\L{\mathcal L}

\def\0{\boldsymbol 0}

\newcommand{\R}{\mathbb R}

\newcommand{\eps}{\epsilon}

\newtoks\by
\newtoks\paper
\newtoks\book
\newtoks\jour
\newtoks\yr
\newtoks\pages
\newtoks\vol
\newtoks\publ

\def\name[#1, #2]{#1 #2}
\def\ota{{\hbox{\bf ???}}}
\def\cLear{\by=\ota\paper=\ota\book=\ota\jour=\ota\yr=\ota
\pages=\ota\vol=\ota\publ=\ota}
\def\endpaper{\the\by, \textit{\the\paper},
{\the\jour} \textbf{\the\vol} (\the\yr), \the\pages.\cLear}
\def\endbook{\the\by, \textit{\the\book},
\the\publ, \the\yr.\cLear}
\def\endpap{\the\by, \textit{\the\paper}, \the\jour.\cLear}
\def\endproc{\the\by, \textit{\the\paper}, \the\book, \the\publ,
\the\yr, \the\pages.\cLear}

\renewcommand{\d}{\, \mathrm{d}} %differential

\begin{document}
\title[$L^p$ Dirichlet problem for elliptic and parabolic operators]{A structure theorem for elliptic and parabolic operators with applications to  homogenization\\ of operators of Kolmogorov type}

%\address{Steve Hofmann\\Department of Mathematics, University of Missouri, Columbia, MO 65211, USA} \email{hofmanns@missouri.edu}

\address{Malte Litsg{\aa}rd \\Department of Mathematics, Uppsala University\\
S-751 06 Uppsala, Sweden}
\email{malte.litsgard@math.uu.se}

\address{Kaj Nystr\"{o}m\\Department of Mathematics, Uppsala University\\
S-751 06 Uppsala, Sweden}
\email{kaj.nystrom@math.uu.se}

\thanks{K. N was partially supported by grant  2017-03805 from the Swedish research council (VR)}

\author{Malte Litsg{\aa}rd and Kaj Nystr{\"o}m}
\maketitle
\begin{abstract}
\noindent \medskip
We consider the operators
                  \begin{eqnarray*}
    \nabla_X\cdot(A(X)\nabla_X),\  \nabla_X\cdot(A(X)\nabla_X)-\partial_t,\ \nabla_X\cdot(A(X)\nabla_X)+X\cdot\nabla_Y-\partial_t,
    \end{eqnarray*}
    where $X\in \Omega$, $(X,t)\in \Omega\times \mathbb R$ and $(X,Y,t)\in \Omega\times \mathbb R^m\times \mathbb R$, respectively, and where  $\Omega\subset\mathbb R^m$ is  a (unbounded) Lipschitz domain with defining function $\psi:\mathbb R^{m-1}\to\mathbb R$ being Lipschitz with constant bounded by $M$.  Assume that
    the elliptic  measure associated to the first of these operators is mutually absolutely continuous with respect to the surface measure
    $\d \sigma(X)$, and that the corresponding Radon-Nikodym derivative or Poisson kernel satisfies
    a scale invariant reverse H{\"o}lder inequality in $L^p$, for some fixed $p$, $1<p<\infty$,  with constants depending only on the constants of $A$, $m$ and the Lipschitz constant of $\psi$, $M$. Under this assumption we prove that then the same conclusions are also true  for the  parabolic measures associated to the second and third operator with $\d \sigma(X)$ replaced by  the surface measures
    $\d \sigma(X)\d t$ and $\d \sigma(X)\d Y\d t$, respectively. This structural theorem allows us to reprove several results previously established in the literature as well as to deduce new results in, for example, the context of homogenization for operators of Kolmogorov type.  Our proof of the structural theorem is based on recent results established by the authors concerning boundary Harnack inequalities for operators of Kolmogorov type in divergence form with bounded, measurable and uniformly elliptic coefficients.
    \\

\noindent

\medskip

\end{abstract}

%\newpage
%\tableofcontents

%\newpage

    \setcounter{equation}{0} \setcounter{theorem}{0}
%Let $\Omega\subset\mathbb R^n$ be a bounded domain, i.e., a bounded, open and connected set.
    \section{Introduction}

    Let $\Omega\subset\mathbb R^m$, $m\geq 2$, be a (unbounded) Lipschitz domain
         \begin{eqnarray}\label{Lip}
 \Omega=\{ X=(x,x_{m})\in\mathbb R^{m-1}\times\mathbb R\mid  x_m>\psi(x)\},
    \end{eqnarray}
     where $\psi:\mathbb R^{m-1}\to\mathbb R$ is Lipschitz with constant bounded by $M$.  Let $A=A(X)=\{a_{i,j}(X)\}$ be a real $m\times m$ matrix-valued, measurable function such that $A(X)$ is symmetric and
    \begin{equation}\label{eq:Aellip}
  \kappa^{-1} |\xi|^2 \leq \sum_{i,j=1}^{m} a_{i,j}(X)\xi_i \xi_j \leq \kappa |\xi|^2,
\end{equation}
for some $1 \leq \kappa < \infty$ and for all $\xi \in \R^{m}$, $X\in\mathbb R^m$. We consider the divergence form operators
                  \begin{eqnarray*}
    &&\L_{\mathcal{E}}:=\nabla_X\cdot(A(X)\nabla_X),\notag\\
    &&\L_{\mathcal{P}}:=\nabla_X\cdot(A(X)\nabla_X)-\partial_t,\notag\\
    &&\L_{\mathcal{K}}:= \nabla_X\cdot(A(X)\nabla_X)+X\cdot\nabla_Y-\partial_t,
    \end{eqnarray*}
     in $\mathbb R^{2m+1}$, $m\geq 1$, equipped with coordinates $(X,Y,t):=(x_1,...,x_{m},y_1,...,y_{m},t)\in \mathbb R^{m}\times\mathbb R^{m}\times\mathbb R$. Obviously
     $\L_{\mathcal{E}}$ only makes reference to the $X$ coordinate, $\L_{\mathcal{P}}$ makes reference to the $X$ and $t$ coordinates and $\L_{\mathcal{K}}$ makes reference to all coordinates. The subscripts $\mathcal{E}$, $\mathcal{P}$, $\mathcal{K}$, refer to Elliptic, Parabolic and Kolmogorov.

     $\L_{\mathcal{E}}$ is the standard second order elliptic PDE with only measurable, bounded and uniformly elliptic coefficients much studied ever since the breakthroughs of Moser, Nash, DeGiorgi and others. $\L_{\mathcal{P}}$ is the corresponding parabolic version, and $\L_{\mathcal{K}}$ is an operator of Kolmogorov type in divergence form which up to now has only been modestly studied and understood.  Recently, in  \cite{Ietal} the authors extended the  De Giorgi-Nash-Moser (DGNM) theory, which in its original form only considers elliptic
or parabolic equations in divergence form, to (hypoelliptic) equations with rough coefficients including the operator $\L_{\mathcal{K}}$ assuming \eqref{eq:Aellip}. Their result is the correct scale- and translation-invariant estimates for local H{\"o}lder continuity and the Harnack inequality for weak solutions.

     To give some perspective on the operator $\L_{\mathcal{K}}$, recall that  the operator
    $$\K:=\nabla_X\cdot \nabla_X+X\cdot\nabla_Y-\partial_t$$
      was originally introduced and studied by
Kolmogorov, see \cite{K}. Kolmogorov noted that $\K$ is an example of a degenerate parabolic
operator having strong regularity properties and he proved that $\K$ has a fundamental solution which is smooth off its diagonal. Today, using the terminology introduced by H{\"o}rmander, see \cite{H}, we can conclude that $\K$ is hypoelliptic. Naturally, for the operator $\L_{\mathcal{K}}$, assuming only measurable coefficients and \eqref{eq:Aellip}, the methods of Kolmogorov and H{\"o}rmander can not be directly applied to establish the DGNM theory and related estimates.

In this paper we are interested in the $L^p$ Dirichlet problem for the operators $\L_{\mathcal{E}}$, $\L_{\mathcal{P}}$, $\L_{\mathcal{K}}$ in the (unbounded) Lipschitz domains
$\Omega$, $\Omega\times \mathbb R$ and $\Omega\times \mathbb R^m\times \mathbb R$ respectively, and where
$X\in \Omega$, $(X,t)\in \Omega\times \mathbb R$ and $(X,Y,t)\in \Omega\times \mathbb R^m\times \mathbb R$.  In particular, we consider the operators
$\L_{\mathcal{P}}$ and  $\L_{\mathcal{K}}$ in $t$-independent and $(Y,t)$-independent domains, respectively. We introduce a (physical) measure
  $\sigma_{\K}$ on $\partial\Omega\times \mathbb R^m\times \mathbb R$,
 \begin{eqnarray}\label{surfac+}
 \quad\quad \d\sigma_{\K}(X,Y,t):=\sqrt{1+|\nabla_{x}\psi(x)|^2}\d x\d Y\d t,\ (X,Y,t)\in\partial\Omega\times \mathbb R^m\times \mathbb R.
 \end{eqnarray}
 We refer to  $\sigma_{\K}$ as the surface measure on $\partial\Omega\times \mathbb R^m\times \mathbb R$ where the subscript $\K$ indicates that we consider a setting appropriate for operators of Kolmogorov type. The corresponding measures relevant for $\L_{\mathcal{E}}$ and $\L_{\mathcal{P}}$ are $\sigma_\E$ and $\sigma_\P$,
 \begin{eqnarray}\label{surfac++}
  \d\sigma_{\E}(X):=\sqrt{1+|\nabla_{x}\psi(x)|^2}\d x,\  \d\sigma_{\P}(X,t):=\d \sigma_{\mathcal{E}}(X)\d t,
 \end{eqnarray}
 where $X\in\partial\Omega$ and $(X,t)\in\partial\Omega\times\mathbb R$, respectively.

The main results of the paper are
Theorem \ref{Ainfty}, Theorem \ref{Ainfty+} and Theorem \ref{DPequiv} stated in Section \ref{secstruc} below.  Using these theorems we can derive new results concerning the $L^p$ Dirichlet problem for $\L_\K$ using results previously only proved for $\L_\E$ or $\L_\P$, and we can also conclude that some results proved in the literature concerning $\L_\P$ are straightforward consequences of the corresponding results for $\L_\E$. In particular, the main result of Fabes and Salsa \cite{FSa} concerning parabolic measure is a consequence of the classical result of Dahlberg \cite{D} concerning harmonic measure. Our proofs of Theorem \ref{Ainfty}, Theorem \ref{Ainfty+} and Theorem \ref{DPequiv} are based on our recent results in \cite{LN} concerning boundary Harnack inequalities for operators of Kolmogorov type in divergence form with bounded, measurable and uniformly elliptic coefficients.

 Theorem \ref{Ainfty}, Theorem \ref{Ainfty+} and Theorem \ref{DPequiv}, and their consequences, are deduced under the assumptions
 \begin{align}\label{keyassump}
 (i)&\mbox{ $\Omega\subset\mathbb R^{m}$ is a (unbounded)  Lipschitz domain  with constant $M$},\notag\\
 (ii)&\mbox{ $A$ satisfies \eqref{eq:Aellip} with constant $\kappa$},\notag\\
 (iii)&\mbox{ $A$ satisfies the qualitative assumptions stated in \eqref{compre} and \eqref{eq2+re} below}.
 \end{align}
 All quantitative estimates will only depend on $m$, $\kappa$ and $M$ and Theorem \ref{Ainfty} and Theorem \ref{Ainfty+} are by their nature of local character. However, we have chosen to state our results in the unbounded geometric setting $ \Omega\times \mathbb R^m\times \mathbb R$. To avoid being diverted by additional technical issues caused by the unbounded setting we assume \eqref{compre}. \eqref{eq2+re} is only imposed to ensure that all results (e.g. the existence of fundamental solutions) and all estimates used in the paper can be found in the literature. {One can dispense of \eqref{eq2+re}  at the expense of additional arguments.}

We consider the following problems and  we refer to the bulk of the paper for all definitions and in particular for the definition of weak solutions to
$\L_\K u =0$ in $\Omega\times \mathbb R^m\times \mathbb R$.

\begin{definition}\label{solvability} Assume that $\Omega\subset\mathbb R^{m}$ is a (unbounded)  Lipschitz domain  with constant $M$. Assume that $A$ satisfies \eqref{eq:Aellip} with constant $\kappa$, and \eqref{compre}. Given $p\in (1,\infty)$ we say that the Dirichlet problem for $\L_\K u =0$ in $\Omega\times \mathbb R^m\times \mathbb R $ is solvable in $L^p(\partial \Omega\times \mathbb R^m\times \mathbb R,\d\sigma_\K)$ if there exists, for every $f \in L^p(\partial \Omega\times \mathbb R^m\times \mathbb R,\d\sigma_\K)$, a weak solution to the Dirichlet problem
\begin{equation*}\label{dpweak}
\begin{cases}
	\L_{\K} u = 0  &\text{in} \ \Omega\times \mathbb R^m\times \mathbb R, \\
      u = f  & \text{n.t. on} \ \partial \Omega\times \mathbb R^m\times \mathbb R,
\end{cases}
\end{equation*}
and a constant $c$, depending only on $m$, $\kappa$, $M$ and $p$, such that
\[
\|N(u)\|_{L^p(\partial \Omega\times \mathbb R^m\times \mathbb R,\d\sigma_\K)}
	 \le c \|f\|_{L^p(\partial \Omega\times \mathbb R^m\times \mathbb R,\d\sigma_\K)},
\]
where $N(u)$ is introduced in subsection \ref{ntmaxop}. For short we say that $D_{\K}^p(\partial \Omega\times \mathbb R^m\times \mathbb R,\d\sigma_\K)$ is solvable. If the solution is unique then we say that the  Dirichlet problem for $\L_\K u =0$ in $\Omega $ is uniquely solvable in $L^p(\partial \Omega\times \mathbb R^m\times \mathbb R,\d\sigma_\K)$. For short we write that  $D_{\K}^p(\partial \Omega\times \mathbb R^m\times \mathbb R,\d\sigma_\K)$ is uniquely solvable. The notion  that $D_{\E}^p(\partial \Omega,\d\sigma_\E)$ and $D_{\P}^p(\partial \Omega\times  \mathbb R,\d\sigma_\P)$ are uniquely solvable are defined analogously.
\end{definition}

Using our structural theorems (i.e. combining Theorem \ref{Ainfty}, Theorem \ref{Ainfty+} and Theorem \ref{DPequiv}) we can conclude that if  $D_{\E}^p(\partial \Omega,\d\sigma_\E)$  is uniquely solvable for some $p\in (1,\infty)$, then also
$D_{\K}^p(\partial \Omega\times \mathbb R^m\times \mathbb R,\d\sigma_\K)$  is uniquely solvable. We can use this insight to state a number of results
concerning the solvability of $D_{\K}^p(\partial \Omega\times \mathbb R^m\times \mathbb R,\d\sigma_\K)$ and in particular we can conclude the following.

\begin{theorem}\label{per1-} Assume \eqref{keyassump}. Assume also
\begin{equation}\label{eq:Aperiod-}
  A(x,x_{m})=A(x), \quad x \in \R^{m-1}, \ x_{m} \in \R,
\end{equation}
i.e. $A$ is independent of $x_m$. Then there exists $\delta=\delta(m,\kappa,M)\in (0,1)$ such that if $2-\delta<p<\infty$, then
$D_{\K}^p(\partial \Omega\times \mathbb R^m\times \mathbb R,\d\sigma_\K)$ is uniquely solvable.
\end{theorem}

\begin{theorem}\label{per1} Assume \eqref{keyassump}. Assume also
\begin{equation}\label{eq:Aperiod}
  A(x,x_{m} + 1)=A(x,x_{m}), \quad x \in \R^{m-1}, \ x_{m} \in \R,
\end{equation}
i.e. $A$ is 1-periodic in $x_m$, and that $A$ satisfies a Dini-type condition in the $x_{m}$ variable,
\begin{equation}\label{eq:Dini}
  \int_0^1 \frac{\theta(\rho)^2}{\rho} \d\rho < \infty,
\end{equation}
where $\theta(\rho):= \sup \{|A(x,\lambda_1) - A(x,\lambda_2)| \mid  x\in \R^{m-1}, \ |\lambda_1 - \lambda_2| \leq \rho \}$.  Then there exists $\delta=\delta(m,\kappa,M)\in (0,1)$ such that if $2-\delta<p<\infty$, then
$D_{\K}^p(\partial \Omega\times \mathbb R^m\times \mathbb R,\d\sigma_\K)$ is uniquely solvable.
\end{theorem}

Using our structural theorems it follows that Theorem \ref{per1-} is a consequence of \cite{JK} and that Theorem \ref{per1} is a consequence of \cite{KS}. By the same argument we can conclude that the main result in \cite{FSa} is a consequence of \cite{D} and that the main result in \cite{CS} is a consequence of \cite{KS}.

With Theorem \ref{per1} in place we are also able to analyze a homogenization problem for operators of Kolmogorov type. In this case we assume, in addition to \eqref{eq:Aellip},  that
\begin{equation}\label{perA}
A(X+Z)=A(X),\quad \text{for all }Z\in \mathbb Z^{m},
\end{equation}
and that
\begin{equation}\label{DiniA}
  \int_0^1 \frac{\Theta(\rho)^2}{\rho} \d\rho < \infty,
\end{equation}
where $\Theta(\rho):= \sup \{|A(X) - A(\tilde X)| \mid X,\tilde X\in \R^{m}, \ |X-\tilde X| \leq \rho \}$. That is, $A$ is periodic with respect to the lattice $\mathbb Z^{m}$ and $A$ satisfies a Dini condition in all variables.

We consider, for $\eps>0$, the operator $\L_\E^\epsilon$,
\begin{equation}
    \L_\E^\epsilon :=\nabla_X\cdot(A^\epsilon(X)\nabla_X),\  A^\eps(X) := A({X}/{\eps}).
\end{equation}
Let
\[
\bar \L_\E := \nabla_X\cdot(\bar A\nabla_X),
\]
where the matrix $\bar A$ is determined by
\begin{equation}\label{barA}
\bar{A}\alpha := \int_{(0,1)^m}A(X)\nabla_X w_\alpha(X) \d X,
\quad \alpha\in \R^{m},
\end{equation}
and the auxiliary function $w_\alpha$ solves the problem
\begin{equation*}
\left\{\begin{array}{l}
\nabla_X\cdot\left(A(X)\nabla_X w_\alpha(X)\right)=0 \text{ in }(0,1)^{m},
\vspace{0.25cm} \\
w_\alpha(X)-\alpha X\text{ is }1-\text{periodic (in all variables)}, \vspace{0.25cm} \\
\displaystyle \int_{(0,1)^{m}}(w_\alpha(X)-\alpha X)\d X=0.
\end{array}\right.
\end{equation*}
Finally, we also introduce
\begin{equation}\label{Lepsilon}
    \L_\K^\epsilon :=\L_\E^\epsilon+{X}\cdot\nabla_Y-\partial_t,\ \bar\L_\K :=\bar\L_\E+{X}\cdot\nabla_Y-\partial_t.
\end{equation}

We prove the following homogenization result.
\begin{theorem}\label{th:homogenization}
Assume \eqref{keyassump}.  Assume also \eqref{perA} and \eqref{DiniA}.  Then there exists $\delta=\delta(m,\kappa,M)\in (0,1)$ such that the following is true. Consider $\epsilon>0$.
Given $p$, $2-\delta<p<\infty$, $f\in L^p(\partial \Omega\times \mathbb R^m\times \mathbb R,\d\sigma_\K)$, there exists a unique weak solution $u_\epsilon$ to the Dirichlet problem
\begin{equation*}
\begin{cases}
	\L_{\K}^\epsilon u_\epsilon = 0  &\text{in} \ \Omega\times \mathbb R^m\times \mathbb R, \\
      u_\epsilon = f  & \text{n.t. on} \ \partial \Omega\times \mathbb R^m\times \mathbb R,
\end{cases}
\end{equation*}
and  a constant $c=c(m,\kappa,M,p)$, $1\leq c<\infty$, such that
\[
\|N(u_\epsilon)\|_{L^p(\partial \Omega\times \mathbb R^m\times \mathbb R,\d\sigma_\K)}
	 \le c \|f\|_{L^p(\partial \Omega\times \mathbb R^m\times \mathbb R,\d\sigma_\K)}.
\]
Moreover, $u_\eps \to \bar u$ locally uniformly in $\Omega\times \mathbb R^m\times \mathbb R$ as $\eps\to 0$, and $\bar u$ is the unique weak solution to the Dirichlet problem
\begin{equation}\label{homogeneous}
\begin{cases}
	\bar\L_{\K}\bar u = 0  &\text{in} \ \Omega\times \mathbb R^m\times \mathbb R, \\
      \bar u = f  & \text{n.t. on} \ \partial \Omega\times \mathbb R^m\times \mathbb R,
\end{cases}
\end{equation}
and there exists a constant $c=c(m,\kappa,M,p)$, $1\leq c<\infty$, such that
\[
\|N(\bar u)\|_{L^p(\partial \Omega\times \mathbb R^m\times \mathbb R,\d\sigma_\K)}
	 \le c \|f\|_{L^p(\partial \Omega\times \mathbb R^m\times \mathbb R,\d\sigma_\K)}.
\]
\end{theorem}

 Theorem \ref{per1}, and the first part of Theorem \ref{th:homogenization}, was proved  in \cite{KS} for $\L_\E$. In \cite{KS} also the Neumann and regularity problems are treated. The theory for the Neumann and regularity problems is based on the use of integral identities to estimate certain non-tangential maximal functions. Homogenization of Neumann and regularity problems for $\L_\P$ and $\L_\K$ remain interesting  open problems.

% as outlined in the beginning of Section \ref{sec1.5} below

 To be clear, the main idea of this paper is that results concerning the $L^p$ Dirichlet problem for the operator $\L_\K$  in domains  $\Omega\times \mathbb R^m\times \mathbb R$ (and
for the operator $\L_\P$  in domains  $\Omega\times \mathbb R$) can be derived from the corresponding results for the operator $\L_\E$  in $\Omega$, using boundary estimates and in particular boundary Harnack inequalities for the operator $\L_\K$ $(\L_\P)$. In the case of $\L_\K$ the latter results are established in \cite{LN}. However, the relevant  results in \cite{LN} hold for   more general operators \begin{eqnarray*}
   \nabla_X\cdot(A(X,Y,t)\nabla_X)+X\cdot\nabla_Y-\partial_t,
    \end{eqnarray*}
    and in the more general class of domains
\begin{eqnarray*}
 \{(X,Y,t)=(x,x_{m},y,y_{m},t)\in\mathbb R^{2m+1} \mid \ x_m>\tilde\psi(x,y,t)\}.
    \end{eqnarray*}
    In particular, in \cite{LN} we allow for $(Y,t)$-dependent coefficients and domains. Therefore, one can repeat the analysis of this paper, taking any result concerning the solvability of the $L^p$ Dirichlet problem for parabolic operators
    \begin{eqnarray*}
   \nabla_X\cdot(A(X,t)\nabla_X)-\partial_t,
    \end{eqnarray*}
    in Lip(1,1/2) domains, as the point of departure. The results are the corresponding results for the operator $$\nabla_X\cdot(A(X,t)\nabla_X)+X\cdot\nabla_Y-\partial_t$$ in $Y$-independent Lipschitz type domains. Similarly, focusing only on $\L_\E$ and $\L_\P$ one can replace $\Omega\subset \mathbb R^m$ by a NTA-domain in the sense of \cite{JK1}, having a $(m-1)$-dimensional Ahlfors-regular boundary in the sense of \cite{DS}, \cite{DS1} (see also \cite{DJ}).

 The rest of the paper is organized as follows. In Section \ref{secprim}, which is of more preliminary nature, we introduce notation and state definitions including the notion of weak solutions. In this section we also discuss the Dirichlet problem, see Theorem \ref{EUDP}, and we point out that in Theorem 1.3 and Theorem 1.4 in \cite{LN1} we simply missed stating the obvious restriction $u\in L^\infty(\Omega\times \mathbb R^m\times \mathbb R)$ under which the proofs in \cite{LN1} are given. With this clarification,
Theorem \ref{EUDP} is a special case of Theorem 1.4 in \cite{LN1}. In Section
 \ref{secstruc} we state  our structural theorems: Theorem \ref{Ainfty}, Theorem \ref{Ainfty+} and Theorem \ref{DPequiv}.  In Section \ref{sec1.5} we state a number of lemmas concerning the interior regularity of weak solutions and concerning the boundary behaviour of non-negative solutions to
 $\L_\K u=0$; the latter were recently established in \cite{LN}. In Section \ref{sec2} we prove Theorem \ref{Ainfty} and Theorem \ref{Ainfty+}. In Section \ref{sec3} we prove Theorem \ref{DPequiv}
and hence, as outlined above and as a consequence, we prove Theorem \ref{per1-}  and Theorem \ref{per1}.  In Section \ref{fabes} we also give,  as we believe that the argument may be of independent interest in the case of operators of Kolmogorov type, a proof of Theorem \ref{per1-} using Rellich type inequalities along the proof of the corresponding result for the heat equation in \cite{FSa}. In Section \ref{sec4}  we apply our findings to homogenization, giving new results for homogenization of operators of Kolmogorov type, and in particular we prove Theorem \ref{th:homogenization}.

\section{Preliminaries}\label{secprim}

\subsection{Group law and metric} The natural family of dilations jointly for the operators $\L_{\mathcal E}$, $\L_{\mathcal P}$, $\L_\K$, $(\delta_r)_{r>0}$, on $\R^{N+1}$, $N:=2m$, is defined by
\begin{equation}\label{dil.alpha.i}
 \delta_r (X,Y,t) =(r X, r^3 Y,r^2 t),
\end{equation}
for $(X,Y,t) \in \R^{N +1}$,  $r>0$.   Furthermore, the classes of operators $\L_{\mathcal E}$, $\L_{\mathcal P}$, $\L_\K$ are closed under the  group law
\begin{equation}\label{e70}
(\tilde X,\tilde Y,\tilde t)\circ (X, Y,t)=(\tilde X+X,\tilde Y+Y-t\tilde X,\tilde t+t),
\end{equation}
where $(X,Y,t),\ (\tilde X,\tilde Y,\tilde t)\in \R^{N+1}$. Note that
\begin{equation}\label{e70+}
(X,Y,t)^{-1}=(-X,-Y-tX,-t),
\end{equation}
and hence
\begin{equation}\label{e70++}
(\tilde  X,\tilde  Y,\tilde  t)^{-1}\circ (X,Y,t)=(X-\tilde  X,Y-\tilde  Y+(
t-\tilde  t)\tilde  X,t-\tilde  t),
\end{equation}
whenever $(X,Y,t),\ (\tilde X,\tilde Y,\tilde t)\in \R^{N+1}$.

Given $(X,Y,t)\in \R^{N+1}$ we let \begin{equation}\label{kolnormint}
\|(X,Y, t)\|:=|(X,Y)|\!+|t|^{\frac{1}{2}},\ |(X,Y)|:=\big|X\big|+\big|Y\big|^{1/3}.
\end{equation}
We recall the following pseudo-triangular
inequalities: there exists a positive constant ${c}$ such that
\begin{eqnarray}\label{e-ps.tr.in}
\quad \|(X,Y,t)^{-1}\|\le {c}  \| (X,Y,t) \|,\ \|(X,Y,t)\circ (\tilde  X,\tilde  Y,\tilde  t)\| \le  {c}  (\| (X,Y,t) \| + \| (\tilde  X,\tilde  Y,\tilde  t)
\|),
\end{eqnarray}
whenever $(X,Y,t),(\tilde  X,\tilde  Y,\tilde  t)\in \R^{N+1}$. Using \eqref{e-ps.tr.in} it  follows immediately that
\begin{equation} \label{e-triangularap}
    \|(\tilde  X,\tilde  Y,\tilde  t)^{-1}\circ (X,Y,t)\|\le c \, \|(X,Y,t)^{-1}\circ (\tilde  X,\tilde  Y,\tilde  t)\|,
\end{equation}
whenever $(X,Y,t),(\tilde  X,\tilde  Y,\tilde  t)\in \R^{N+1}$. Let
\begin{equation}\label{e-ps.distint}
    d((X,Y,t),(\tilde X,\tilde Y,\tilde t)):=\frac 1 2\bigl( \|(\tilde X,\tilde Y,\tilde t)^{-1}\circ (X,Y,t)\|+\|(X,Y,t)^{-1}\circ (\tilde X,\tilde Y,\tilde t)\|).
\end{equation}
Using \eqref{e-triangularap} it follows that
\begin{equation}\label{e-ps.dist}
   \|(\tilde X,\tilde Y,\tilde t)^{-1}\circ (X,Y,t)\|\approx d((X,Y,t),(\tilde X,\tilde Y,\tilde t))\approx \|(X,Y,t)^{-1}\circ (\tilde X,\tilde Y,\tilde t)\|,
\end{equation}
with constants of comparison independent of $(X,Y,t),(\tilde X,\tilde Y,\tilde t)\in \R^{N+1}$. Again using \eqref{e-ps.tr.in} we also see that
\begin{equation} \label{e-triangular}
    d((X,Y,t),(\tilde X,\tilde Y,\tilde t))\le {c} \bigl(d((X,Y,t),(\hat X, \hat Y,\hat t))+d((\hat X, \hat Y,\hat
t),(\tilde X,\tilde Y,\tilde t))\bigr ),
\end{equation}
whenever $(X,Y,t),(\hat X, \hat Y,\hat t),(\tilde X,\tilde Y,\tilde t)\in \R^{N+1}$, and hence $d$ is a symmetric quasi-distance. Based on $d$ we introduce the balls
\begin{equation}\label{e-BKint}
    \mathcal{B}_r(X,Y,t):= \{ (\tilde X,\tilde Y,\tilde t) \in\mathbb R^{N+1} \mid d((\tilde X,\tilde Y,\tilde t),(X,Y,t)) <
r\},
\end{equation}
for $(X,Y,t)\in \R^{N+1}$ and $r>0$.  The measure of the ball $\mathcal{B}_r(X,Y,t)$ is $|\mathcal{B}_r(X,Y,t)|=c(m)r^{{\bf q}}$, where ${\bf q}:=4m+2$.

\subsection{Surface cubes and reference points} Let $\Omega\subset\mathbb R^m$, $m\geq 2$, be a (unbounded) Lipschitz domain as defined in \eqref{Lip} and with constant $M$. Let
\begin{equation} \label{space}\Sigma:=\partial \Omega\times \R^m\times\R =\{(x,x_{m},y,y_{m},t)\in\mathbb R^{N+1}\mid  x_m=\psi(x)\}.\end{equation}
An observation is that $(\Sigma,d,\d\sigma_\K)$ is a space of homogeneous type in the sense of \cite{CW}, with homogeneous dimension  ${\bf q}-1$. Furthermore, $(\mathbb R^{N+1},d,\d X\d Y\d t)$ is also a space of homogeneous type in the sense of \cite{CW}, but with homogeneous dimension  ${\bf q}$.

Let $$Q:=(-1,1)^m\times (-1,1)^m\times (-1,1)$$ and
$$Q_r=\delta_r Q:=\{(r X, r^3 Y,r^2 t)\mid  (X,Y,t)\in Q\}.$$
Given a point $(X_0,Y_0,t_0)\in \R^{N+1}$ we let
$$ Q_r(X_0,Y_0,t_0):=(X_0,Y_0,t_0)\circ Q_r:=\{(X_0,Y_0,t_0)\circ(X,Y,t)\mid  (X,Y,t)\in Q_r\}.$$
Furthermore, if $(X_0,Y_0,t_0)\in \partial\Omega\times \mathbb R^m\times \mathbb R$ then we set
$$\Delta_r(X_0,Y_0,t_0):=(\partial\Omega\times \mathbb R^m\times \mathbb R)\cap Q_r(X_0,Y_0,t_0).$$
We will frequently, and for brevity, write $ Q_r$ and $\Delta_r$ for $ Q_r(X_0,Y_0,t_0)$ and $\Delta_r(X_0,Y_0,t_0)$ whenever the point $(X_0,Y_0,t_0)$ is clear from the context. At instances we will simply also write $\Delta$ for $\Delta_r(X_0,Y_0,t_0)$ whenever the point $(X_0,Y_0,t_0)$ and the scale $r$ does not have to be stated explicitly. Given a positive constant $c$, $c\Delta:=\Delta_{cr}(X_0,Y_0,t_0)$.

Given $\varrho>0$ and $\Lambda>0$, we let
\begin{equation}\label{pointsref2}
\begin{split}
    A_{\varrho,\Lambda}^+&:= \left(0,\Lambda\varrho,0, -\tfrac 2 3\Lambda\varrho^3,\varrho^2\right)\in\mathbb R^{m-1}\times\mathbb R\times\mathbb R^{m-1}\times\mathbb R\times\mathbb R,\\
    A_{\varrho,\Lambda}^-&:= \left(0,\Lambda\varrho,0, \tfrac 2 3\Lambda\varrho^3,-\varrho^2\right)\in\mathbb R^{m-1}\times\mathbb R\times\mathbb R^{m-1}\times\mathbb R\times\mathbb R,
\end{split}
\end{equation}
and $$A_{\varrho,\Lambda}^\pm(X_0,Y_0,t_0):=(X_0,Y_0,t_0)\circ A_{\varrho,\Lambda}^\pm,$$
whenever $(X_0,Y_0,t_0)\in\mathbb R^{N+1}$. Furthermore, given $\Delta:=\Delta_r(X_0,Y_0,t_0)$ we let
$$A_{\Delta,\Lambda}^\pm:=A_{r,\Lambda}^\pm(X_0,Y_0,t_0).$$

\subsection{Qualitative assumptions on the coefficients} Central to our arguments are the boundary estimates recently proved  in \cite{LN}.  In \cite{LN} we consider solutions to the equation $\L u=0$ where $\L$ is the operator
              \begin{eqnarray}\label{e-kolm-ndre}
   \nabla_X\cdot(A(X,Y,t)\nabla_X)+X\cdot\nabla_Y-\partial_t,
    \end{eqnarray}
    in $\mathbb R^{N+1}$, $N=2m$, $m\geq 1$, $(X,Y,t):=(x_1,...,x_{m},y_1,...,y_{m},t)\in \mathbb R^{m}\times\mathbb R^{m}\times\mathbb R$. We assume that $$A=A(X,Y,t)=\{a_{i,j}(X,Y,t)\}_{i,j=1}^{m}$$ is a real-valued, $m\times m$-dimensional, symmetric matrix valued function satisfying
    \begin{eqnarray}\label{eq2re}
      \kappa^{-1}|\xi|^2\leq \sum_{i,j=1}^{m}a_{i,j}(X,Y,t)\xi_i\xi_j,\quad \ \ |A(X,Y,t)\xi\cdot\zeta|\leq \kappa|\xi||\zeta|,
    \end{eqnarray}
    for some $\kappa\in [1,\infty)$, and for all $\xi,\zeta\in \mathbb R^{m}$, $(X,Y,t)\in\mathbb R^{N+1}$. Throughout \cite{LN} we also assume that
    \begin{eqnarray}\label{compre}
    \quad A=A(X,Y,t)\equiv I_m\mbox{ outside some arbitrary but fixed compact subset of $\mathbb R^{N+1}$},
    \end{eqnarray}
    and that
        \begin{eqnarray}\label{eq2+re}
    a_{i,j}\in C^\infty(\mathbb R^{N+1})
    \end{eqnarray}
    for all $i,j\in\{1,...,m\}$. In \cite{LN} the assumptions in \eqref{compre} and \eqref{eq2+re} are only used in a qualitative fashion.  In particular, from the perspective of the operator the constants of the quantitative estimates in \cite{LN} only depend on $m$ and $\kappa$. To be consistent with \cite{LN}, we in \eqref{keyassump} have included the qualitative assumptions stated in \eqref{compre}, \eqref{eq2+re}.

 \subsection{Function spaces} Let $U_X\subset\mathbb R^{m}$, $U_Y\subset\mathbb R^{m}$ be bounded domains, i.e., bounded, open and  connected sets in $\mathbb R^m$. Let $J\subset\mathbb R$  be an open and bounded interval. We denote by ${H}_X^1(U_X)$ the Sobolev
space of functions $g\in L_{}^2(U_X)$  whose distribution gradient in $U_X$ lies in $(L^2(U_X))^m$, i.e.
  \begin{eqnarray*}\label{fspace-}
{H}_X^1(U_X):=\{g\in L_{X}^2(U_X)\mid  \nabla_Xg\in (L^2(U_X))^m\},
    \end{eqnarray*}
    and we set
    $$||g||_{{H}_X^1(U_X)}:=||g||_{L^2(U_X)}+|||\nabla_Xg|||_{L^2(U_X)},\ g\in {H}_X^1(U_X).$$
    We let ${H}_{X,0}^1(U_X)$ denote the closure of $C_0^\infty(U_X)$ in the norm of ${H}_X^1(U_X)$. If $U_X$ is a bounded Lipschitz domain, then
    $C^\infty(\overline{U_X})$ is dense in ${H}_X^1(U_X)$. In particular, equivalently we could define ${H}_X^1(U_X)$ as the closure of $C^\infty(\overline{U_X})$
     in the norm $||\cdot||_{{H}_X^1(U_X)}$. We let ${H}_X^{-1}(U_X)$ denote the dual to ${H}_X^1(U_X)$, whose elements act on functions in ${H}_{X,0}^1(U_X)$ through the
     duality pairing $\langle \cdot,\cdot\rangle:=\langle \cdot,\cdot\rangle_{H_X^{-1}(U_X),H_{X,0}^{1}(U_X)}$.

     In analogy with the definition of ${H}_X^1(U_X)$, we let $W(U_X\times U_{Y}\times J)$ be the closure of $C^\infty(\overline{U_X\times U_{Y}\times J)}$ in the norm
     \begin{align}\label{weak1-+}
     ||u||_{W(U_X\times U_{Y}\times J)}&:=||u||_{L_{Y,t}^2(U_{Y}\times J,H^1(U_X))}+||(-X\cdot\nabla_Y+\partial_t)u||_{L_{Y,t}^2(U_{Y}\times J,{H}_X^{-1}(U_X))}\notag\\
     &:=\biggl (\iint_{U_{Y}\times J}||u(\cdot,Y,t)||_{{H}_X^1(U_X)}^2\, \d Y\d t\biggl )^{1/2}\\
     &+\biggl (\iint_{U_{Y}\times J}||(-X\cdot\nabla_Y+\partial_t)u(\cdot,Y,t)||_{{H}_X^{-1}(U_X)}^2\, \d Y\d t\biggl )^{1/2}.\notag
    \end{align}
    In particular, $W(U_X\times U_{Y}\times J)$ is a Banach space and $u\in W(U_X\times U_{Y}\times J)$ if and only if \begin{eqnarray}\label{weak1-}
u\in L_{Y,t}^2(U_{Y}\times J,H_X^1(U_X))\mbox{ and } (-X\cdot\nabla_Y+\partial_t)u\in  L_{Y,t}^2(U_{Y}\times J,H_X^{-1}(U_X)).
    \end{eqnarray}

    Let $\Omega\subset\mathbb R^m$, $m\geq 2$, be a (unbounded) Lipschitz domain as defined in \eqref{Lip} and with constant $M$. We say that
    $u\in W_{\mbox{loc}}(\Omega\times \mathbb R^{m}\times \mathbb R)$ if $u\in W(U_X\times U_{Y}\times J)$ whenever $U_X\subset\mathbb R^{m}$, $U_Y\subset\mathbb R^{m}$ are bounded domains, $J\subset\mathbb R$  is an open and bounded interval, and $\overline{U_X\times U_{Y}\times J}$ is compactly contained in $\Omega\times \mathbb R^{m}\times \mathbb R$.

\subsection{Weak solutions} Let $U_X$, $U_{Y}$ and $J$ be as introduced in the previous subsection. We say that $u$ is a weak solution to
    \begin{equation}\label{dpweak+}
	\L_{\K} u = 0  \text{ in} \ U_X\times U_{Y}\times J,
\end{equation}
    if $u\in W(U_X\times U_{Y}\times J)$ and if
\begin{equation}\label{weak3}
    \begin{split}
     0 =&\iiint_{U_X\times U_{Y}\times J}\ A(X)\nabla_Xu\cdot \nabla_X\phi\, \d X \d Y \d t\\
    &+\iint_{U_{Y}\times J}\ \langle (-X\cdot\nabla_Y+\partial_t)u(\cdot,Y,t),\phi(\cdot,Y,t)\rangle\, \d Y \d t,
\end{split}
\end{equation}
for all $ \phi\in L_{Y,t}^2(U_{Y}\times J,H_{X,0}^1(U_X))$. Here, again, $\langle \cdot,\cdot\rangle=\langle \cdot,\cdot\rangle_{H_X^{-1}(U_X),H_{X,0}^{1}(U_X)}$ is the duality pairing between $H_X^{-1}(U_X)$ and $H_{X,0}^1(U_X)$.

\begin{definition}\label{weaklip} Let $\Omega\subset\mathbb R^m$, $m\geq 2$, be a (unbounded) Lipschitz domain as defined in \eqref{Lip} and with constant $M$. We say that $u$ is a weak solution to
    \begin{equation}\label{dpweak+g}
	\L_{\K} u = 0  \text{ in} \ \Omega\times \mathbb R^{m}\times \mathbb R,
\end{equation}
if $u\in W_{\mbox{loc}}(\Omega\times \mathbb R^{m}\times \mathbb R)$ and if $u$  satisfies \eqref{weak3}, whenever $\overline{U_X\times U_{Y}\times J}$ is compactly contained in $\Omega\times \mathbb R^{m}\times \mathbb R$.
\end{definition}

 Note that if $u$ is a weak solution to the equation $\L_\K u=0$  in $\Omega\times \mathbb R^{m}\times \mathbb R$, then it is a weak solution in the sense of distributions, i.e.
                              \begin{eqnarray}\label{weak4}
                              \iiint_{}\ \bigl(A(X)\nabla_Xu\cdot \nabla_X\phi-u(-X\cdot \nabla_Y+\partial_t)\phi\bigr )\, \d X \d Y \d t=0,
                               \end{eqnarray}
                               whenever $\phi\in C_0^\infty(\Omega\times \mathbb R^{m}\times \mathbb R)$.

\subsection{The Dirichlet problem and associated boundary measures}  In \cite{LN1} we have conducted a  study of the existence and uniqueness of weak solutions
 to $$\nabla_X\cdot(A(X,Y,t)\nabla_X u)+X\cdot\nabla_Yu-\partial_tu=0,$$ as well as the existence and uniqueness of weak solutions to the Dirichlet problem with continuous boundary data. On  \cite{LN1}, Theorem 1.2, Theorem 1.3, and Theorem 1.4, are particularly relevant to this paper.  Theorem 1.2 in \cite{LN1} concerns the existence of weak solutions to \eqref{e-bvpuu}. However, in \cite{LN1} a stronger notion of weak solutions
is used, see Definition 2 in \cite{LN1}, as we there demand certain Sobolev regularity up to the boundary of $\Omega\times \mathbb R^m\times \mathbb R$. Theorem 1.3 in \cite{LN1} concerns the uniqueness of weak solutions to \eqref{e-bvpuu} and in Theorem 1.4 in \cite{LN1} we consider the continuous Dirichlet problem and the representation of the solution using associated parabolic measures. We here state the following consequence of these results.

\begin{theorem}\label{EUDP} Assume that $A$ satisfies \eqref{eq:Aellip} and \eqref{compre}. Let $f\in
C_0(\partial\Omega\times \mathbb R^m\times \mathbb R)$. Then there exists $u\in C(\bar \Omega\times \mathbb R^m\times \mathbb R )$ such that $u=u_f$ is a  weak solution to the Dirichlet problem
\begin{equation} \label{e-bvpuu}
\begin{cases}
	\L_{\K} u = 0  &\text{in} \ \Omega\times \mathbb R^m\times \mathbb R , \\
      u = f  & \text{on} \ \partial \Omega\times \mathbb R^m\times \mathbb R,
\end{cases}
\end{equation}
in the sense of Definition \ref{weaklip}. If $u$ is bounded, then $u=u_f$ is the unique weak solution to \eqref{e-bvpuu} and in this case there exists, for every $(X,Y,t)\in \Omega\times \mathbb R^m\times \mathbb R $, a unique probability
measure  $\omega_{\K}(X,Y,t,\cdot)$ on $\partial\Omega\times \mathbb R^m\times \mathbb R $ such that
\begin{eqnarray}  \label{1.1xxuu}
u(X,Y,t)=\iiint_{\partial\Omega\times \mathbb R^m\times \mathbb R }f(\tilde X,\tilde Y,\tilde t)\d \omega_{\K}(X,Y,t,\tilde X,\tilde Y,\tilde t).
\end{eqnarray}
\end{theorem}
\begin{proof} As stated above, the notion of weak solutions introduced in  Definition \ref{weaklip} is weaker than the notion of weak solutions introduced in Definition 2 in \cite{LN1}. In particular, concerning the existence part of Theorem \ref{EUDP}, Theorems 1.2-1.4 in \cite{LN1} give a stronger result. Concerning
uniqueness,  Theorem 1.3 and Theorem 1.4 in \cite{LN1}, an important piece of information is
neglected in the statements of these two theorems. As can be seen from the proofs of Theorem 1.3 and Theorem 1.4 in \cite{LN1}, this information concerns the fact that in the unbounded setting  $\Omega\times \mathbb R^m\times \mathbb R$ we need a condition at infinity to ensure uniqueness and what we prove is the uniqueness of bounded weak solutions. In particular, in Theorem 1.3 in \cite{LN1} it  should be stated that
$g\in W(\mathbb R^{N+1})\cap L^\infty(\mathbb R^{N+1})$ and that $u$ is unique if $u\in L^\infty(\Omega\times \mathbb R^m\times \mathbb R)$. Similarly, in Theorem 1.4 in \cite{LN1} it  should be stated that $u$ is unique if $u\in L^\infty(\Omega\times \mathbb R^m\times \mathbb R)$. In Theorem 1.3 and Theorem 1.4 in \cite{LN1} we simply missed stating the obvious restriction $u\in L^\infty(\Omega\times \mathbb R^m\times \mathbb R)$ under which the proofs in \cite{LN1} are given. With this clarification,
Theorem \ref{EUDP} is a special case of Theorem 1.4 in \cite{LN1}.
\end{proof}

The measure $\omega_{\K}(X,Y,t,E)$ introduced in Theorem \ref{EUDP} is referred to as the parabolic measure, or Kolmogorov measure to distinguish it from the parabolic measure associated to $\L_\mathcal{P}$, associated to $\L_{\K}$ in $\Omega\times \mathbb R^m\times \mathbb R$, at $(X,Y, t)\in \Omega\times \mathbb R^m\times \mathbb R $ and of $E\subset\partial\Omega\times \mathbb R^m\times \mathbb R $. Properties of $\omega_{\K}(X,Y,t,\cdot)$ govern the Dirichlet problem in \eqref{e-bvpuu}. The corresponding elliptic and parabolic measures on $\partial\Omega$ and $\partial\Omega\times \mathbb R$, $\omega_{\E}$ and  $\omega_{\P}$, are introduced analogously.

\subsection{The non-tangential maximal operator}\label{ntmaxop} Given a (unbounded)  Lipschitz domain $\Omega \subset\mathbb R^{m}$ with constant $M$, $$(X_0,Y_0,t_0)=((x_0,\psi(x_0)),Y_0,t_0)\in\partial\Omega \times \R^m \times \R,$$ and $\eta>0$, we introduce the (non-tangential) cone
\begin{equation}\label{nt-cone}
    \Gamma^\eta(X_0,Y_0,t_0)
    := \{(X,Y,t)\in \Omega \times \R^m \times \R \mid  d((X,Y,t),(X_0,Y_0,t_0)) < \eta |x_m-\psi(x_0)|\}.
\end{equation}
Given a function $u$ defined in $\Omega\times\R^m\times \R$  we consider the non-tangential maximal operator
\begin{equation}\label{eq:N*}
N^{\eta}(u)(X_0,Y_0,t_0):= \sup_{(X,Y,t) \in \Gamma^\eta(X_0,Y_0,t_0)} |u(X,Y,t)|.
\end{equation}
If $f$ is defined on $\partial \Omega \times \R^m \times \R$ and $(X_0,Y_0,t_0)\in\partial \Omega \times \R^m \times \R$, then we say that $u(X_0,Y_0,t_0)=f(X_0,Y_0,t_0)$ non-tangentially (n.t.) if
\[
\lim_{\substack{(X,Y,t)\in \Gamma^\eta(X_0,Y_0,t_0)\\ (X,Y,t)\to (X_0,Y_0,t_0)}}u(X,Y,t)=f(X_0,Y_0,t_0),
\]
where $\eta=\eta(M)$ is chosen so that $(\partial \Omega \times \R^m \times \R) \cap\Gamma^\eta(X_0,Y_0,t_0)=\{(X_0,Y_0,t_0)\}$. With this choice of $\eta$ we simply write $N(u)$ for $ N^\eta(u)$. Furthermore, given $\delta>0$ we introduce the truncated cone
\begin{equation}\label{nt-conetrunc}
    \Gamma^\eta_\delta(X_0,Y_0,t_0):=\Gamma^\eta(X_0,Y_0,t_0)\cap\mathcal{B}_\delta(X_0,Y_0,t_0),
\end{equation}
and the truncated non-tangential maximal operator
\begin{equation}\label{eq:N*trunc}
N^{\eta}_\delta(u)(X_0,Y_0,t_0):= \sup_{(X,Y,t) \in \Gamma_\delta^\eta(X_0,Y_0,t_0)} |u(X,Y,t)|.
\end{equation}
Again with $\eta$ fixed, we write $N_\delta(u)$ for $ N^\eta_\delta(u)$. For more on non-tangential maximal functions in the elliptic context we refer to \cite{Ke}.

\subsection{Conventions}
Throughout the paper we will use following conventions. $c$ will, if not otherwise stated, denote a constant satisfying $1\leq c<\infty$. We write that $c_1\lesssim c_2$ if $c_1/c_2$ is bounded from above by a positive constant depending only on $m$, $\kappa$, and $M$, if not otherwise stated. We write $c_1\approx c_2$ if $c_1\lesssim c_2$ and $c_2\lesssim c_1$.

Given a point  $(X,Y,t)\in\mathbb R^m\times \R^m \times \R$ we let $\pi_X(X,Y,t):=X$, $\pi_{X,t}(X,Y,t):=(X,t)$. Similarly, if $\Delta\subset\partial\Omega\times \R^m \times \R$, then we let $\pi_X(\Delta)$ denote the projection of
 $\Delta$ onto the $X$ coordinate, we let $\pi_{X,t}(\Delta)$ denote the projection of
 $\Delta$ onto the $(X,t)$ coordinates.

\section{Statements of the structural theorems}\label{secstruc}

Our structural theorems concern the quantitative relations between the measures $\omega_{\mathcal{E}}$,  $\omega_{\mathcal{P}}$,  $\omega_{\mathcal{K}}$ and the (physical) measures $\sigma_{\E}$, $\sigma_{\P}$, $\sigma_{\K}$. We first prove the following relations between the measures.

\begin{theorem}\label{Ainfty} Assume \eqref{keyassump}. Let $\omega_{\mathcal{E}}$,  $\omega_{\mathcal{P}}$,  and $\omega_{\mathcal{K}}$ be the elliptic, parabolic and Kolmogorov measure associated to
$\L_\mathcal{E}$, $\L_\mathcal{P}$, $\L_\mathcal{K}$ in $\Omega $, $\Omega \times \R$ and $ \Omega \times \R^m \times \R$, respectively. Then there exist
 $\Lambda=\Lambda(m,M)$, $1\leq \Lambda<\infty$ and $c=c(m,\kappa,M)$,  $1\leq c<\infty$ such that the following is true. Consider $\Delta:=\Delta_r(X_0,Y_0,t_0)\subset \partial\Omega\times \mathbb R^m\times \mathbb R$. Then
 \begin{align*}
 \frac {\sigma_\K(\Delta)\omega_{\K}(A_{c\Delta,\Lambda}^+,\tilde \Delta)}{\sigma_\K(\tilde\Delta)}&\approx
 \frac {\sigma_\P(\pi_{X,t}(\Delta))\omega_{\P}(\pi_{X,t}(A_{c\Delta,\Lambda}^+),\pi_{X,t}(\tilde \Delta))}{\sigma_\P(\pi_{X,t}(\tilde \Delta))}\\
 &\approx \frac {\sigma_\E(\pi_{X}(\Delta))\omega_{\E}(\pi_{X}(A_{c\Delta,\Lambda}^+),\pi_{X}(\tilde \Delta))}{\sigma_\E(\pi_{X}(\tilde \Delta))},\end{align*}
whenever $\tilde \Delta\subset \Delta$.
\end{theorem}

Theorem \ref{Ainfty} states that the measures $\omega_{\K}(A_{c\Delta,\Lambda}^+,\cdot)$, $\omega_{\P}(\pi_{X,t}(A_{c\Delta,\Lambda}^+),\cdot)$, $\omega_{\E}(\pi_{X}(A_{c\Delta,\Lambda}^+),\cdot)$ are all comparable in the sense stated when evaluated on the surface cube  $\tilde \Delta\subset \Delta$. As we will prove,  if $\tilde \Delta=\Delta_{\tilde r}$ and if
   \begin{eqnarray}\label{assump}
   \lim_{{\tilde r}\to 0} \frac {\omega_{\E}(\pi_{X}(A_{c\Delta,\Lambda}^+),\pi_{X}( \Delta_{\tilde r}))}{\sigma_\E(\pi_{X}( \Delta_{\tilde r}))}\end{eqnarray}
   exists, then also the limits
     \begin{eqnarray}
 \lim_{{\tilde r}\to 0}\frac {\omega_{\K}(A_{c\Delta,\Lambda}^+, \Delta_{\tilde r})}{\sigma_\K(\Delta_{\tilde r})}\mbox{ and } \lim_{{\tilde r}\to 0}\frac {\omega_{\P}(\pi_{X,t}(A_{c\Delta,\Lambda}^+),\pi_{X,t}( \Delta_{\tilde r}))}{\sigma_\P(\pi_{X,t}( \Delta_{\tilde r}))}
 \end{eqnarray}
  exist and all limits are comparable in the sense of Theorem \ref{Ainfty}. Indeed, using \eqref{assump} we will be able to deduce that the Poisson kernels
    \begin{align*}
   K_\E (\pi_{X}(A_{c\Delta,\Lambda}^+),X)&:=\frac {\d\omega_{\E}}{\d\sigma_\E}(\pi_{X}(A_{c\Delta,\Lambda}^+),X),\notag\\
   K_\P (\pi_{X,t}(A_{c\Delta,\Lambda}^+),X,t)&:=\frac {\d\omega_{\P}}{\d\sigma_\P}(\pi_{X,t}(A_{c\Delta,\Lambda}^+),X,t),\notag\\
   K_\K (A_{c\Delta,\Lambda}^+,X,Y,t)&:=\frac {\d\omega_{\K}}{\d\sigma_\K}(A_{c\Delta,\Lambda}^+,X,Y,t),\end{align*}
   are all well-defined on $\Delta$ and that
       \begin{align*}
       \sigma_\K(\Delta)K_\K (A_{c\Delta,\Lambda}^+,X,Y,t)&\approx \sigma_\P(\pi_{X,t}(\Delta)) K_\P (\pi_{X,t}(A_{c\Delta,\Lambda}^+),X,t)\\
       &\approx \sigma_\E(\pi_{X}(\Delta)) K_\E (\pi_{X}(A_{c\Delta,\Lambda}^+),X),\end{align*}
       whenever $(X,Y,t)\in \Delta$.

   Given $q$, $1<q<\infty$, we say that
   $K_\E (X):=K_\E (\pi_{X}(A_{c\Delta,\Lambda}^+),X)\in B_q(\pi_X(\Delta),\d\sigma_\E)$ with constant $\Gamma$, $1\leq\Gamma<\infty$, if
      \begin{eqnarray}\label{Bq_E}
\biggl (\barint_{\pi_X(\tilde\Delta)}|K_\E (X)|^q\, \d\sigma_\E(X)\biggr )^{1/q}\leq \Gamma\biggl (\barint_{\pi_X(\tilde\Delta)}|K_\E (X)|\, \d\sigma_\E(X)\biggr ),
\end{eqnarray}
 for all $\tilde\Delta\subset\Delta$. Analogously, $K_\P (X,t):=K_\P (\pi_{X,t}(A_{c\Delta,\Lambda}^+),X,t)\in B_q(\pi_{X,t}(\Delta),\d\sigma_\P)$ and
 $K_\K (X,Y,t):=K_\K (A_{c\Delta,\Lambda}^+,X,Y,t)\in B_q(\Delta,\d\sigma_\K)$ with constant $\Gamma$, if
\begin{equation}\label{Bq}
    \begin{split}
        \biggl (\bariint_{\pi_{X,t}(\tilde\Delta)}|K_\P (X,t)|^q\, \d\sigma_\P(X,t)\biggr )^{1/q}&\leq \Gamma\biggl (\bariint_{\pi_{X,t}(\tilde\Delta)}|K_\P (X,t)|\, \d\sigma_\P(X,t)\biggr ),\mbox{ and}\\
        \biggl (\bariiint_{\tilde\Delta}|K_\K (X,Y,t)|^q\, \d\sigma_\K(X,Y,t)\biggr )^{1/q}&\leq \Gamma\biggl (\bariiint_{\tilde\Delta}|K_\K (X,Y,t)|\, \d\sigma_\K(X,Y,t)\biggr ),
    \end{split}
\end{equation}
respectively, for all $\tilde\Delta\subset\Delta$.

We can now state our second main result.

   \begin{theorem}\label{Ainfty+} Assume \eqref{keyassump}. Let $\omega_{\mathcal{E}}$,  $\omega_{\mathcal{P}}$,  and $\omega_{\mathcal{K}}$ be as in the statement of Theorem \ref{Ainfty}.  Then there exist
 $\Lambda=\Lambda(m,M)$, $1\leq \Lambda<\infty$ and $c=c(m,\kappa,M)$,  $1\leq c<\infty$ such that the following is true. Consider $\Delta:=\Delta_r(X_0,Y_0,t_0)\subset \partial\Omega\times \mathbb R^m\times \mathbb R$. Assume that $\omega_{\E}(\pi_X(A_{c\Delta,\Lambda}^+),\cdot)$ is mutually absolutely continuous on
 $\pi_X(\Delta)$  with respect to $\sigma_\E$ and that the associated Poisson kernel $K_\E (X):=K_\E (\pi_{X}(A_{c\Delta,\Lambda}^+),X)$ satisfies
 $$K_\E \in B_q(\pi_X(\Delta),\d\sigma_\E)$$ for some $q$, $1<q<\infty$, and with constant $\Gamma$, $1\leq\Gamma<\infty$.  Then $\omega_{\P}(\pi_{X,t}(A_{c\Delta,\Lambda}^+),\cdot)$ and $\omega_{\K}(A_{c\Delta,\Lambda}^+,\cdot)$ are mutually absolutely continuous on
 $\pi_{X,t}(\Delta)$  and $\Delta$ with respect to $\sigma_\P$ and $\sigma_\K$, respectively, and the associated Poisson kernels $K_\P (X,t):=K_\P (\pi_{X,t}(A_{c\Delta,\Lambda}^+),X,t)$ and
 $K_\K (X,Y,t):=K_\K (A_{c\Delta,\Lambda}^+,X,Y,t)$ satisfy
 $$K_\P\in B_q(\pi_{X,t}(\Delta),\d\sigma_\P),\ K_\K \in B_q(\Delta,\d\sigma_\K),$$
 with constant $\tilde\Gamma$, $\tilde\Gamma=\tilde\Gamma(m,\kappa,M,\Gamma)$.
\end{theorem}

We also prove the following theorem.
\begin{theorem}\label{DPequiv} Assume \eqref{keyassump}. Let  $p\in (1,\infty)$ be given and let $q$ denote the index dual to $p$. Assume that $\omega_{\K}(A_{c\Delta,\Lambda}^+,\cdot)$ is mutually absolutely continuous on $\Delta$ with respect to  $\sigma_\K$, for all $\Delta:=\Delta_r(X_0,Y_0,t_0)\subset \partial\Omega\times \mathbb R^m\times \mathbb R$. Then the following statements are equivalent.
\begin{align*}
 (i)&\mbox{ $K_\K (A_{c\Delta,\Lambda}^+,\cdot,\cdot,\cdot)\in B_q(\Delta,\d\sigma_\K)$ for all $\Delta\subset \partial\Omega\times \mathbb R^m\times \mathbb R$, with a uniform constant $\Gamma$.}\notag\\
 (ii)&\mbox{ $D_{\K}^p(\partial \Omega\times \mathbb R^m\times \mathbb R,\d\sigma_\K)$ is solvable.}
\end{align*}
Furthermore, if  $D_{\K}^p(\partial \Omega\times \mathbb R^m\times \mathbb R,\d\sigma_\K)$ is solvable then it  is uniquely solvable.
\end{theorem}

\section{Local regularity and boundary estimates}\label{sec1.5}

In this section we state a number of the lemmas concerning the interior regularity of weak solution and the boundary behaviour of non-negative solutions. The boundary estimates  are proven in \cite{LN} for the more general operators stated in \eqref{e-kolm-ndre}, assuming \eqref{eq2re}, \eqref{compre} and \eqref{eq2+re}. Concerning geometry, in \cite{LN} we consider unbounded domains $\tilde\Omega\subset\mathbb R^{N+1}$ of the form
\begin{eqnarray}\label{dom-re}
 \tilde\Omega=\{(X,Y,t)=(x,x_{m},y,y_{m},t)\in\mathbb R^{N+1} \mid \ x_m>\tilde\psi(x,y,y_m,t)\},
    \end{eqnarray}
    imposing restrictions on $\tilde\psi$ of Lipschitz character accounting for the underlying non-Euclidean group structure. In particular, we also allowed for $(Y,t)$-dependent domains. Up to a point, the results in \cite{LN} are established allowing $A=A(X,Y,t)$ and $\tilde\psi=\tilde\psi(x,y,y_m,t)$ to depend on all variables with $y_m$ included. However, the more refined results established are derived assuming in addition that $A$ as well as $\psi$ are independent of  the variable $y_m$. The reason for this is discussed in detail in \cite{LN}. Obviously, the operators $\L_\K$ considered in this paper are,  as  $A=A(X)$, special cases of the more general operators of Kolmogorov type considered in \cite{LN}. Also, the geometric setting of \cite{LN} is more demanding compared to the domains considered in this paper, as $\Omega\times\mathbb R^m\times\mathbb R$ is a special case of the domains in \eqref{dom-re}.

Below we  formulate
 the necessary auxiliary and boundary type estimate results, needed in our proofs and in particular in the proofs of Theorem \ref{Ainfty}, Theorem \ref{Ainfty+} and Theorem \ref{DPequiv}, in the context of $\L_\K$ as these results follow from \cite{LN}. For the corresponding results for $\L_\E$ and $\L_\P$ we refer to \cite{Ke} and \cite{FS, FSY, FGS, N}, respectively.

 \subsection{Energy estimates and local regularity}

 Consider $(X_0,Y_0,t_0)\subset \mathbb R^{N+1}$. In the following we will frequently use the notation $Q_{\rho}:=Q_{\rho}(X_0,Y_0,t_0)$ for $\rho>0$.

 \begin{lemma}\label{lem1en} Assume that $u$ is a weak solution to $\L_\K u=0$ in $Q_{2r}=Q_{2r}(X_0,Y_0,t_0)\subset \mathbb R^{N+1}$.  Then
\begin{eqnarray*}
 \iiint_{Q_{r}}\ |\nabla_Xu|^2\, \d X\d Y \d t\lesssim \frac{1}{r^2} \iiint_{Q_{2r}}\ |u|^2\, \d X\d Y \d t.
\end{eqnarray*}
\end{lemma}
\begin{proof}
This is an energy estimate that can be proven using standard arguments. We refer to \cite{LN} for further details.
\end{proof}

The following two lemmas are proved in \cite{Ietal}.
 \begin{lemma}\label{lem1} Assume that $u$ is a weak solution to $\L_\K u=0$ in $Q_{2r}(X_0,Y_0,t_0)\subset \mathbb R^{N+1}$.  Given $p\in [1,\infty)$ there exists a constant $c=c(m,\kappa,p)$,
$1\leq c<\infty$, such that
\begin{eqnarray}
\sup_{Q_r}\ |u|\leq c\biggl (\bariiint_{Q_{2r}}\ |u|^p\, \d X\d Y \d t\biggr )^{1/p}.
\end{eqnarray}
\end{lemma}

 \begin{lemma}\label{lem1+} Assume that $u$ is a weak solution to $\L_\K u=0$ in $Q_{2r}(X_0,Y_0,t_0)\subset \mathbb R^{N+1}$.  Then there exists $\alpha
=\alpha (m,\kappa)\in (0,1)$, such that
\begin{equation}
|u(X,Y,t)-u(\tilde X,\tilde Y,\tilde t)|\lesssim\biggl (\frac{d((X,Y,t),(\tilde X,\tilde Y,\tilde t))}{r}\biggr )^{\alpha
}\sup_{Q_{2r} }|u|
\end{equation}%
whenever $(X,Y,t), (\tilde X,\tilde Y,\tilde t)\in  Q_{r}(X_0,Y_0,t_0)$.
\end{lemma}

To state the Harnack inequality we introduce
\begin{eqnarray}\label{pastcyl}
 Q_r^-(X_0,Y_0,t_0):=Q_r(X_0,Y_0,t_0)\cap\{(X,Y,t)\mid t_0-r^2<t<t_0\}.
\end{eqnarray}
The following Harnack inequality is proved in \cite{Ietal}.
\begin{lemma}\label{harnack} {There exist constants $c=c(m,\kappa)>1$ and $\alpha, \beta, \gamma, \theta \in (0,1)$, with
$0 < \alpha < \beta < \gamma < \theta^2$, such that the following is true. Assume that $u$ is a non-negative weak solution to
$\L_\K u=0$ in $Q_r^-(X_0,Y_0,t_0)\subset\mathbb R^{N+1}$. Then,
$$
    \sup_{\widetilde Q^-_{r}(X_0,Y_0,t_0)} u \leq c\inf_{\widetilde Q^+_{r}(X_0,Y_0,t_0)} u,
$$
}
where
\begin{equation*}
\begin{split}
       \widetilde Q^+_{r}(X_0,Y_0,t_0) & = \big\{ (X,Y,t) \in Q_{\theta r}^-(X_0,Y_0,t_0) \mid t_0 - \alpha r^2 \le t \le t_0
\big\},\\
   \widetilde Q^-_{r}(X_0,Y_0,t_0) & = \big\{ (X,Y,t) \in Q_{\theta r}^-(X_0,Y_0,t_0) \mid t_0 - \gamma r^2 \le t \le t_0 -\beta
r^2  \big\}.
\end{split}
\end{equation*}
\end{lemma}

\begin{remark}
Note that the constants $\alpha, \beta, \gamma, \theta$ appearing in Lemma \ref{harnack} can not be chosen arbitrarily.
\end{remark}

 \subsection{Estimates for (non-negative) solutions}

 We refer to \cite{LN} for the proofs of the following results.

\begin{lemma}\label{lem4.5-Kyoto1}
Assume \eqref{keyassump}. Let $(X_0,Y_0,t_0)\in\partial\Omega\times\mathbb R^m\times\mathbb R$ and $r>0$. Let $u$ be a weak  solution of $\L_\K u=0$ in $(\Omega\times\mathbb R^m\times\mathbb R)\cap Q_{2r}(X_0,Y_0,t_0) $, vanishing continuously on  $(\partial\Omega\times\mathbb R^m\times\mathbb R)\cap Q_{2r}(X_0,Y_0,t_0) $. Then,
there exists $\alpha
=\alpha (m,\kappa,M)\in (0,1)$, such that
\begin{equation}
u(X,Y,t)\lesssim\biggl (\frac{d((X,Y,t),(X_0,Y_0,t_0))}{r}\biggr )^{\alpha
}\sup_{(\Omega\times\mathbb R^m\times\mathbb R)\cap Q_{2r}(X_0,Y_0,t_0) }u
\end{equation}%
whenever $(X,Y,t)\in (\Omega\times\mathbb R^m\times\mathbb R)\cap Q_{r/c}(X_0,Y_0,t_0)$.
\end{lemma}

\begin{lemma}\label{lem4.7}  Let $\Omega$ and $A$ be as in Lemma \ref{lem4.5-Kyoto1}.  There
exist $\Lambda=\Lambda(m,M)$, $c=c(m,\kappa,M)$, and  $\gamma=\gamma(m,\kappa,M)$, $0<\gamma<\infty$, such that {the} following
holds.  Let $(X_0,Y_0,t_0)\in\partial\Omega\times\mathbb R^m\times\mathbb R$ and $r>0$. Assume that $u$ is a non-negative weak solution to $\L_\K u=0$ in
$(\Omega\times\mathbb R^m\times\mathbb R)\cap Q_{2r}(X_0,Y_0,t_0)$. Then
\begin{equation}
\begin{split}
u(X,Y,t)&\lesssim(\varrho/d)^\gamma u(A_{\varrho,\Lambda}^+(X_0,Y_0,t_0)),\\
u(X,Y,t)&\gtrsim (d/\varrho)^\gamma u(A_{\varrho,\Lambda}^-(X_0,Y_0,t_0)),
\end{split}
\end{equation}
whenever $(X,Y,t)\in (\Omega\times\mathbb R^m\times\mathbb R)\cap Q_{2\varrho/c}(X_0,Y_0,t_0)$,  $0<\varrho<r/c$, and where $d:=d((X,Y,t),\partial\Omega\times\mathbb R^m\times\mathbb R)$.
\end{lemma}

\begin{theorem}\label{thm:carleson} Let $\Omega$ and $A$ be as in Lemma \ref{lem4.5-Kyoto1}. Then there exist $\Lambda=\Lambda(m,M)$ and $c=c(m,\kappa,M)$ such that the following holds. Let $(X_0,Y_0,t_0)\in\partial \Omega\times\mathbb R^m\times\mathbb R$ and $r>0$.  Assume that  $u$ is a non-negative weak solution to $\L_\K u=0$ in
$(\Omega\times\mathbb R^m\times\mathbb R)\cap Q_{2r}(X_0,Y_0,t_0)$,
vanishing continuously on $(\partial\Omega\times\mathbb R^m\times\mathbb R)\cap Q_{2r}(X_0,Y_0,t_0)$. Then
\begin{equation*}
u(X,Y,t)\lesssim u(A^+_{\varrho,\Lambda}(X_0,Y_0,t_0))
\end{equation*}
whenever $(X,Y,t)\in (\Omega\times\mathbb R^m\times\mathbb R)\cap Q_{2\varrho/c}(X_0,Y_0,t_0)$, $0<\varrho<r/c$.
\end{theorem}

\begin{theorem}\label{thm:quotients}  Let $\Omega$ and $A$ be as in Lemma \ref{lem4.5-Kyoto1}.
 Then there exist $\Lambda=\Lambda(m,M)$ and $c=c(m,\kappa,M)$ such that the following holds. Let $(X_0,Y_0,t_0)\in\partial\Omega\times\mathbb R^m\times\mathbb R$ and
 $r>0$.  Assume that $u$ and $v$ are non-negative weak solutions to
 $\L_\K u=0$ in $\Omega\times\mathbb R^m\times\mathbb R$,  vanishing continuously on $(\partial\Omega\times\mathbb R^m\times\mathbb R))\cap Q_{2r}(X_0,Y_0,t_0)$. Let $\varrho_0=r/c$,
 \begin{equation}\label{singa1u1}
 \begin{split}
    m_1^+&=v(A_{\varrho_0,\Lambda}^+(X_0,Y_0,t_0)),\ m_1^-=v(A_{\varrho_0,\Lambda}^-(X_0,Y_0,t_0)),\\
    m_2^+&=u(A_{\varrho_0,\Lambda}^+(X_0,Y_0,t_0)),\ m_2^-=u(A_{\varrho_0,\Lambda}^-(X_0,Y_0,t_0)),
 \end{split}
 \end{equation}
and assume $m_1^->0$, $m_2^->0$. Then there exist constants
$c_1=c_1(m,M)$ and $$c_2=c_2(m,\kappa,M, m_1^+/m_1^-,m_2^+/m_2^-),$$
$1\leq
c_1,c_2<\infty$,  such that if we let $\varrho_1=\varrho_0/c_1$, then
$$c_2^{-1}\frac{v(A_{\varrho,\Lambda}(\tilde X_0,\tilde Y_0,\tilde t_0))}{u(A_{\varrho,\Lambda}(\tilde X_0,\tilde Y_0,\tilde t_0))}\leq \frac { v ( X,Y,t ) }{ u ( X,Y,t ) }\leq c_2\frac{v(A_{\varrho,\Lambda}(\tilde X_0,\tilde Y_0,\tilde t_0))}{u(A_{\varrho,\Lambda}(\tilde X_0,\tilde Y_0,\tilde t_0))},$$
whenever $(X,Y,t)\in (\Omega\times\mathbb R^m\times\mathbb R)\cap Q_{\varrho/c_1}(\tilde X_0,\tilde Y_0,\tilde t_0)$, for some  $0<\varrho<\varrho_1$ and $(\tilde X_0,\tilde Y_0,\tilde t_0)\in
(\partial\Omega\times\mathbb R^m\times\mathbb R)\cap Q_{\varrho_1}(X_0,Y_0,t_0)$.
\end{theorem}

\subsection{Estimates of Green functions and parabolic measures}  The adjoint operator of  $\L_\K$ is defined as
\begin{equation}\label{PDEagg}
\L_\K^\ast:=\nabla_X\cdot(A(X)\nabla_X)-X\cdot\nabla_Y+\partial_t,
\end{equation}
as $A$ is assumed to be symmetric.

\begin{remark}
We remark that for non-negative weak solutions to the adjoint equation $\L_\K^\ast u = 0$, adjoint versions of Lemma \ref{lem4.7}, Theorem \ref{thm:carleson}, and Theorem \ref{thm:quotients} hold. The statements in the adjoint versions are the same, except that the roles of $A^+_{\varrho,\Lambda}(X_0,Y_0,t_0)$ and $A^-_{\varrho,\Lambda}(X_0,Y_0,t_0)$ are reversed.
\end{remark}

\begin{definition}\label{fund}
A fundamental solution for $\L_\K$ is a continuous and positive function $\Gamma_\K=\Gamma_\K(X,Y,t,\tilde X,\tilde Y,\tilde t)$, defined for
$\tilde t<t$ and $(X,Y), (\tilde X,\tilde Y)\in\R^{N}$, such that
\begin{itemize}
  \item[(i)] {$\Gamma_\K( \cdot,\cdot,\cdot, \tilde X,\tilde Y,\tilde t)$} is a weak solution of $\L_\K u=0$ in $\mathbb R^N\times (\tilde t,\infty)$ and
  $\Gamma_\K(X,Y,t,\cdot,\cdot,\cdot)$ is a weak solution of $\L_\K^{*} u=0$ in $\mathbb R^N\times (-\infty,t)$,
  \item[(ii)] for any bounded function $\phi\in C(\R^{N})$ and $(X,Y), (\tilde X,\tilde Y)\in\R^{N}$, we have
\begin{align}
  \lim_{(X,Y,t)\to(\tilde X,\tilde Y,\tilde t)\atop t>\tilde t}u(X,Y,t)=\phi(\tilde X,\tilde Y), %\qquad y\in\R^{d},
  \qquad \lim_{(\tilde X,\tilde Y,\tilde t)\to(X,Y,t)\atop t>\tilde t}v(\tilde X,\tilde Y,\tilde t)=\phi(X,Y), %,\qquad x\in\R^{d},
\end{align}
where
  \begin{equation}\label{ae11}
  \begin{split}
 u(X,Y,t)&:=\iint_{\R^{N}}\Gamma_\K(X,Y,t,\tilde X,\tilde Y,\tilde t)\phi(\tilde X,\tilde Y)\, \d\tilde X\d\tilde Y,\\
 v(\tilde X,\tilde Y,\tilde t)&:=\iint_{\R^{N}}\Gamma_\K(X,Y,t,\tilde X,\tilde Y,\tilde t)\phi(X,Y)\, \d X\d Y.%,\qquad T>t,\ y\in\R^{d}.
 \end{split}
\end{equation}
\end{itemize}
\end{definition}

\begin{lemma}\label{lem_fsolbounds}  Assume that $A$ satisfies \eqref{eq2+re}. Then there exists a fundamental solution to $\L_\K$ in the sense of Definition \ref{fund}. Let $\Gamma_\K(X,Y,t,\tilde X,\tilde Y,\tilde t)$ be the fundamental solution to $\L_\K$. Then we have the upper bound
\begin{equation}\label{fundsolbd}
    \Gamma_\K(X,Y,t,\tilde X,\tilde Y,\tilde t) \lesssim \frac{1}{d((X,Y,t),(\tilde X,\tilde Y,\tilde t))^{\mathbf{q}-2}},
\end{equation}
for all $(X,Y,t)$, $(\tilde X,\tilde Y,\tilde{t})$ with $t>\tilde t$.
\end{lemma}
\begin{proof} We refer to \cite{DeM,DiFP,P1} for the existence of the  fundamental solution for $\L$ under the additional condition that the coefficients are
H\"older continuous. See also \cite{LPP}. For the quantitative estimate we refer to Lemma 4.17 in \cite{LN} and the subsequent discussion.
\end{proof}

Assume that $\Omega\subset\mathbb R^{m}$ is a (unbounded)  Lipschitz domain  with constant $M$. We define the Green function associated to $\L_\K$ for $\Omega\times \mathbb R^m\times \mathbb R$, with pole at $(\hat X,\hat Y,\hat t)\in \Omega\times \mathbb R^m\times \mathbb R$, as
\begin{equation}\label{ghh1-}
\begin{split}
G_{\K}(X,Y,t,\hat X,\hat Y,\hat t)=\ &\Gamma_{\K}(X,Y,t, \hat X,\hat Y,\hat t)\\
&-\iiint_{\partial\Omega\times\R^m\times\R}
\Gamma_{\K}(\tilde X,\tilde Y,\tilde t, \hat X,\hat Y,\hat t)\d\omega_{\K}(X,Y,t,\tilde X,\tilde Y,\tilde t),
\end{split}
\end{equation}
where $\Gamma_{\K}$ is the fundamental solution to the operator $\L_{\K}$. If we instead consider $(X,Y,t)\in \Omega\times \mathbb R^m\times \mathbb R$ as fixed, then, for $(\hat X,\hat Y,\hat t)\in
\Omega\times \mathbb R^m\times \mathbb R$,
\begin{equation}\label{ghh1---}
\begin{split}
G_{\K} (X,Y,t,\hat X,\hat Y,\hat t)=\ &\Gamma_{\K}(X,Y,t, \hat X,\hat Y,\hat t)\\
&-\iiint_{\partial\Omega\times\R^m\times\R}
\Gamma_{\K}(X,Y,t, \tilde X,\tilde Y,\tilde t)\d\omega_{\K}^\ast(\hat X,\hat Y,\hat t,\tilde X,\tilde Y,\tilde t),
\end{split}
\end{equation}
where
$\omega_{\K}^\ast(\hat X,\hat Y,\hat t,\cdot)$ is the associated adjoint Kolmogorov measure
relative to $(\hat X,\hat Y,\hat t)$ and $\Omega\times \mathbb R^m\times \mathbb R$.  The corresponding Green functions associated to $\L_\E$  and $\L_\P$, for $\Omega$ and $\Omega\times  \mathbb R$, are denoted $G_{\E}$ and $G_{\P}$, respectively.

Let $\theta\in C_0^\infty(\mathbb R^{N+1})$. The following representation formulas are proved in Lemma 8.3 in \cite{LN},
 \begin{equation}\label{ghh10-}
 \begin{split}
  \theta( \hat X,\hat Y,\hat t)&=\iiint_{\partial\Omega\times\R^m\times\R}\theta(   X,   Y,   t)\d \omega_\K( \hat X,\hat Y,\hat t,   X,   Y,   t)\\
  &-\iiint_{\Omega\times\R^m\times\R} A(   X)\nabla_{  X}G_\K  ( \hat X,\hat Y,\hat t,  X,  Y,  t )\cdot\nabla_{  X}\theta(  X,  Y,   t)
 \d  X\d   Y \d   t\\
 &+\iiint_{\Omega\times\R^m\times\R} G_\K  ( \hat X,\hat Y,\hat t,   X,   Y,  t )( X\cdot\nabla_{  Y}-\partial_{ t})\theta(  X,  Y,  t)
 \d  X\d  Y \d  t,\\
  \theta( \hat X,\hat Y,\hat t)&=\iiint_{\partial\Omega\times\R^m\times\R}\theta(   X,   Y,  t)\d \omega_\K^\ast( \hat X,\hat Y,\hat t,  X,  Y,  t)\\
  &-\iiint_{\Omega\times\R^m\times\R}  A(  X)\nabla_{  X} G_\K (  X,  Y,  t, \hat X,\hat Y,\hat t )\cdot \nabla_{  X}\theta(  X,  Y,  t) \d  X\d  Y \d  t\\
  &+\iiint_{\Omega\times\R^m\times\R} G_\K (  X,   Y,   t, \hat X,\hat Y,\hat t)(-  X\cdot\nabla_{  Y}+\partial_{ t})\theta(  X,  Y,  t) \d  X\d  Y \d  t,
\end{split}
\end{equation}
whenever $( \hat X,\hat Y,\hat t)\in\Omega\times\R^m\times\R$.

The following lemmas, Lemma \ref{greenmeasurerelation} and Lemma \ref{greenestimate}, are proved in \cite{LN}, see in particular section 8 in \cite{LN}. Theorem \ref{thm:doub} stated below is one of the main results in \cite{LN}.

\begin{lemma}\label{greenmeasurerelation}
 Let $\Omega$ and $A$ be as in Lemma \ref{lem4.5-Kyoto1}.  Then there exist
 $\Lambda=\Lambda(m,M)$, $1\leq \Lambda<\infty$,   $c=c(m,\kappa,M)$,  $1\leq c<\infty$, such that the following is true. Let $(X_0,Y_0,t_0)\in\partial\Omega\times\mathbb R^m\times\mathbb R$, $0<\varrho<\infty$. Then
\begin{align*}
    \varrho^{{\bf q}-2}G_\K (X,Y,t,A^+_{\varrho,\Lambda}(X_0,Y_0,t_0)) &\lesssim \omega_{\mathcal{K}}(X,Y,t,\Delta_{\varrho}(X_0,Y_0,t_0))\notag\\
    &\lesssim\varrho^{{\bf q}-2}G_\K (X,Y,t,A^-_{c\varrho,\Lambda}(X_0,Y_0,t_0)),
\end{align*}
whenever $(X,Y,t)\in\Omega\times\R^m\times \R$, $t\geq t_0+c\varrho^2$.
\end{lemma}
\begin{lemma}\label{greenestimate}
 Let $\Omega$ and $A$ be as in Lemma \ref{lem4.5-Kyoto1}.  Then there exist
 $\Lambda=\Lambda(m,M)$, $1\leq \Lambda<\infty$,   $c=c(m,\kappa,M)$,  $1\leq c<\infty$, such that the following is true. Let $(X_0,Y_0,t_0)\in\partial\Omega\times\mathbb R^m\times\mathbb R$, $0<\varrho<\infty$. Then
\begin{align*}
G_\K (X,Y,t,A^-_{\varrho,\Lambda}(X_0,Y_0,t_0))\lesssim G_\K (X,Y,t,A^+_{\varrho,\Lambda}(X_0,Y_0,t_0))\lesssim G_\K (X,Y,t,A^-_{\varrho,\Lambda}(X_0,Y_0,t_0)),
\end{align*}
whenever $(X,Y,t)\in\Omega\times\R^m\times \R$, $t\geq t_0+c\varrho^2$.
\end{lemma}

\begin{theorem}\label{thm:doub}  Let $\Omega$ and $A$ be as in Lemma \ref{lem4.5-Kyoto1}.  Then there exist
 $\Lambda=\Lambda(m,M)$, $1\leq \Lambda<\infty$,   $c=c(m,\kappa,M)$,  $1\leq c<\infty$, such that the following is true. Let $(X_0,Y_0,t_0)\in\partial\Omega\times\mathbb R^m\times\mathbb R$, $0<\varrho_0<\infty$. Then
\begin{eqnarray*}
\omega_\K\bigl (A_{c\varrho_0,\Lambda}^+(X_0,Y_0,t_0), \Delta_{2\varrho}(\tilde X_0,\tilde Y_0,\tilde t_0)\bigr )\lesssim\omega_\K\bigl (A_{c\varrho_0,\Lambda}^+(X_0,Y_0,t_0), \Delta_{\varrho}(\tilde X_0,\tilde Y_0,\tilde t_0)\bigr )
\end{eqnarray*}
for all $\Delta_{\varrho}(\tilde X_0,\tilde Y_0,\tilde t_0)$, $(\tilde X_0,\tilde Y_0,\tilde t_0)\in\partial\Omega\times\mathbb R^m\times\mathbb R$, such that $\Delta_{\varrho}(\tilde X_0,\tilde Y_0,\tilde t_0)\subset \Delta_{4\varrho_0}(X_0,Y_0,t_0)$.
\end{theorem}

\section{Proof of the structural theorems: Theorem \ref{Ainfty} and Theorem \ref{Ainfty+}}\label{sec2}

The purpose of the section is to prove  Theorem \ref{Ainfty} and Theorem \ref{Ainfty+}. Throughout the section we assume \eqref{keyassump}. Let $\omega_{\mathcal{E}}$,  $\omega_{\mathcal{P}}$,  and $\omega_{\mathcal{K}}$ be as in the statement of Theorem \ref{Ainfty}.

\subsection{Proof of Theorem \ref{Ainfty}} To prove Theorem \ref{Ainfty} we need to prove that there exist
 $\Lambda=\Lambda(m,M)$, $1\leq \Lambda<\infty$,   $c=c(m,\kappa,M)$,  $1\leq c<\infty$ such that if $\Delta:=\Delta_r(X_0,Y_0,t_0)\subset \partial\Omega\times \mathbb R^m\times \mathbb R$, then the estimates stated in the theorems hold  whenever $\tilde \Delta\subset \Delta$. The proof  of Theorem \ref{Ainfty} is based on the relation between $\omega_{\E}$, $\omega_{\P}$, $\omega_{\K}$ and the corresponding Green functions and boundary Harnack inequalities.

To start the proof we first note that an immediate consequence of Lemma \ref{greenmeasurerelation} is that there exists $c=c(m,\kappa,M)$, $1\leq c<\infty$, such that given $\Delta:=\Delta_r(X_0,Y_0,t_0)\subset \partial\Omega\times \mathbb R^m\times \mathbb R$, we have
\begin{eqnarray}\label{gensalsa-a}
 \tilde r^{{\bf q}-2} G_{\K}(A_{c\Delta,\Lambda}^+, A_{\tilde\Delta,\Lambda}^+)\lesssim\omega_{\K}(A_{c\Delta,\Lambda}^+, \tilde\Delta)\lesssim\tilde r^{{\bf q}-2}G_{\K}(A_{c\Delta,\Lambda}^+, A_{c\tilde\Delta,\Lambda}^-),
\end{eqnarray}
whenever $\tilde\Delta=\Delta_{\tilde r} \subset \Delta$.
Using this, and the corresponding results for  $\L_\E$ and $\L_\P$, see \cite{Ke} and \cite{FS, FSY, FGS, N}, we obtain
\begin{equation}\label{corr}
\begin{split}
 \frac { G_{\E}(\pi_X(A_{c\Delta,\Lambda}^+), \pi_X(A_{\tilde\Delta,\Lambda}^+))}{G_{\K}(A_{c\Delta,\Lambda}^+, A_{c\tilde\Delta,\Lambda}^-)}&\lesssim\frac {\sigma_\K(\tilde\Delta)\omega_{\E}(\pi_{X}(A_{c\Delta,\Lambda}^+),\pi_{X}(\tilde \Delta))}{\sigma_\E(\pi_{X}(\tilde \Delta))\omega_{\K}(A_{c\Delta,\Lambda}^+,\tilde \Delta)}\\
 &\lesssim\frac {G_{\E}(\pi_X(A_{c\Delta,\Lambda}^+), \pi_X(A_{c\tilde\Delta,\Lambda}^-))}{G_{\K}(A_{c\Delta,\Lambda}^+, A_{\tilde\Delta,\Lambda}^+)},
 \end{split}
\end{equation}
and
\begin{align*}
 \frac { G_{\P}(\pi_{X,t}(A_{c\Delta,\Lambda}^+), \pi_{X,t}(A_{\tilde\Delta,\Lambda}^+))}{G_{\K}(A_{c\Delta,\Lambda}^+, A_{c\tilde\Delta,\Lambda}^-)}&\lesssim\frac {\sigma_\K(\tilde\Delta)\omega_{\P}(\pi_{X,t}(A_{c\Delta,\Lambda}^+),\pi_{X,t}(\tilde \Delta))}{\sigma_\P(\pi_{X,t}(\tilde \Delta))\omega_{\K}(A_{c\Delta,\Lambda}^+,\tilde \Delta)}\\
 &\lesssim\frac {G_{\P}(\pi_{X,t}(A_{c\Delta,\Lambda}^+), \pi_{X,t}(A_{c\tilde\Delta,\Lambda}^-))}{G_{\K}(A_{c\Delta,\Lambda}^+, A_{\tilde\Delta,\Lambda}^+)}.
\end{align*}
To this end we will now prove the theorem  only for  $\omega_\mathcal{K}$, the proof for  $\omega_\mathcal{P}$ being analogous. We first relate $G_{\K}(A_{c\Delta,\Lambda}^+, A_{c\tilde\Delta,\Lambda}^-)$ and $G_{\K}(A_{c\Delta,\Lambda}^+, A_{\tilde\Delta,\Lambda}^+)$. Using that $G_{\K}(A_{c\Delta,\Lambda}^+, \cdot,\cdot,\cdot) $ solves the adjoint equation we can apply the adjoint version of Lemma \ref{lem4.7} to conclude that
\begin{align*}
G_{\K}(A_{c\Delta,\Lambda}^+, A_{\tilde\Delta,\Lambda}^+)&\gtrsim  G_{\K}(A_{c\Delta,\Lambda}^+, A_{c\tilde\Delta,\Lambda}^+),\\ G_{\K}(A_{c\Delta,\Lambda}^+, A_{c\tilde\Delta,\Lambda}^-)&\gtrsim  G_{\K}(A_{c\Delta,\Lambda}^+, A_{\tilde\Delta,\Lambda}^-).
\end{align*}
Hence
\begin{eqnarray}\label{eq:Gquotrels1}
\frac {G_{\K}(A_{c\Delta,\Lambda}^+, A_{c\tilde\Delta,\Lambda}^+)}{G_{\K}(A_{c\Delta,\Lambda}^+, A_{c\tilde\Delta,\Lambda}^-)}\lesssim\frac {G_{\K}(A_{c\Delta,\Lambda}^+, A_{\tilde\Delta,\Lambda}^+)}{G_{\K}(A_{c\Delta,\Lambda}^+, A_{c\tilde\Delta,\Lambda}^-)}\lesssim \frac {G_{\K}(A_{c\Delta,\Lambda}^+, A_{\tilde\Delta,\Lambda}^+)}{G_{\K}(A_{c\Delta,\Lambda}^+, A_{\tilde\Delta,\Lambda}^-)}.
\end{eqnarray}
Therefore, applying Lemma \ref{greenestimate} twice,
\begin{eqnarray}\label{claim1}
\frac {G_{\K}(A_{c\Delta,\Lambda}^+, A_{\tilde\Delta,\Lambda}^+)}{G_{\K}(A_{c\Delta,\Lambda}^+, A_{c\tilde\Delta,\Lambda}^-)}\approx 1.
\end{eqnarray}
Furthermore, by the standard elliptic Harnack inequality
\begin{eqnarray}\label{claim1+}
G_{\E}(\pi_X(A_{c\Delta,\Lambda}^+), \pi_X(A_{\tilde\Delta,\Lambda}^+))\approx G_{\E}(\pi_X(A_{c\Delta,\Lambda}^+), \pi_X(A_{c\tilde\Delta,\Lambda}^-)).
\end{eqnarray}
Putting \eqref{corr}-\eqref{claim1+} together  we can conclude that
\begin{eqnarray}\label{claim1++}
\frac {\sigma_\K(\tilde\Delta)\omega_{\E}(\pi_{X}(A_{c\Delta,\Lambda}^+),\pi_{X}(\tilde \Delta))}{\sigma_\E(\pi_{X}(\tilde \Delta))\omega_{\K}(A_{c\Delta,\Lambda}^+,\tilde \Delta)}\approx \frac {G_{\E}(\pi_X(A_{c\Delta,\Lambda}^+), \pi_X(A_{\tilde\Delta,\Lambda}^+))}{G_{\K}(A_{c\Delta,\Lambda}^+, A_{\tilde\Delta,\Lambda}^+)}.
\end{eqnarray}
Next, using  Theorem \ref{thm:quotients}
\begin{eqnarray*}
\frac {G_{\E}(\pi_X(A_{c\Delta,\Lambda}^+), \pi_X(A_{\tilde\Delta,\Lambda}^+))}{G_{\K}(A_{c\Delta,\Lambda}^+, A_{\tilde\Delta,\Lambda}^+)}\approx\frac {G_{\E}(\pi_X(A_{c\Delta,\Lambda}^+), \pi_X(A_{\Delta,\Lambda}^+))}{G_{\K}(A_{c\Delta,\Lambda}^+, A_{\Delta,\Lambda}^+)}.
\end{eqnarray*}
Furthermore, $G_{\E}(\pi_X(A_{c\Delta,\Lambda}^+), \pi_X(A_{\Delta,\Lambda}^+))\approx r^{-m}\approx (r\sigma_\E(\pi_{X}(\Delta)))^{-1}$ by classical estimates for the fundamental solution second order elliptic equations in divergence form, see \cite{Ke}. We claim that
\begin{eqnarray}\label{clam}
G_{\K}(A_{c\Delta,\Lambda}^+, A_{\Delta,\Lambda}^+)\approx r^{2-{\bf q}}\approx (r \sigma_\K(\Delta))^{-1}.
\end{eqnarray}
To prove this we first note that the upper bound on $G_{\K}(A_{c\Delta,\Lambda}^+, A_{\Delta,\Lambda}^+)$ follows from Lemma \ref{lem_fsolbounds}. The proof of the lower bound on $G_{\K}(A_{c\Delta,\Lambda}^+, A_{\Delta,\Lambda}^+)$ is a bit more subtle but can be achieved analogously to the proof of the estimate in display (9.11) in \cite{LN}. Using \eqref{clam}, we deduce
\begin{eqnarray*}
\frac {G_{\E}(\pi_X(A_{c\Delta,\Lambda}^+), \pi_X(A_{\tilde\Delta,\Lambda}^+))}{G_{\K}(A_{c\Delta,\Lambda}^+, A_{\tilde\Delta,\Lambda}^+)}\approx\frac {G_{\E}(\pi_X(A_{c\Delta,\Lambda}^+), \pi_X(A_{\Delta,\Lambda}^+))}{G_{\K}(A_{c\Delta,\Lambda}^+, A_{\Delta,\Lambda}^+)}\approx \frac {\sigma_\K(\Delta)}{\sigma_\E(\pi_{X}(\Delta))}.
\end{eqnarray*}
Combing this with \eqref{claim1++},
\begin{eqnarray*}
\frac {\sigma_\E(\pi_{X}(\Delta))\sigma_\K(\tilde\Delta)\omega_{\E}(\pi_{X}(A_{c\Delta,\Lambda}^+),\pi_{X}(\tilde \Delta))}{\sigma_\K(\Delta)\sigma_\E(\pi_{X}(\tilde \Delta))\omega_{\K}(A_{c\Delta,\Lambda}^+,\tilde \Delta)}\approx 1.
\end{eqnarray*}
This proves  Theorem \ref{Ainfty}.

\subsection{Proof of  Theorem \ref{Ainfty+}}  Again we will only prove the theorem  for $\omega_\mathcal{K}$, the proof for $\omega_\mathcal{P}$ being analogous. Assume that $\omega_{\E}(\pi_X(A_{c\Delta,\Lambda}^+),\cdot)$ is mutually absolutely continuous on
 $\pi_X(\Delta)$  with respect to $\sigma_\E$ and that the associated Poisson kernel $K_\E (X):=K_\E (\pi_{X}(A_{c\Delta,\Lambda}^+),X)$ satisfies
 $$K_\E \in B_q(\pi_X(\Delta),\d\sigma_\E)$$ for some $q$, $1<q<\infty$, and with constant $\Gamma$, $1\leq\Gamma<\infty$.  To prove Theorem \ref{Ainfty+} for $\omega_\mathcal{K}$ we have to prove that  $\omega_{\K}(A_{c\Delta,\Lambda}^+,\cdot)$ is mutually absolutely continuous on $\Delta$ with respect to  $\sigma_\K$, and that the associated Poisson kernel
 $K_\K (X,Y,t):=K_\K (A_{c\Delta,\Lambda}^+,X,Y,t)$ satisfies $ K_\K \in B_q(\Delta,\d\sigma_\K)$ with a constant $\tilde\Gamma$, $\tilde\Gamma=\tilde\Gamma(m,\kappa,M,\Gamma)$.

 Let $\d\mu_\K:=\omega_{\E}(\pi_X(A_{c\Delta,\Lambda}^+),\pi_X(\cdot))\d Y\d t$. To prove that $\omega_{\K}(A_{c\Delta,\Lambda}^+,\cdot)$ is absolutely continuous on $\Delta$ with respect to  $\sigma_\K$ it suffices to prove that $\omega_{\K}(A_{c\Delta,\Lambda}^+,\cdot)\ll \mu_\K$ on  $\Delta$ and
 that  $\mu_\K\ll \sigma_\K$ on $\Delta$. Recall that $\d\sigma_\K(X,Y,t)=\d\sigma_\E(X)\d Y\d t$. However,
 as $\mu_\K$ and $\sigma_\K$ are defined through the stated product structure it  follows immediately that $\mu_\K\ll\sigma_\K$ on $\Delta$ as
 $\omega_{\E}(\pi_X(A_{c\Delta,\Lambda}^+),\cdot)\ll \sigma_\E$ on $\pi_X(\Delta)$. In particular, by the assumptions it suffices to prove that $\omega_{\K}(A_{c\Delta,\Lambda}^+,\cdot)$ is absolutely continuous on $\Delta$ with respect to  $\mu_\K$ and we will do this by using  Theorem \ref{Ainfty}.

 Recall that we previously observed that $(\Sigma,d,\d\sigma_\K)$, where $\Sigma$ was introduced in \eqref{space},  is a space of homogeneous type in the sense of \cite{CW}. By the results in \cite{Ch} there exists what we here will refer to as a dyadic grid on
 $\Sigma$ having a number of important properties in relation to $d$. To formulate this we introduce, for any $(X,Y,t)\in\Sigma$ and
$E\subset \Sigma$,
\begin{equation}
   {\rm dist} ((X,Y,t),E):=\inf \{ d((X,Y,t),(\tilde X,\tilde Y,\tilde t)) \mid (\tilde X,\tilde Y,\tilde t)\in E\},
\end{equation}
and we let
\begin{equation}
    \diam(E):=\sup \{ d((X,Y,t),(\tilde X,\tilde Y,\tilde t)) \mid (X,Y,t),\ (\tilde X,\tilde Y,\tilde t)\in E\}.
\end{equation}
Using  \cite{Ch} we can conclude that there exist
constants $ \alpha>0,\, \beta>0$ and $c_*<\infty$,  such that for each $k \in \mathbb{Z}$
there exists a collection of Borel sets, $\mathbb{D}_k$,  which we will call cubes, such that
$$
\mathbb{D}_k:=\{Q_{j}^k\subset\Sigma\mid j\in \mathfrak{I}_k\},$$ where
$\mathfrak{I}_k$ denotes some  index set depending on $k$, satisfying
\begin{eqnarray}\label{cubes}
(i)&&\mbox{$\Sigma=\cup_{j}Q_{j}^k\,\,$ for each
$k\in{\mathbb Z}$.}\notag\\
(ii)&&\mbox{If $m\geq k$ then either $Q_{i}^{m}\subset Q_{j}^{k}$ or
$Q_{i}^{m}\cap Q_{j}^{k}=\emptyset$.}\notag\\
(iii)&&\mbox{For each $(j,k)$ and each $m<k$, there is a unique
$i$ such that $Q_{j}^k\subset Q_{i}^m$.}\notag\\
(iv)&&\mbox{$\diam\big(Q_{j}^k\big)\leq c_* 2^{-k}$.}\notag\\
(v)&&\mbox{Each $Q_{j}^k$ contains $\Sigma\cap \mathcal{B}_{\alpha2^{-k}}(X^k_{j}, Y^k_{j},t^k_{j})$ for some $(X^k_{j}, Y^k_{j},t^k_j)\in\Sigma$.}\notag\\
(vi)&&\mbox{$\sigma_\K(\{(X,Y,t)\in Q^k_j:{\rm dist}((X,Y,t),\Sigma\setminus Q^k_j)\leq \varrho \,2^{-k}\big\})\leq
c_*\,\varrho^\beta\,\sigma_\K (Q^k_j),$}\notag\\
&&\mbox{for all $k,j$ and for all $\varrho\in (0,\alpha)$.}
\end{eqnarray}
 We shall denote by  $\mathbb{D}=\mathbb{D}(\Sigma)$ the collection of all
$Q^k_j$, i.e. $$\mathbb{D} := \cup_{k} \mathbb{D}_k.$$
Note that \eqref{cubes} $(iv)$ and $(v)$ imply that for each cube $Q\in\mathbb{D}_k$,
there is a point $(X_Q,Y_Q,t_Q)\in \Sigma$, and  a cube  $Q_{r}(X_Q,Y_Q,t_Q)$ such that
$r\approx 2^{-k} \approx {\rm diam}(Q)$
and \begin{equation}\label{cube-ball}
\Delta_{r}(X_Q,Y_Q,t_Q)\subset Q \subset \Delta_{cr}(X_Q,Y_Q,t_Q),\end{equation}
for some uniform constant $c$. We let
\begin{equation}\label{cube-ball2}
\Delta_Q:= \Delta_{r}(X_Q,Y_Q,t_Q),\end{equation}
and we shall refer to the point $(X_Q,Y_Q,t_Q)$ as the center of $Q$. Given a dyadic cube $Q\subset\Sigma$, we define its $\gamma$ dilate  by
\begin{equation}\label{dilatecube}
\gamma Q:= \Delta_{\gamma \diam(Q)}(X_Q,Y_Q,t_Q).
\end{equation}
For a dyadic cube $Q\in \mathbb{D}_k$, we let $\ell(Q) = 2^{-k}$, and we shall refer to this quantity as the length
of $Q$.  Clearly, $\ell(Q)\approx \diam(Q).$

We now prove that $\omega_{\K}(A_{c\Delta,\Lambda}^+,\cdot)$ is absolutely continuous on $\Delta$ with respect to  $\mu_\K$ using  Theorem \ref{Ainfty}. Indeed, let $E\subset \Delta$ and $\delta>0$, and let $\{Q_j\}$ be a (finite) dyadic Vitali covering of
$E$ such that
$$\mu_\K(\cup Q_j)<\mu_\K(E)+\delta,$$
and such that $\gamma Q_i\cap \gamma Q_j=\emptyset$ for some small $\gamma=\gamma(m,M)>0$ whenever $i\neq j$. Using
Theorem \ref{Ainfty} and the doubling property of  $\omega_{\E}(\pi_X(A_{c\Delta,\Lambda}^+),\cdot)$ we see that
\begin{equation}
\begin{split}
\omega_{\K}(A_{c\Delta,\Lambda}^+, Q_j)&\leq \omega_{\K}(A_{c\Delta,\Lambda}^+, \Delta_{cQ_j})\\
&\lesssim \omega_{\E}(\pi_{X}(A_{c\Delta,\Lambda}^+),\pi_{X}(\Delta_{cQ_j}))\ell(cQ_j)^{3m+2}\lesssim \mu_\K(\gamma Q_j),
\end{split}
\end{equation}
where now the implicit constants may depend on $|\Delta|$, which is fixed.
Hence
\begin{equation}
\begin{split}
\omega_{\K}(A_{c\Delta,\Lambda}^+, E)&\leq \sum_j \omega_{\K}(A_{c\Delta,\Lambda}^+, Q_j)\\
&\lesssim \sum_j \mu_\K(\gamma Q_j)\lesssim\mu_\K(\cup Q_j)\lesssim(\mu_\K(E)+\delta).
\end{split}
\end{equation}
In particular, given $\epsilon >0$ there exists $\delta=\delta(m,\kappa,M,\epsilon,|\Delta|)>0$ such that if $E\subset \Delta$, and if
$\mu_\K(E)<\delta$, then $\omega_{\K}(A_{c\Delta,\Lambda}^+, E)<\epsilon$, proving that $\omega_{\K}(A_{c\Delta,\Lambda}^+, \cdot)\ll \mu_K$.

By the above we can conclude that $\omega_{\K}(A_{c\Delta,\Lambda}^+, \cdot)\ll \sigma_\K$ on $\Delta$ and that
 \begin{eqnarray*}
K_\K (A_{c\Delta,\Lambda}^+,X,Y,t):=\frac {\d\omega_{\K}}{\d\sigma_\K}(A_{c\Delta,\Lambda}^+,X,Y,t))=\lim_{{\tilde r}\to 0}\frac {\omega_{\K}(A_{c\Delta,\Lambda}^+, \Delta_{\tilde r}(X,Y,t))}{\sigma_\K(\Delta_{\tilde r}(X,Y,t))}\end{eqnarray*}
  exists and is well-defined for $\sigma_\K$-almost every $(X,Y,t)\in\Delta$. Using Theorem \ref{Ainfty}
\begin{equation}\label{kernelrels}
\begin{split}
    \sigma_\K(\Delta)K_\K (A_{c\Delta,\Lambda}^+,X,Y,t)&\approx \sigma_\P(\pi_{X,t}(\Delta))K_\P (\pi_{X,t}(A_{c\Delta,\Lambda}^+),X,t)\\
    &\approx \sigma_\E(\pi_{X}(\Delta)) K_\E (\pi_{X}(A_{c\Delta,\Lambda}^+),X),
\end{split}
\end{equation}
whenever $(X,Y,t)\in \Delta$. Using the assumption on $K_\E (X)=K_\E (\pi_{X}(A_{c\Delta,\Lambda}^+),X)$, and \eqref{kernelrels}, it follows that
 $K_\K (X,Y,t):=K_\K (A_{c\Delta,\Lambda}^+,X,Y,t)$ satisfies $$K_\K \in B_q(\Delta,\d\sigma_\K)$$ with a constant $\tilde\Gamma$, $\tilde\Gamma=\tilde\Gamma(m,\kappa,M,\Gamma)$.  This completes the proof of Theorem \ref{Ainfty+}.

\section{The $L^p$ Dirichlet problem for $\L_\K$: Theorem \ref{DPequiv}}\label{sec3}

Recall the notation $\Sigma$ introduced in \eqref{space}.  Given $f\in L^1_{\mbox{loc}}(\Sigma,\d\sigma_\K)$, we let
$$M(f)(X,Y,t):=\sup_{\Delta_r=\Delta_r(X,Y,t)\subset\Sigma}\bariiint_{\Delta_r} |f|\, \d\sigma_\K,$$
denote the Hardy-Littlewood maximal function of $f$, with respect to $\sigma_\K$. In the following we assume that  $\omega_{\K}(A_{c\Delta,\Lambda}^+,\cdot)$ is mutually absolutely continuous on $\Delta$ with respect to  $\sigma_\K$ for every $\Delta:=\Delta_r(X_0,Y_0,t_0)\subset \partial\Omega\times \mathbb R^m\times \mathbb R$.

We first prove that $(i)$ implies $(ii)$ and
hence we assume, given $\Delta\subset \partial\Omega\times \mathbb R^m\times \mathbb R$, that $K_\K (A_{c\Delta,\Lambda}^+,\cdot,\cdot,\cdot)\in B_q(\Delta,\d\sigma_\K)$. As   $\omega_\K$ is a doubling measure we can use the classical results of Coifman-Fefferman, see Theorem IV in \cite{CF}, to conclude that
$K_\K (A_{c\Delta,\Lambda}^+,\cdot,\cdot,\cdot)\in B_{\tilde q}(\Delta,\d\sigma_\K)$ for some $\tilde q>q$ independent of $\Delta$. Let $\tilde p$ be the index dual to $\tilde q$ and note that $\tilde p<p$.

Consider first
$f \in C_0(\partial \Omega\times \mathbb R^m\times \mathbb R)$.  Let $(X_0,Y_0,t_0)\in\partial\Omega\times \mathbb R^m\times \mathbb R$, and recall the (non-tangential) cone $\Gamma^\eta(X_0,Y_0,t_0)$.  Let $(\hat X,\hat Y,\hat t) \in \Gamma^\eta(X_0,Y_0,t_0)$ and let $\delta:=d((\hat X,\hat Y,\hat t), (X_0,Y_0,t_0))$. Then, by Theorem \ref{EUDP} we know that there exists a unique bounded weak solution to $\L_\K u=0$ in $\Omega\times \mathbb R^m\times \mathbb R$ with $u=f$ on
$\partial\Omega\times \mathbb R^m\times \mathbb R$. Furthermore,
$$u(\hat X,\hat Y,\hat t)=\iiint_{\partial\Omega}K_\K(\hat X,\hat Y,\hat t, X, Y, t)f( X, Y, t)\, \d\sigma_\K( X, Y, t).$$
We write
\begin{align*}
u(\hat X,\hat Y,\hat t)&=\iiint_{\Delta_{4\delta}(X_0,Y_0,t_0)}K_\K(\hat X,\hat Y,\hat t, X, Y, t)f( X, Y, t)\, \d\sigma_\K( X, Y, t)\notag\\
& \quad +\sum_{j=2}^\infty \iiint_{R_j(X_0,Y_0,t_0)}K_\K(\hat X,\hat Y,\hat t, X, Y, t)f( X, Y, t)\, \d\sigma_\K( X, Y, t)\notag\\
&=:u_1(\hat X,\hat Y,\hat t)+\sum_{j=2}^\infty u_j(\hat X,\hat Y,\hat t),
\end{align*}
where $R_j(X_0,Y_0,t_0):=\Delta_{2^{j+1}\delta}(X_0,Y_0,t_0)\setminus \Delta_{2^{j}\delta}(X_0,Y_0,t_0)$.  Using
\begin{eqnarray}\label{kerdef}
K_\K (\hat X,\hat Y,\hat t,X,Y,t)=\frac {\d\omega_{\K}}{\d\sigma_\K}(\hat X,\hat Y,\hat t,X,Y,t))=\lim_{{\tilde r}\to 0}\frac {\omega_{\K}(\hat X,\hat Y,\hat t, \Delta_{\tilde r}(X,Y,t))}{\sigma_\K(\Delta_{\tilde r}(X,Y,t))},\end{eqnarray}
in combination with Theorem \ref{thm:carleson}, we see that
\begin{align*}
K_\K(\hat X,\hat Y,\hat t, X, Y, t)\lesssim K_\K(A_{c\Delta_{4\delta},\Lambda}^+, X, Y, t),
\end{align*}
whenever $( X, Y, t)\in \Delta_{4\delta}(X_0,Y_0,t_0)$, and where $\Delta_{4\delta}:=\Delta_{4\delta}(X_0,Y_0,t_0)$. Hence, using Cauchy-Schwarz,
\begin{align*}
|u_1(\hat X,\hat Y,\hat t)|
&\leq \sigma_\K(\Delta_{4\delta})\biggl (\bariiint_{\Delta_{4\delta}}| K_\K(A_{c\Delta_{4\delta},\Lambda}^+, X, Y, t)|^{\tilde q}\d\sigma_\K\biggr )^{1/\tilde q}(M(|f|^{\tilde p})(X_0,Y_0,t_0))^{1/\tilde p}\notag\\
&\leq c \omega_\K(A_{c\Delta_{4\delta},\Lambda}^+,\Delta_{4\delta})(M(|f|^{\tilde p})(X_0,Y_0,t_0))^{1/\tilde p}\\
&\leq c(M(|f|^{\tilde p})(X_0,Y_0,t_0))^{1/\tilde p}
\end{align*}
by $(i)$. Similarly, using also Lemma \ref{lem4.5-Kyoto1} we have
\begin{align*}
K_\K(\hat X,\hat Y,\hat t, X, Y, t)\lesssim 2^{-\alpha j}  K_\K(A_{c\Delta_{2^j\delta},\Lambda}^+, X, Y, t)
\end{align*}
whenever $( X, Y, t)\in R_j(X_0,Y_0,t_0)$. Using this estimate, and essentially just repeating the estimates  conducted in the estimate of $u_1$, we deduce that
\begin{align*}
|u_j(\hat X,\hat Y,\hat t)|\leq  c2^{-\alpha j}\big(M(|f|^{\tilde p})(X_0,Y_0,t_0)\big)^{1/\tilde p}.
\end{align*}
In particular,
\begin{align*}
|u(\hat X,\hat Y,\hat t)|\leq |u_1(\hat X,\hat Y,\hat t)|+\sum_{j=2}^\infty |u_j(\hat X,\hat Y,\hat t)|\leq c\big(M(|f|^{\tilde p})(X_0,Y_0,t_0)\big)^{1/\tilde p},
\end{align*}
and hence
\begin{equation*}
N(u)(X_0,Y_0,t_0)\leq c\big(M(|f|^{\tilde p})(X_0,Y_0,t_0)\big)^{1/\tilde p}.
\end{equation*}
We can conclude that
\begin{align}\label{impest}
\|N(u)\|_{L^{p}(\partial \Omega\times \mathbb R^m\times \mathbb R,\d\sigma_\K)}&\leq c\|(M(|f|^{\tilde p}))^{1/\tilde p}\|_{L^{p}(\partial \Omega\times \mathbb R^m\times \mathbb R,\d\sigma_\K)}\notag\\
&\leq c \|f\|_{L^{p}(\partial \Omega\times \mathbb R^m\times \mathbb R,\d\sigma_\K)},
\end{align}
by the continuity of the Hardy-Littlewood maximal function and where the constant $c$ depends only on $(m,\kappa,M,p)$.  We now remove the restriction that $f \in C_0(\partial \Omega\times \mathbb R^m\times \mathbb R)$. Indeed, given $f\in L^{p}(\partial \Omega\times \mathbb R^m\times \mathbb R,\d\sigma_\K)$ there exist, by density of $C_0(\partial \Omega\times \mathbb R^m\times \mathbb R)$ in
$L^{p}(\partial \Omega\times \mathbb R^m\times \mathbb R,\d\sigma_\K)$, a sequence of functions $\{f_j\}$, $f_j \in C_0(\partial \Omega\times \mathbb R^m\times \mathbb R)$, converging to $f$ in $L^{p}(\partial \Omega\times \mathbb R^m\times \mathbb R,\d\sigma_\K)$. In particular, there exists a sequence of functions $\{u_j\}$ where $u_j$ is the unique bounded weak solution to $\L_\K u_j=0$ in $\Omega\times \mathbb R^m\times \mathbb R$ with $u_j=f_j$ on
$\partial\Omega\times \mathbb R^m\times \mathbb R$. By \eqref{impest},
\begin{align}\label{cachy}
\|N(u_k-u_l)\|_{L^{p}(\partial \Omega\times \mathbb R^m\times \mathbb R,\d\sigma_\K)}\leq c \|f_k-f_l\|_{L^{p}(\partial \Omega\times \mathbb R^m\times \mathbb R,\d\sigma_\K)}\to 0\mbox{ as $k,l\to\infty$}.
\end{align}
Consider $U_X\times U_Y\times J\subset\mathbb R^{N+1}$, where $U_X\subset\mathbb R^{m}$ and $U_Y\subset\mathbb R^{m}$ are bounded domains and $J=(a,b)$ with $-\infty<a<b<\infty$. Assume that $\overline{U_X\times U_Y\times J}$ is compactly contained in  $\Omega\times\mathbb R^m\times\mathbb R$ and that the distance from
$\overline{U_X\times U_Y\times J}$ to $\partial\Omega\times\R^m\times\R$ is $r>0$. By a covering argument with cubes of size say $r/2$,  Lemma \ref{lem1}, and the finiteness of $N(u_j)$ in $L^p(\partial \Omega\times \mathbb R^m\times \mathbb R,\d\sigma_\K)$, it follows that
$\{u_j\}$ is uniformly bounded in $L^2(U_X\times U_Y\times J)$ whenever $\overline{U_X\times U_Y\times J}$ is compactly contained in $\Omega\times\mathbb R^m\times\mathbb R$. Using this, and the energy estimate of Lemma \ref{lem1en}, we can conclude that
\begin{equation}\label{ubdedWuucau}
\mbox{$\|\nabla_X u_j\|_{L^2(U_X\times U_Y\times J)}$ is uniformly bounded}.
\end{equation}
Using \eqref{ubdedWuucau} and the weak formulation of the equation
$\L_\K u_j=0$  it follows that $(X\cdot\nabla_Y-\partial_t)u_j$ is uniformly bounded, with respect to $j$, in  $L_{Y,t}^2(U_Y\times J,H_X^{-1}(U_X))$.
Let $ W(U_X\times U_Y\times J)$  be defined as in \eqref{weak1-+}. By the above argument we can conclude, whenever  $\overline{U_X\times U_Y\times J}$ is compactly contained in $\Omega\times\mathbb R^m\times\mathbb R$, that
\begin{equation}\label{ubdedWcau}
    \mbox{$\|u_j\|_{W(U_X\times U_Y\times J)}$ is uniformly bounded}.
\end{equation}
Using \eqref{cachy}, and arguing as in the deductions in \eqref{ubdedWuucau} and \eqref{ubdedWcau}, we can also conclude that
\begin{equation}\label{ubdedWcau+}
    \mbox{$\|u_k-u_l\|_{W(U_X\times U_Y\times J)}\to 0$ as $k,l\to\infty$}.
\end{equation}
Using \eqref{ubdedWcau+} it follows that a subsequence $\{u_{j_k}\}$ of $\{u_{j}\}$ converges to a weak solution $u$ to
\begin{align*}
	\L_{\K} u = 0  &\text{ in} \ \Omega\times \mathbb R^m\times \mathbb R,
\end{align*}
and that
\[
\|N(u)\|_{L^p(\partial \Omega\times \mathbb R^m\times \mathbb R,\d\sigma_\K)}
	 \le c \|f\|_{L^p(\partial \Omega\times \mathbb R^m\times \mathbb R,\d\sigma_\K)}.
\]
Note also, using the notation introduced above, that
\begin{align}\label{cachytrunk}
\|N(u-u_j)\|_{L^{p}(\partial \Omega\times \mathbb R^m\times \mathbb R,\d\sigma_\K)}\leq c \|f-f_j\|_{L^{p}(\partial \Omega\times \mathbb R^m\times \mathbb R,\d\sigma_\K)}\to 0\mbox{ as $j\to\infty$}.
\end{align}
To complete the proof that  $(i)$ implies $(ii)$  we have to prove that $u = f$ n.t. on $\partial \Omega\times \mathbb R^m\times \mathbb R$. Consider
$f\in L^{p}(\partial \Omega\times \mathbb R^m\times \mathbb R,\d\sigma_\K)$ and let $\{f_j\}$, $f_j \in C_0(\partial \Omega\times \mathbb R^m\times \mathbb R)$, be a sequence of functions converging to $f$ in $L^{p}(\partial \Omega\times \mathbb R^m\times \mathbb R,\d\sigma_\K)$. Let $(X_0,Y_0,t_0)\in\partial\Omega\times \mathbb R^m\times \mathbb R$ be a Lebesgue point
of $f$. Given $\delta>0$ we have
\begin{align}\label{atrunk}
N_\delta(u-f)(X_0,Y_0,t_0)&\leq N_\delta(u-u_j)(X_0,Y_0,t_0)+N_\delta(u_j-f_j)(X_0,Y_0,t_0)\notag\\
&+M(f-f_j)(X_0,Y_0,t_0),
\end{align}
where $N_\delta$ was introduced in \eqref{eq:N*trunc} and $N_\delta(u-f)(X_0,Y_0,t_0)$ should be interpreted as $$\sup_{(X,Y,t) \in \Gamma_\delta^\eta(X_0,Y_0,t_0)} |u(X,Y,t)-f(X_0,Y_0,t_0)|.$$
In the following we assume, as we may without loss of generality, that $(0,0,0)\in
\partial \Omega\times \mathbb R^m\times \mathbb R$. Given $\epsilon>0$ small and $R\gg 1$, let $$S_\epsilon(R,\delta):=\{(X,Y,t)\in\Delta_R(0,0,0):\ N_\delta(u-f)(X,Y,t)>\epsilon\}.$$ Using \eqref{atrunk}, weak estimates and \eqref{cachytrunk} we deduce
\begin{align}\label{atrunkmeas}
\sigma_\K(S_\epsilon(R,\delta))\leq c\epsilon^{-p}\bigl ( \|f-f_j\|^p_{L^{p}(\partial \Omega\times \mathbb R^m\times \mathbb R,\d\sigma_\K)}+
\|N_\delta(u_j-f_j)\|^p_{L^{p}(\Delta_R(0,0,0),\d\sigma_\K)}\bigr ).
\end{align}
Now letting $\delta\to 0$, $j\to\infty$, $R\to\infty$, in that order, we deduce that the set of points $(X_0,Y_0,t_0)\in \partial \Omega\times \mathbb R^m\times \mathbb R$ at which
\[
\lim_{\substack{(X,Y,t)\in \Gamma^\eta(X_0,Y_0,t_0)\\ (X,Y,t)\to (X_0,Y_0,t_0)}}|u(X,Y,t)-f(X_0,Y_0,t_0)|>\epsilon,
\]
has $\sigma_\K$ measure zero. As $\epsilon$ is arbitrary we can conclude that $u = f$ n.t. on $\partial \Omega\times \mathbb R^m\times \mathbb R$.

Next we prove that $(ii)$ implies $(i)$ and
hence we assume  that $D_{\K}^p(\partial \Omega\times \mathbb R^m\times \mathbb R,\d\sigma_\K)$ is solvable. Let $(X_0,Y_0,t_0)\in\partial\Omega$,
$\Delta:=\Delta_r(X_0,Y_0,t_0)\subset \partial\Omega\times \mathbb R^m\times \mathbb R$, $f \in C_0(\Delta)$, $f\geq 0$. Let $u$ be the unique bounded solution to the Dirichlet problem with boundary data $f$. Then
$$u(A_{c\Delta,\Lambda}^+)=\iiint_{\Delta}K_\K(A_{c\Delta,\Lambda}^+, X, Y, t)f( X, Y, t)\, \d\sigma_\K( X, Y, t).$$
Using the estimate in Lemma \ref{lem1}, and $(ii)$,
\begin{align*}
u(A_{c\Delta,\Lambda}^+)&\lesssim \biggl (\bariiint_{Q_{r/c}(A_{c\Delta,\Lambda}^+)} |u( X, Y, t)|^p\, \d  X\d  Y\d  t\biggr )^{1/p}\\
&\lesssim \biggl (\frac 1{\sigma_\K(\Delta)}\iiint_{4\Delta} |N(u)( X, Y, t)|^p\, \d\sigma_\K( X, Y, t)\biggr )^{1/p}\\
&\lesssim \biggl (\frac 1{\sigma_\K(\Delta)}\iiint_{\Delta} |f( X, Y, t)|^p\, \d\sigma_\K( X, Y, t)\biggr )^{1/p}.
\end{align*}
In particular, for all $f \in C_0(\Delta)$  with $\| f \|_{L^{p}(\partial \Omega\times \mathbb R^m\times \mathbb R,\d\sigma_\K)}=1$, we have
\begin{equation*}
\biggl |\iiint_{\Delta}K_\K(A_{c\Delta,\Lambda}^+, X, Y, t)f( X, Y, t)\, \d\sigma_\K( X, Y, t)\biggr |\leq \biggl (\frac 1{\sigma_\K(\Delta)}\biggr )^{1/p}.
\end{equation*}
Hence, {since $(\Delta, \sigma_\K)$ is a finite measure space,}
\begin{equation*}
\biggl (\iiint_{\Delta}|K_\K(A_{c\Delta,\Lambda}^+, X, Y, t)|^{q}\d\sigma_\K( X, Y, t)\biggr )^{1/q}\leq \biggl (\frac 1{\sigma_\K(\Delta)}\biggr )^{1/p}.
\end{equation*}
Furthermore,  Lemma \ref{lem4.5-Kyoto1} and Lemma \ref{lem4.7} imply
\begin{equation*}
\iiint_{\Delta}K_\K(A_{c\Delta,\Lambda}^+, X, Y, t)\d\sigma_\K( X, Y, t)=\omega_\K(A_{c\Delta,\Lambda}^+,\Delta)\gtrsim 1.
\end{equation*}
Combining the estimates,
\begin{equation*}\label{K1a}
\biggl (\bariiint_{\Delta}|K_\K(A_{c\Delta,\Lambda}^+, X, Y, t)|^{q}\d\sigma_\K( X, Y, t)\biggr )^{1/q}\lesssim \bariiint_{\Delta}K_\K(A_{c\Delta,\Lambda}^+, X, Y, t)\d\sigma_\K( X, Y, t).
\end{equation*}
Hence $K_\K (A_{c\Delta,\Lambda}^+,\cdot,\cdot,\cdot)\in B_q(\Delta,\d\sigma_\K)$ and  the proof that $(ii)$ implies $(i)$ is complete. Put together we have proved that the statements in Theorem \ref{DPequiv} $(i)$ and $(ii)$ are equivalent.

\subsection{Proof of the uniqueness statement in Theorem \ref{DPequiv}}
Having proved that Theorem \ref{DPequiv} $(i)$ and $(ii)$ are equivalent it remains to prove that if   $D_{\K}^p(\partial \Omega\times \mathbb R^m\times \mathbb R,\d\sigma_\K)$ is solvable, then $D_{\K}^p(\partial \Omega\times \mathbb R^m\times \mathbb R,\d\sigma_\K)$ is uniquely solvable. That is, we have to prove that if
$N(u)\in {L^p(\partial \Omega\times \mathbb R^m\times \mathbb R,\d\sigma_\K)}$, and  if {$u$ is a weak solution} to the Dirichlet problem
\begin{equation*}
\begin{cases}
	\L_{\K} u = 0  &\text{in} \ \Omega\times \mathbb R^m\times \mathbb R, \\
      u = 0  & \text{n.t. on} \ \partial \Omega\times \mathbb R^m\times \mathbb R,
\end{cases}
\end{equation*}
then $u\equiv 0$ in $\Omega\times \mathbb R^m\times \mathbb R$. Note that the proof of this is considerably more involved compared to the corresponding arguments in the elliptic setting, \cite{Ke} and \cite{KS}. One reason is, again, the (time)-lag in the Harnack inequality for parabolic equations.

To start the proof we fix $(\hat X,\hat Y,\hat t)\in\Omega\times \mathbb R^m\times \mathbb R$ and we intend to prove that $u(\hat X,\hat Y,\hat t)=0$. Let $\theta\in C_0^\infty(\Omega\times \mathbb R^m\times \mathbb R)$ with $\theta=1$ in a neighborhood of $(\hat X,\hat Y,\hat t)$. Then, using \eqref{ghh10-},
    \begin{equation}\label{eq9+moa}
    \begin{split}
    u(\hat X,\hat Y,\hat t)&=(u\theta)(\hat X,\hat Y,\hat t)\\
    &=-\iiint{A( X)\nabla_{ X}G_\K (\hat X,\hat Y,\hat t, X, Y, t)}\cdot \nabla_{ X} (u\theta)( X, Y, t)\,
    \d  X \d Y\d t\\
    &\quad+\iiint{G_\K (\hat X,\hat Y,\hat t, X, Y, t)}( X\cdot\nabla_{ Y}-\partial_{ t})(u\theta)( X, Y, t)\, \d  X \d Y\d t.
    \end{split}
    \end{equation}
    By the results in \cite{Ietal}, see Lemma \ref{lem1+}, we know that any weak solution to $\L_\K u=0$ is H{\"o}lder continuous. As $A$ is independent of $(Y,t)$,
    it follows that partial derivatives of $u$ with respect to $Y$ and $t$ are also weak solutions. As a consequence, as $A$ is independent of $(Y,t)$, any weak solution to $\L_\K u=0$ is $C^\infty$-smooth as a function of $(Y,t)$.  Hence the term
    $( X\cdot\nabla_{ Y}-\partial_{ t})(u\theta)$ appearing in the last display is well-defined. Using \eqref{eq9+moa}, and that $\L_\K u=0$,
       \begin{eqnarray}\label{eq9+moab}
    |u(\hat X,\hat Y,\hat t))|\lesssim (I+II+III),
    \end{eqnarray}
    where
        \begin{align}\label{eq9+moabc}
            I&:=\iiint|G_\K (\hat X,\hat Y,\hat t, X, Y, t))|\nabla_{ X} u( X, Y, t)||\nabla_{ X}\theta( X, Y, t)|\, \d  X \d Y\d t,\notag\\
    II&:=\iiint|\nabla_{ X}G_\K (\hat X,\hat Y,\hat t, X, Y, t)|u( X, Y, t)||\nabla_{ X}\theta( X, Y, t)|\, \d  X \d Y\d t,\\
    III&:=\iiint|G_\K (\hat X,\hat Y,\hat t, X, Y, t)||u( X, Y, t)||(\partial_{ t}- X\cdot\nabla_{ Y})\theta( X, Y, t)|\, \d  X \d Y\d t.\notag
    \end{align}

     Recall the notation  $Q:=(-1,1)^m\times (-1,1)^m\times (-1,1)$. Given $(\hat X,\hat Y,\hat t)=(\hat x,\hat x_m,\hat Y,\hat t)\in \Omega\times \mathbb R^m\times \mathbb R$ fixed, we have $$((\hat x,\psi(\hat x)),\hat Y,\hat t)\in \partial\Omega\times \mathbb R^m\times \mathbb R$$ fixed. We consider $ Q_R((\hat x,\psi(\hat x)),\hat Y,\hat t)=((\hat x,\psi(\hat x)),\hat Y,\hat t)\circ  Q_R$ and we let $\epsilon$ and $R$ satisfy
    $$\mbox{$\epsilon<\lambda/8$, $R>8\lambda$ where $\lambda:=\hat x_{m}-\psi(\hat x)$.}$$
    When taking limits, we will always first let $\epsilon\to 0$ before letting $R\to \infty$.

     Let $\varphi_1=\varphi_1(X,Y,t)\in C_0^\infty( Q_{2R}((\hat x,\psi(\hat x)),\hat Y,\hat t))$, $0\leq\varphi_1\leq 1$, be such that $\varphi_1\equiv 1$ on
     $Q_{R}((\hat x,\psi(\hat x)),\hat Y,\hat t)$. Let $\varphi_2=\varphi_2(X)=\varphi_2(x,x_m)$ be a smooth function with range $[0,1]$ such that $\varphi_2(x,x_m)\equiv 1$
     on $\{(x,x_m):\  x_m\geq \psi( x)+2\epsilon\}$, $\varphi_2(x,x_m)\equiv 0$
     on $\{(x,x_m):\  x_m\leq \psi( x)+\epsilon\}$. Note that $\varphi_1$ can be constructed so that
     $||R\nabla_X\varphi_1||_{L^\infty}+||R^2(X\cdot\nabla_{Y}-\partial_{t})\varphi_1||_{L^\infty}\lesssim 1$. Similarly, $\varphi_2$ can be constructed so that
     $||\epsilon\nabla_X\varphi_2||_{L^\infty}\leq c$ where $c$ is independent of $\epsilon$. We let $$\theta=\theta(X,Y,t)=\theta(x,x_m,Y,t):=\varphi_1(X,Y,t)\varphi_2(x,x_m).$$
     Then $\theta\in C_0^\infty( Q_{2R}((\hat x,\psi(\hat x)),\hat Y,\hat t)$, $0\leq\theta\leq 1$, $\theta\equiv 1$ on
    the set of points $ ( X, Y, t)=( x, x_m, Y, t)\in Q_{R}((\hat x,\psi(\hat x)),\hat Y,\hat t)$ which satisfies $ x_m\geq \psi( x)+2\epsilon$ and  $\theta\equiv 0$ on
    the set of points in $ ( X, Y, t)=( x, x_m, Y, t)\in Q_{R}((\hat x,\psi(\hat x)),\hat Y,\hat t)$ which satisfies $ x_m\leq \psi( x)+\epsilon$.  Let
    \begin{eqnarray*}
(i)&&D_1:= Q_{2R}((\hat x,\psi(\hat x)),\hat Y,\hat t)\cap  \{( X, Y, t)\mid  \psi( x)+\epsilon< x_m<\psi( x)+2\epsilon\},\notag\\
(ii)&&D_2:= Q_{2R}((\hat x,\psi(\hat x)),\hat Y,\hat t)\cap  \{( X, Y, t) \mid \psi( x)+R< x_m<\psi( x)+2R\},\notag\\
(iii)&&D_3:= D_4\cap  \{( X, Y, t) \mid \psi( x)+2\epsilon\leq x_m\leq \psi( x)+R\},
    \end{eqnarray*}
where
    \begin{eqnarray*}\label{eq9+moabcdllla}
D_4:= Q_{2R}((\hat x,\psi(\hat x)),\hat Y,\hat t)\setminus  Q_{R}((\hat x,\psi(\hat x)),\hat Y,\hat t).
    \end{eqnarray*}
Using this notation, the domains where the integrands in $I,II,$ and $III$ are non-zero are contained in
    the union $D_1\cup D_2\cup D_3$. By the construction of $\theta$,
                \begin{eqnarray*}
(i')&& ||\epsilon\nabla_X\theta||_{L^\infty(D_1)}+||R^2(X\cdot\nabla_{Y}-\partial_{t})\theta||_{L^\infty(D_1)}\leq c,\notag\\
(ii')&& ||R\nabla_X\theta||_{L^\infty(D_2)}+||R^2(X\cdot\nabla_{Y}-\partial_{t})\theta||_{L^\infty(D_2)}\leq c,\notag\\
(iii')&& ||R\nabla_X\theta||_{L^\infty(D_3)}+||R^2(X\cdot\nabla_{Y}-\partial_{t})\theta||_{L^\infty(D_3)}\leq c,
    \end{eqnarray*}
    where $c$ is a constant which is independent of $\epsilon$ and $R$. Note that if $( X, Y, t)\in D_3$, then $\theta( X, Y, t)=\varphi_1(X,Y,t)$ and this explains $(iii')$.

    Using the sets $D_1,D_2,$ and $D_3$, and letting $$G_\K (\cdot,\cdot,\cdot):=G_\K (\hat X,\hat Y,\hat t,\cdot,\cdot,\cdot),$$ we see that
              \begin{align}\label{eq9+moabcg}
            I+II+III&\lesssim T_1+T_2+T_3,
    \end{align}
    where
        \begin{align*}
          T_1&:= \frac 1{\epsilon^2}\iiint_{D_1}(\epsilon|G_\K ||\nabla_{ X} u|+\epsilon|\nabla_{ X} G_\K || u|+\epsilon^2R^{-2}| G_\K || u|)\, \d X \d Y\d t,\notag\\
           T_2&:= \frac 1{R^2}\iiint_{D_2}(R|G_\K ||\nabla_{ X} u|+R|\nabla_{ X} G_\K || u|+| G_\K || u|)\, \d X \d Y\d t,\\
            T_3&:= \frac 1{R^2}\iiint_{D_3}(R|G_\K ||\nabla_{ X} u|+R|\nabla_{ X} G_\K || u|+| G_\K || u|)\, \d X \d Y\d t.
    \end{align*}
    We need to estimate $T_1,T_2,$ and $T_3$. To improve readability we will in the following use the notation
\[
\Delta_{\rho} := (\partial\Omega\times\R^m\times\R)\cap Q_{\rho}((\hat x,\psi(\hat x)),\hat Y,\hat t)\mbox{ for $\rho>0$}.
\]

We first consider $T_1$. We start by estimating the contribution from the term $| G_\K || u|$ and in this case we prove a harder estimate than we need. The argument will be used for further reference. Note that
\begin{equation*}
\begin{split}
    &\frac{1}{\epsilon^2} \iiint_{D_1}{ |G_\K | |u| \d X \d Y \d t}\notag\\
    &\lesssim
    \iiint_{\Delta_{2R}}{\tilde N_{\epsilon}(u)\left(
        \frac{1}{\epsilon} \int_{\psi( x)+\epsilon}^{\psi( x)+2\epsilon}{ \frac{G_\K ((x,x_m), Y, t)}{\epsilon} \d{x}_m}\right)\d\sigma_\mathcal{K}}\\
    &\lesssim \|\tilde N_\epsilon(u)\|_{L^p(\Delta_{2R},\d\sigma_\mathcal{K})} \left( \iiint_{\Delta_{2R}}{\left(
      \frac{1}{\epsilon} \int_{\psi( x)+\epsilon}^{\psi( x)+2\epsilon}{ \frac{G_\K ((x,x_m), Y, t)}{\epsilon} \d{x}_m}\right)^{q} \d\sigma_\mathcal{K}} \right)^{1/q},
\end{split}
\end{equation*}
where $\tilde N_\epsilon$ is a truncated maximal operator defined as
$$\tilde N_\epsilon(u)( X, Y, t):=\sup_{\psi( x)< x_m<\psi( x)+2\epsilon}|u((x,x_m), Y, t)|.$$
Using  Lemma \ref{greenmeasurerelation} and the definition of $K_\mathcal{K}$, see \eqref{kerdef}, we have, for every $( X, Y, t)\in\Delta_{2R}$, $1\leq\sigma\leq 2$, and denoting by $e_m$ the unit vector in $\mathbb R^m$ pointing into $\Omega$ in the $x_m$ direction,
\begin{align*}
\lim_{\epsilon\rightarrow 0}\frac{G_\K (\hat X,\hat Y,\hat t, X+\sigma\epsilon e_m, Y, t)}{\epsilon} &\lesssim \lim_{\epsilon\rightarrow 0} \frac{\omega_{\mathcal{K}}(\hat X,\hat Y,\hat t,\Delta_{c\sigma\epsilon}( X, Y, t))}{\epsilon^{{\bf q}-1}}\notag\\
& \lesssim K_\mathcal{K}(\hat X,\hat Y,\hat t, X, Y, t).
\end{align*}
Note that if $\hat t\leq t$, then  this is trivial as the left hand side is identically zero. If $\hat t> t$, then we may apply Lemma \ref{greenmeasurerelation} in the deduction as we are considering the limiting situation $\epsilon\rightarrow 0$.
Using these estimates, and Lebesgue's theorem  on dominated convergence, we obtain
\begin{equation}\label{limit}
\begin{split}
&\limsup_{\epsilon\to 0}\frac{1}{\epsilon^2} \iiint_{D_1}{ |G_\K | |u| \d X \d Y \d t}\\
& \lesssim \biggl (\limsup_{\epsilon\to 0}\|\tilde N_\epsilon(u)\|_{L^p(\Delta_{2R},\d\sigma_\mathcal{K})} \biggr )\|K_\mathcal{K}(\hat X,\hat Y,\hat t,\cdot,\cdot,\cdot)\|_{L^q(\Delta_{2R},\d\sigma_\mathcal{K})}=0,
\end{split}
\end{equation}
as $u$ vanishes at the boundary in the non-tangential sense.  We next consider the term
\[
\frac{1}{\epsilon} \iiint_{D_1}{ |G_\K | |\nabla_Xu| \d X \d Y \d t}.
\]
In this case, we first note, using Lemma \ref{lem4.7} and the construction of $D_1$, that if $\epsilon$ is small enough, then
\begin{align}\label{upper}
G_\K (X, Y, t)=G_\K (\hat X,\hat Y,\hat t, X, Y, t)\lesssim (R/\lambda)^\gamma G_\K (A^+_{c\Delta_R,\Lambda}, X, Y, t)
\end{align}
whenever $(X,Y,t)\in D_1$. Let $\{Q_j\}$ be all Whitney cubes in a Whitney decomposition of $\Omega\times\mathbb R^m\times\mathbb R$ which intersects $D_1$. Then $|Q_j|\approx \epsilon^{\bf q}$. Using \eqref{upper} and  H{\"o}lder's inequality
\begin{equation}\label{esto-}
    \begin{split}
       & \frac{1}{\epsilon} \iiint_{D_1}{ |G_\K | |\nabla_Xu| \d X \d Y \d t}\notag\\
        &\lesssim (R/\lambda)^\gamma\frac{1}{\epsilon}\sum_j \iiint_{Q_j}{ G_\K (A^+_{c\Delta_R,\Lambda}, X, Y, t) |\nabla_Xu| \d X \d Y \d t}\notag\\
        &\lesssim   (R/\lambda)^\gamma\frac{1}{\epsilon}\sum_j  \left(\iiint_{Q_j}{ |G_\K (A^+_{c\Delta_R,\Lambda}, X, Y, t)|^2 \d X \d Y \d t}\right)^{1/2} \left(\iiint_{Q_j}{ |\nabla_Xu|^2 \d X \d Y \d t}\right)^{1/2}.
    \end{split}
\end{equation}
Using the adjoint version of Lemma \ref{lem4.7}, and Lemma \ref{greenestimate}, we see that
\begin{equation}\label{esto}
    \begin{split}
  \sup_{4Q_j}G_\K (A^+_{c\Delta_R,\Lambda}, X, Y, t)\lesssim \inf_{4Q_j}G_\K (A^+_{c\Delta_R,\Lambda}, X, Y, t).
    \end{split}
\end{equation}
Furthermore, using the energy estimate of Lemma \ref{lem1en}, assuming that the Whitney decomposition is such that $8Q_j\subset \Omega\times\mathbb R^m\times\mathbb R$,
\begin{equation}\label{esto+}
    \begin{split}
        \iiint_{Q_j}{|\nabla_Xu|^2\d X \d Y \d t} &\lesssim \epsilon^{-2}\iiint_{2Q_j}{|u|^2\d X \d Y \d t}\lesssim\epsilon^{-2} |Q_j| \big(\sup_{2Q_j}|u|\big)^2.
    \end{split}
\end{equation}
Using \eqref{esto-}-\eqref{esto+} we deduce
\begin{equation}\label{esto-+}
    \begin{split}
        &\frac{1}{\epsilon} \iiint_{D_1}{ |G_\K | |\nabla_Xu| \d X \d Y \d t}\\
         &\lesssim (R/\lambda)^\gamma\frac{1}{\epsilon^2}
        \sum_j |Q_j|\big(\inf_{4Q_j} G_\K (A^+_{c\Delta_R,\Lambda}, X, Y, t)\big)\big(\sup_{2Q_j} |u( X, Y, t)|\big).
    \end{split}
\end{equation}
Using Lemma \ref{lem1}
\begin{eqnarray}\label{usesubsolest}
\sup_{2Q_j}\ |u|\lesssim\biggl (\bariiint_{4Q_j}\ |u|\, \d X\d Y \d t\biggr ).
\end{eqnarray}
This inequality in combination with \eqref{esto-+} imply that
\begin{equation}\label{esto-+re}
    \begin{split}
       & \frac{1}{\epsilon} \iiint_{D_1}{ |G_\K | |\nabla_Xu| \d X \d Y \d t}\\
        &\lesssim (R/\lambda)^\gamma\frac{1}{\epsilon^2}\iiint_{\tilde D_1}G_\K (A^+_{c\Delta_R,\Lambda}, X, Y, t))|u( X, Y, t)|\d X \d Y \d t,
    \end{split}
\end{equation}
where $\tilde D_1$ is the enlargement of $D_1$ defined as the union of the cubes $\{4Q_j\}$. We can now repeat the argument leading up to \eqref{limit}, with $G_\K $ replaced by $G_\K (A^+_{c\Delta_R,\Lambda}, \cdot,\cdot,\cdot)$ and with $D_1$ replaced by $\tilde D_1$, to conclude that
\begin{equation}\label{esto-+a}
    \begin{split}
       \limsup_{\epsilon\to 0} \frac{1}{\epsilon} \iiint_{D_1}{ |G_\K | |\nabla_Xu| \d X \d Y \d t}=0.
    \end{split}
\end{equation}
The remaining term in $T_1$ can be handled analogously and hence we can conclude that
\begin{equation}\label{1limit}
T_1 \rightarrow 0,\:\text{as }\epsilon\rightarrow 0.
\end{equation}

Next we consider $T_2$ and we first consider the contribution from the term
\begin{equation}\label{1sttermT2}
    \frac 1{R^2}\iiint_{D_2}| G_\K || u|\, \d X \d Y\d t.
\end{equation}
In this case we first note, using Lemma \ref{lem_fsolbounds}, that
 $$G_\K (X,Y,t)=G_\K (\hat X,\hat Y,\hat t,X,Y,t)\leq \Gamma_\K (\hat X,\hat Y,\hat t,X,Y,t)\lesssim R^{2-{\bf q}}$$
 whenever $(X,Y,t)\in D_2$. Hence,
\begin{equation*}
    \begin{split}
        \frac 1{R^2}\iiint_{D_2}| G_\K || u|\, \d X \d Y\d t &\lesssim R^{1-{\bf q}}\iiint_{\Delta_{2R}}{N(u) \d\sigma_\mathcal{K}}\\
        &\lesssim R^{1-{\bf q}} R^{({\bf q}-1)(1-1/p)} \|N(u)\|_{L^p(\Delta_{2R},\d\sigma_\mathcal{K})}\\
        &= R^{\frac{1-\bf q}{p}} \|N(u)\|_{L^p(\Delta_{2R},\d\sigma_\mathcal{K})} \rightarrow 0,\text{as }R\rightarrow \infty.
    \end{split}
\end{equation*}
We next consider the contribution from the term
\[
    \frac 1{R}\iiint_{D_2}| G_\K ||\nabla_{ X} u|\, \d X \d Y\d t.
\]
Using the energy estimate of Lemma \ref{lem1en}, as well as Lemma \ref{lem1},
\begin{equation*}
    \begin{split}
     \biggl (\iiint_{D_2}|\nabla_{ X} u|^2\, \d X \d Y\d t\biggr )^{1/2}&\lesssim R^{-1-{\bf q}/2}\iiint_{\tilde D_2}|u|\, \d X \d Y\d t,
    \end{split}
\end{equation*}
where $\tilde D_2$ is an enlargement of $D_2$. Using this, and also again using the bound on $G_\K $ stated above, we see that
\begin{equation*}
    \begin{split}
        \frac 1{R}\iiint_{D_2}| G_\K ||\nabla_{ X} u|\, \d X \d Y\d t &\lesssim R^{1-{\bf q}/2}R^{-1-{\bf q}/2}\iiint_{\tilde D_2}|u|\, \d X \d Y\d t\notag\\
        &\lesssim R^{1-{\bf q}}\iiint_{\Delta_{4R}}|N(u)|\d\sigma_K\notag\\
        &\lesssim R^{1-{\bf q}} R^{({\bf q}-1)/q} \|N(u)\|_{L^p(\Delta_{4R},\d\sigma_\mathcal{K})}\\
        &\lesssim R^{\frac{1-{\bf q}}{p}} \|N(u)\|_{L^p(\Delta_{4R},\d\sigma_\mathcal{K})} \rightarrow 0,\:\text{as }R\rightarrow \infty.
    \end{split}
\end{equation*}
The remaining term in $T_2$ can be handled analogously and hence we can conclude that
\begin{equation}\label{2limit}
T_2 \rightarrow 0,\: \text{as } R\rightarrow \infty.
\end{equation}

Finally we consider $T_3$. The term in $T_3$ containing the integrand $|G_\K ||u|$ can be handled as we handled the term in \eqref{1sttermT2}. To handle the other terms  we first recall that by construction $G_\K (\hat X,\hat Y,\hat t, X, Y, t)\neq 0$ if and only if $t<\hat t$. Furthermore,
for $(X, Y, t)\in D_3$ fixed  $G_\K (\cdot,\cdot,\cdot, X, Y, t)$ is a non negative solution to $\L_\K u=0$ in $(\Omega\times\mathbb R^m\times\mathbb R)\cap Q_{R}((\hat x,\psi(\hat x)),\hat Y,\hat t)$. In particular, if $R$ is large enough, then by Theorem \ref{thm:carleson} we have that
\begin{align}\label{upperre}
G_\K (\hat X,\hat Y,\hat t, X, Y, t)\lesssim  G_\K (A^+_{c^{-1}\Delta_R,\Lambda}, X, Y, t),
\end{align}
whenever $(X,Y,t)\in D_3$ and we can ensure that $A^+_{c^{-1}\Delta_R,\Lambda}\subset Q_{R/2}((\hat x,\psi(\hat x)),\hat Y,\hat t)$. To proceed we let
$C=C(m)\gg 1$ be  a large but fixed constant, and we introduce
\begin{eqnarray*}
D_3^\ast:=D_3\cap  \{( X, Y, t):\  \psi(x)+2\epsilon\leq x_m\leq \psi( x)+R/C\}.
    \end{eqnarray*}
Then the domain of integration  in the terms defining $T_3$ is partitioned into integration over $D_3^\ast$ and $D_3\setminus D_3^\ast$. Integration over the latter set can be handled as we handled $T_2$. Therefore we here only consider the remaining terms in $T_3$ but with domain of integration defined by $D_3^\ast$.   We now let
 $\{Q_j\}$ be all Whitney cubes in a  Whitney decomposition of $\Omega\times\mathbb R^m\times\mathbb R$ which intersects $D_3^\ast$. Focusing on the term in $T_3$  containing the integrand  $|G_\K | |\nabla_Xu|$ we see that
\begin{align}\label{esto-re}
       & \frac{1}{R} \iiint_{D_3^\ast}{ |G_\K | |\nabla_Xu| \d X \d Y \d t}\notag\\
       &\leq \frac{1}{R}\sum_j \iiint_{Q_j\cap D_3^\ast}{ |G_\K | |\nabla_Xu| \d X \d Y \d t}\notag\\
       &\lesssim \frac{1}{R} \sum_j |Q_j|^{1/2} l(Q_j)^{-1}\left(\iiint_{Q_j\cap D_3^\ast}{ |G_\K |^2 \d X \d Y \d t}\right)^{1/2} \left(\bariiint_{4Q_j}{ |u| \d X \d Y \d t}\right)\notag\\
    &\lesssim \frac{1}{R} \sum_j  |Q_j|l(Q_j)^{-1}\bigl(\sup_{Q_j} G_\K (A^+_{c^{-1}\Delta_R,\Lambda}, X, Y, t)\bigr) \left(\bariiint_{4Q_j}{ |u| \d X \d Y \d t}\right),
\end{align}
where we have used  Lemma \ref{lem1en},  Lemma \ref{lem1} and \eqref{upperre}. Furthermore, \eqref{esto} remains valid in this context and hence
\begin{equation}\label{estore}
\begin{split}
&\bigl(\sup_{Q_j} G_\K (A^+_{c^{-1}\Delta_R,\Lambda}, X, Y, t)\bigr) \left(\bariiint_{4Q_j}{ |u| \d X \d Y \d t}\right)\\
&\lesssim \left(\bariiint_{4Q_j}{ G_\K (A^+_{c^{-1}\Delta_R,\Lambda}, X, Y, t)|u| \d X \d Y \d t}\right).
\end{split}
\end{equation}
Combining these insights we see, using the notation $\delta(X):=(x_m-\psi(x))$, that
\begin{align}\label{esto-rere}
        &\frac{1}{R} \iiint_{D_3^\ast}{ |G_\K | |\nabla_Xu| \d X \d Y \d t}\notag\\
        &\lesssim \frac{1}{R} \sum_j l(Q_j)^{-1}\left(\iiint_{4Q_j}{ G_\K (A^+_{c^{-1}\Delta_R,\Lambda}, X, Y, t)|u| \d X \d Y \d t}\right)\\
       &\lesssim \frac{1}{R} \left(\iiint_{\tilde D_3^\ast}{ G_\K (A^+_{c^{-1}\Delta_R,\Lambda}, X, Y, t)|u| \delta(X)^{-1}\d X \d Y \d t}\right),\notag
\end{align}
where $\tilde D_3^\ast$ is a slight enlargement of $D_3^\ast$ due to the enlargement from $Q_j$ to $4Q_j$. In particular,
\begin{equation}\label{esto-rerere}
\begin{split}
        &\frac{1}{R} \iiint_{D_3^\ast}{ |G_\K | |\nabla_Xu| \d X \d Y \d t}\\
        &\lesssim \frac{1}{R} \left(\iiint_{D_5}{ G_\K (A^+_{c^{-1}\Delta_R,\Lambda}, X, Y, t)|u| \delta(X)^{-1}\d X \d Y \d t}\right),
\end{split}
\end{equation}
where {$D_5$ is defined as the set
\begin{eqnarray*}
 (\Omega\times\mathbb R^m\times\mathbb R)\cap \big(Q_{cR}(\hat X,\hat Y,\hat t)\setminus  \lbrace (X,Y,t)\mid (x,\psi(x),Y,t)\in \Delta_{R/c},\ \psi(x)\leq x_m < \psi(x) + 2cR \rbrace\big), 
\end{eqnarray*}
for some $c=c(m)\gg 1$.} Note that points in $D_5$ can be represented as $$(X,Y,t)=((x,\psi(x)),Y,t)+(0,\delta(X),0,0),$$ where
    $((x,\psi(x)),Y,t)\in \Delta_{cR}\setminus  \Delta_{R/c}.$
    Consider one such point $(X,Y,t)$. We claim that
\begin{align}\label{esto-rererere}
    G_\K (A^+_{c^{-1}\Delta_R,\Lambda}, X, Y, t)\delta(X)^{-1}\lesssim {M}\bigl(K_\K(A^+_{c^{-1}\Delta_R,\Lambda},\cdot)\chi_{\Delta_{2cR}\setminus  \Delta_{R/(2c)}}(\cdot)\bigr )((x,\psi(x)),Y,t),
\end{align}
where again ${M}$ denotes the Hardy-Littlewood maximal function on $\partial\Omega\times\mathbb R^m\times \mathbb R$ w.r.t $\sigma_\K$, and
$\chi_{\Delta_{cR}\setminus  \Delta_{R/c}}(\cdot)$ is the indicator function for the set $\Delta_{cR}\setminus  \Delta_{R/c}$. To prove \eqref{esto-rererere} we simply note, using Lemma \ref{greenmeasurerelation}, that
\begin{align*}\label{esto-rererere}
   G_\K (A^+_{c^{-1}\Delta_R,\Lambda}, X, Y, t)\delta(X)^{-1}\lesssim \frac {\omega_\K(A^+_{c^{-1}\Delta_R,\Lambda},\Delta_{cr}((x,\psi(x)),Y,t))}{\sigma_\K(\Delta_{cr}((x,\psi(x)),Y,t))},
\end{align*}
where $ r:=\delta(X)$, and that $\omega_\K(A^+_{c^{-1}\Delta_R,\Lambda},\Delta_{cr}((x,\psi(x)),Y,t))$ can be expressed as
    \begin{align*}
    &\iiint_{\Delta_{cr}((x,\psi(x)),Y,t)}K_\K(A^+_{c^{-1}\Delta_R,\Lambda},\tilde X,\tilde Y,\tilde t)\, \d\sigma_\K(\tilde X,\tilde Y,\tilde t)\notag\\
    &= \iiint_{\Delta_{cr}((x,\psi(x)),Y,t)}K_\K(A^+_{c^{-1}\Delta_R,\Lambda},\tilde X,\tilde Y,\tilde t)\chi_{\Delta_{2cR}\setminus  \Delta_{R/(2c)}}(\tilde X,\tilde Y,\tilde t)\, \d\sigma_\K(\tilde X,\tilde Y,\tilde t).
\end{align*}
Using \eqref{esto-rererere} we can continue the estimate in \eqref{esto-rerere} to conclude that
\begin{align*}
        \frac{1}{R} \iiint_{D_3^\ast}{ |G_\K | |\nabla_Xu| \d X \d Y \d t}\lesssim \iiint_{\Delta_{cR}\setminus  \Delta_{R/c}}{M}\bigl(K_\K(A^+_{c^{-1}\Delta_R,\Lambda},\cdot)\chi_{\Delta_{2cR}\setminus  \Delta_{R/(2c)}}(\cdot)\bigr )N(u)\, \d\sigma_\K.
\end{align*}
Hence, the term on the left hand side in the last display can estimated by
\begin{align*}
        & \biggl (\iiint_{\Delta_{cR}\setminus  \Delta_{R/c}}|K_\K(A^+_{c^{-1}\Delta_R,\Lambda},\cdot)|^q\, \d\sigma_\K\biggr)^{1/q}\biggl (\iiint_{\Delta_{cR}\setminus  \Delta_{R/c}}|N(u)|^p\, \d\sigma_\K\biggr)^{1/p}\notag\\
        &\lesssim (\sigma_\K(\Delta_{cR}))^{1/q-1}\biggl (\iiint_{\Sigma\setminus  \Delta_{R/c}}|N(u)|^p\, \d\sigma_\K\biggr)^{1/p}\to 0,
\end{align*}
as $R\to\infty$. This completes the estimate of the term in $T_3$  containing the integrand  $|G_\K | |\nabla_Xu|$. The term containing the integrand
$|\nabla_XG_\K | |u|$ can be estimated in a similar manner. We omit further details and claim that
\begin{equation}\label{2limitre}
T_3 \rightarrow 0\:\text{ as }R\rightarrow \infty.
\end{equation}

To summarize, we have proved that
\begin{align}\label{eq9+moabcgapa}
      |u(\hat X,\hat Y,\hat t)|&\lesssim\lim_{R\to\infty}\limsup_{\epsilon\to 0}\  (I+II+III)\notag\\
      &\lesssim \lim_{R\to\infty}\limsup_{\epsilon\to 0}\ (T_1+T_2+T_3)=0.
    \end{align}
    I.e. $|u(\hat X,\hat Y,\hat t)|=0$, and as $(\hat X,\hat Y,\hat t)$ is an arbitrary but fixed point in the argument, we can conclude that $u\equiv 0$ in $\Omega\times \mathbb R^m\times \mathbb R$. This completes the proof of uniqueness and hence the proof of Theorem \ref{DPequiv}.

\section{An alternative proof of Theorem \ref{per1-} along the lines of \cite{FSa}}\label{fabes}

In this section we give,  as we believe that the argument may be of independent interest in the case of operators of Kolmogorov type, a proof of the key estimate underlying Theorem \ref{per1-} using Rellich type inequalities instead of the structural theorem. Hence, the proof is along the lines of the corresponding proof for the heat equation in \cite{FSa}. To avoid formal calculations and manipulations
we will, for simplicity, throughout the section assume
        \begin{eqnarray}\label{eq2+}
    \mbox{\eqref{keyassump} and that $\partial\Omega$ is $C^\infty$-smooth.}
    \end{eqnarray}
    The assumptions in \eqref{eq2+} will only be used in a qualitative fashion and the constants of our quantitative estimates will only depend on $m$, $\kappa$ and $M$. The general case follows by  approximation arguments that we leave to the interested reader.

      In addition to \eqref{eq2+} we also assume \eqref{eq:Aperiod-}, i.e. that $A$ is independent of $x_m$.  Then the unique bounded solution to the Dirichlet problem $\L_\K u=0$ in $\Omega\times\R^m\times\R$, $u=f\in C_0(\partial\Omega\times\R^m\times\R)$, equals
\begin{eqnarray*}%\label{G2}
u(\hat X,\hat Y,\hat t)=\iiint_{\partial\Omega\times\mathbb R^m\times\mathbb R} K_\K(\hat X,\hat Y,\hat t, X, Y, t)f( X, Y, t)\, \d\sigma_\K( X, Y, t),
\end{eqnarray*}
and due to \eqref{eq2+},
\begin{eqnarray*}%\label{G3}
 K_\K(\hat X,\hat Y,\hat t, X, Y, t):=A( x)\nabla_{ X}G_\K (\hat X,\hat Y,\hat t, X, Y, t)\cdot N(X),
\end{eqnarray*}
for all $(X, Y, t)\in \partial\Omega\times\mathbb R^m\times\mathbb R$ and where $N( X)$ is the outer unit normal to $\partial\Omega$ at $ X\in\partial\Omega$.

We are going to prove that if $\Delta:=\Delta_r(X_0,Y_0,t_0)\subset \partial\Omega\times \mathbb R^m\times \mathbb R$, then
 \begin{equation}\label{Bqrepeat}
 \begin{split}
        &\biggl (\bariiint_{\tilde\Delta}|K_\K (A_{c\Delta,\Lambda}^+, X, Y, t)|^2\, \d\sigma_\K( X, Y, t)\biggr )^{1/2}\\
        &\lesssim\biggl (\bariiint_{\tilde\Delta}|K_\K (A_{c\Delta,\Lambda}^+, X, Y, t)|\, \d\sigma_\K( X, Y, t)\biggr ),
        \end{split}
\end{equation}
for all $\tilde\Delta\subset\Delta$.  In fact, we claim that it suffices to prove  \eqref{Bqrepeat} for $\tilde\Delta=\Delta$. To see this, we assume that
\eqref{Bqrepeat} holds for all $\Delta$ with $\tilde\Delta$ replaced by $\Delta$, and we start by noting that we have the representations
\begin{align*}%\label{G3}
 K_\K(A_{c\Delta,\Lambda}^+, X, Y, t)&=A( x)\nabla_{ X}G_\K (A_{c\Delta,\Lambda}^+, X, Y, t)\cdot N( X)\notag\\
 &=\frac {\d\omega_{\K}}{\d\sigma_\K}(A_{c\Delta,\Lambda}^+,X,Y,t)\\
 &=\lim_{{\hat r}\to 0}\frac {\omega_{\K}(A_{c\Delta,\Lambda}^+, \Delta_{\hat r}(X,Y,t))}{\sigma_\K(\Delta_{\hat r}(X,Y,t))},
\end{align*}
for $(X,Y,t)\in\Delta$. Consider $(X,Y,t)\in\tilde\Delta$ and $\hat r>0$ small. Writing $\hat\Delta:=\Delta_{\hat r}(X,Y,t)$ and
\begin{align}\label{G3A}
\frac {\omega_{\K}(A_{c\Delta,\Lambda}^+, \hat\Delta)}{\sigma_\K(\hat\Delta)}&=\frac {\omega_{\K}(A_{c\Delta,\Lambda}^+, \hat\Delta)\  \omega_{\K}(A_{c\tilde\Delta,\Lambda}^+, \hat\Delta)}{\omega_{\K}(A_{c\tilde \Delta,\Lambda}^+, \hat\Delta)\  \sigma_\K(\hat\Delta)},
\end{align}
we  first apply Lemma \ref{greenmeasurerelation} to deduce
\begin{align}\label{G3B}
\frac {\omega_{\K}(A_{c\Delta,\Lambda}^+, \hat\Delta)}{\omega_{\K}(A_{c\tilde \Delta,\Lambda}^+, \hat\Delta)}\lesssim\frac {G_{\K}(A_{c\Delta,\Lambda}^+, A_{c\hat \Delta,\Lambda}^-)}{G_{\K}(A_{c\tilde \Delta,\Lambda}^+, A_{c\hat \Delta,\Lambda}^-)}.
\end{align}
Next, applying Theorem \ref{thm:quotients} in \eqref{G3B}, and passing to the limit by letting $\hat r\to 0$ in \eqref{G3A},
\begin{align*}%\label{G3}
 K_\K(A_{c\Delta,\Lambda}^+, X, Y, t)\lesssim  \frac {G_{\K}(A_{c\Delta,\Lambda}^+, A_{4\tilde \Delta,\Lambda}^-)}{G_{\K}(A_{c\tilde \Delta,\Lambda}^+, A_{4\tilde \Delta,\Lambda}^-)} K_\K(A_{c\tilde\Delta,\Lambda}^+, X, Y, t).
\end{align*}
Using this, and \eqref{Bqrepeat} with $\Delta$ replaced by $\tilde\Delta$ (which holds by the assumption), we deduce
\begin{align}\label{Bqrepeata}
        &\biggl (\bariiint_{\tilde\Delta}|K_\K (A_{c\Delta,\Lambda}^+, X, Y, t)|^2\, \d\sigma_\K( X, Y, t)\biggr )^{1/2}\lesssim \frac {G_{\K}(A_{c\Delta,\Lambda}^+, A_{4\tilde \Delta,\Lambda}^-)}{G_{\K}(A_{c\tilde \Delta,\Lambda}^+, A_{4\tilde \Delta,\Lambda}^-)\ \sigma_\K(\tilde\Delta)}.
\end{align}
However, again using the bound $G_{\K}(A_{c\tilde \Delta,\Lambda}^+, A_{4\tilde \Delta,\Lambda}^-)\gtrsim \tilde r^{2-{\bf q}}$, see \eqref{clam}, we see that
\begin{align}\label{Bqrepeatb}
 \frac {G_{\K}(A_{c\Delta,\Lambda}^+, A_{4\tilde \Delta,\Lambda}^-)}{G_{\K}(A_{c\tilde \Delta,\Lambda}^+, A_{4\tilde \Delta,\Lambda}^-)}\frac 1{\sigma_\K(\tilde\Delta)}\lesssim
 \tilde r^{-1}G_{\K}(A_{c\Delta,\Lambda}^+, A_{4\tilde \Delta,\Lambda}^-).
\end{align}
Next, using Lemma \ref{greenestimate}, Lemma \ref{greenmeasurerelation} and Theorem \ref{thm:doub}, in that order, we deduce
\begin{align}\label{Bqrepeatc}
G_{\K}(A_{c\Delta,\Lambda}^+, A_{4\tilde \Delta,\Lambda}^-)\lesssim \tilde r^{2-{\bf q}}\omega_{\K}(A_{c\Delta,\Lambda}^+, \tilde \Delta),
\end{align}
and hence, by combining the estimates above, see that
\begin{align}\label{Bqrepeatd}
        &\biggl (\bariiint_{\tilde\Delta}|K_\K (A_{c\Delta,\Lambda}^+, X, Y, t)|^2\, \d\sigma_\K( X, Y, t)\biggr )^{1/2}\lesssim \frac {\omega_{\K}(A_{c\Delta,\Lambda}^+, \tilde \Delta)}{\sigma_\K(\tilde\Delta)},
\end{align}
which completes the proof of our claim.

Based on the above it remains to prove  \eqref{Bqrepeat} for $\tilde\Delta=\Delta$ and the rest of the proof is devoted to this. We note that we can without loss of generality assume that $(X_0,Y_0 ,t_0)=(0,0,0)$. A key observation in the following argument, and this is a consequence of that $A$ and $\Omega$ are independent of $(Y,t)$, is that
\begin{eqnarray*}%\label{G3}
\mbox{$ K_\K(\hat X,\hat Y,\hat t, X, Y, t)$ depends on $(\hat Y,\hat t,  Y, t)$ only through the differences $(\hat Y-Y)$, $(\hat t-t)$.}
\end{eqnarray*}
In particular,
\begin{eqnarray}\label{G3}
K_\K(\hat X,\hat Y,\hat t, X, Y, t)=K_\K(\hat X,\hat Y-Y,\hat t-t,X,0,0).
\end{eqnarray}
Note that $ \Delta $ is invariant under the change of coordinates $( X, Y, t)\to ( X,- Y,- t)$. Hence,
\begin{align*}
I&:=\iiint_\Delta |K_\K(A_{c\Delta,\Lambda}^+ , X, Y, t)|^2\, \d\sigma_\K( X, Y, t)\notag\\
&=(-1)^{m+1}\iiint_{\Delta }|K_\K(A_{c\Delta,\Lambda}^+, X,- Y,- t)|^2\, \d\sigma_\K( X, Y, t).
\end{align*}
Using \eqref{G3},  Harnack's inequality, i.e. Lemma \ref{harnack}, and more specifically Lemma \ref{lem4.7}, we see that
\begin{eqnarray*}
K_\K(A_{c\Delta,\Lambda}^+, X,-  Y,- t)\lesssim K_\K(A_{4c \Delta ,\Lambda}^+, X, Y, t)
\end{eqnarray*}
for all $( X, Y, t)\in \Delta $. Hence,
\begin{eqnarray}
|K_\K(A_{c\Delta,\Lambda}^+, X,-  Y,- t)|^2\lesssim K_\K(A_{c\Delta,\Lambda}^+, X,-  Y,- t)K_\K(A_{4c \Delta ,\Lambda}^+, X, Y, t)
\end{eqnarray}
for all $( X, Y, t)\in \Delta $. Let $$\phi\in C_0^\infty
(\mathbb R^{N+1}\setminus\bigl(\{A_{c\Delta,\Lambda}^+\}\cup\{A_{4c\Delta ,\Lambda}^+\}\bigr)$$ be such that
  \begin{eqnarray}\label{G5ap}
  \phi( X, Y, t)=1,
\end{eqnarray}
whenever  $( X, Y, t)=((x,x_m),Y,t)$ is such that $( x, Y, t)\in [-r,r]^{m-1}\times
[-r^3,r^3]^m\times[-r^2,r^2]$, $ x_m\in [\psi(x)-r/16,\psi(x)+r/16]$, and
  \begin{eqnarray}\label{G5ap+}
  \phi( X, Y, t)=0,
\end{eqnarray}
whenever $( X, Y, t)=((x,x_m),Y,t)$ is in the complement of the set defined through the restrictions $( x, Y,  t)\in[-2r,2r]^{m-1}\times
[-(2r)^3,(2r)^3]^m\times[-(2r)^2,(2r)^2]$, $ x_m\in [\psi(x)-r/8,\psi(x)+r/8]$. Furthermore, we choose $\phi$ so that
  \begin{eqnarray}\label{G5ap++}
   |\nabla_X\phi( X, Y, t)|\lesssim r^{-1},\ |(X\cdot\nabla_Y-\partial_t)\phi( X, Y, t)|\lesssim r^{-2},
\end{eqnarray}
whenever $( X, Y, t)\in \mathbb R^{N+1}$. We introduce
\begin{eqnarray}
   v( X, Y, t):=G_\K (A_{c\Delta,\Lambda}^+, X,- Y,-  t),\ \tilde v( X, Y, t):=G_\K (A_{4c\Delta ,\Lambda}^+, X, Y, t),
     \end{eqnarray}
     and
  \begin{eqnarray}
\Psi( X, Y, t):= \phi( X, Y, t)\partial_{ x_m}v( X, Y, t).
  \end{eqnarray}
Recalling that $\mathcal{L}_{X,Y,t}^\ast=\nabla_X\cdot(A(X)\nabla_X)-X\cdot\nabla_Y+\partial_t$ and using the definition of the Green function, we see that
\begin{align*}
0&=\iiint_{\Omega\times\mathbb R^m\times\mathbb R} \mathcal{L}^\ast G_\K (A_{4c\Delta ,\Lambda}^+, X, Y, t)\Psi( X, Y, t)\, \d X\d Y\d t\notag\\
&=\iiint_{\Omega\times\mathbb R^m\times\mathbb R} \mathcal{L}^\ast \tilde v( X, Y, t)\Psi( X, Y, t)\, \d X\d Y\d t.
\end{align*}
Hence
\begin{align*}
0=&\iiint_{\Omega\times\mathbb R^m\times\mathbb R} \bigl (\mathcal{L}^\ast \tilde v( X, Y, t)\Psi( X, Y, t)- \tilde v( X, Y, t)\mathcal{L}\Psi( X, Y, t)\bigr )\, \d X\d Y\d t\notag\\
&+\iiint_{\Omega\times\mathbb R^m\times\mathbb R} \tilde v( X, Y, t)\mathcal{L}\Psi( X, Y, t)\, \d X\d Y\d t.
\end{align*}
Using this identity, and integrating by parts,
\begin{equation}\label{G4a}
\begin{split}
0=& \ \iiint_{\partial \Omega\times\mathbb R^m\times\mathbb R} { K_\K(A_{4c\Delta ,\Lambda}^+, X, Y, t) } \Psi( X, Y, t)\, \d \sigma_\K ( X, Y, t)\\
&+\iiint_{\Omega\times\mathbb R^m\times\mathbb R} \tilde v( X, Y, t)\mathcal{L}\Psi( X, Y, t)\, \d X\d Y\d t.
\end{split}
\end{equation}
Note that by construction, $\Psi( X, Y, t)=\partial_{ x_m}v( X, Y, t)$ if $( X, Y, t)\in \Delta $. Consider the vector field $A( x)N(  X)$. Obviously, $A( x)N( X)\cdot N(  X)\leq \kappa$ by the boundedness of $A$ and hence we can write
$$e_m=T( X)+c( X)A( x)N( X)$$
for all $( X, Y, t)\in \Delta $ and for some function $c(\cdot)$ such that $c(  X)\geq  c(m,\kappa,M)$ for all $( X, Y, t)\in \Delta $. Here $T(  X)$ denote a vector tangent to $\partial\Omega$ at $   X$. Using these observations we see that
\begin{eqnarray*}
\Psi( X, Y, t)=\partial_{  x_m}v( X, Y, t)=c(   X)A(   x)N(  X)\cdot\nabla_{   X}v( X, Y, t)
\end{eqnarray*}
whenever $( X, Y, t)\in \Delta $. In particular, using this and the fact that $K_\K(A_{c\Delta ,\Lambda}^+)$ {and $\Psi$ are} non-negative functions,
\begin{eqnarray*}
I\lesssim \biggl |\iiint_{\Omega\times\mathbb R^m\times\mathbb R} \tilde v( X, Y, t)\mathcal{L}\Psi( X, Y, t)\, \d  X \d Y\d  t\biggr |.
\end{eqnarray*}
We next observe that
\begin{align*}%\label{G4c}
\mathcal{L}\Psi( X, Y, t)&=(\nabla_{  X}(A(   x)\nabla_{   X})+{   X}\cdot\nabla_{  Y}-\partial_{   t})\Psi\\
&=2A(   X)\nabla_{   X}(\partial_{  x_m}v)\nabla_{  X}\phi+\partial_{ x_m}v\mathcal{L}\phi+\phi\mathcal{L}(\partial_{x_m}v),
\end{align*}
and that
\begin{eqnarray*}%\label{G4c}
\mathcal{L}v( X, Y, t)=\mathcal{L}(G(A_{c\Delta,\Lambda}, X,- Y,- t))=(\mathcal{L}^\ast G_\K )(A_{c\Delta,\Lambda}, X,- Y,- t)=0.
\end{eqnarray*}
Using this we see that
$$\mathcal{L}(\partial_{ x_m}v)=\mathcal{L}(\partial_{ x_m}v)-\partial_{ x_m}\mathcal{L}(v)=\partial_{ y_m}v.$$
In particular,
\begin{eqnarray*}%\label{G4c}
\mathcal{L}\Psi( X, Y, t)=2A(  x)\nabla_{  X}(\partial_{ x_m}v)\nabla_{  X}\phi+\partial_{ x_m}v\mathcal{L}\phi+\phi\partial_{ y_m}v.
\end{eqnarray*}
We note that these calculations essentially only uses that $A$ is independent of $x_m$. Recall that $\phi$ satisfies \eqref{G5ap}-\eqref{G5ap++} and let $E=(\Omega\times\R^m\times\R)\cap\overline{\{( X, Y, t)\mid \phi( X, Y, t)\neq 0\}}$.  Using this notation and  elementary manipulations,
\begin{eqnarray*}%\label{G4e}
I\lesssim I_1+I_2+I_3+I_4
\end{eqnarray*}
where
\begin{align*}%\label{G4g}
I_1&:=r^{-2}\iiint_{E} |\nabla_{ X}G_\K (A_{c\Delta,\Lambda}^+, X,- Y,- t)|G_\K (A_{4c \Delta ,\Lambda}^+, X, Y, t)\, \d   X \d  Y\d   t,\notag\\
I_2&:=r^{-1}\iiint_{E} |\nabla_{ X}G_\K (A_{c\Delta,\Lambda}^+, X,- Y,- t)|\nabla_{ X} G_\K (A_{4c \Delta ,\Lambda}^+, X, Y, t)|\, \d   X \d  Y\d   t,\notag\\
I_3&:=r^{-1}\iiint_{E} |\nabla_{ X}\partial_{ x_m}G_\K (A_{c\Delta,\Lambda}^+, X,- Y,- t)|G_\K (A_{4c \Delta ,\Lambda}^+, X, Y, t)\, \d   X \d  Y\d   t,\notag\\
I_4&:=\iiint_{E} |\partial_{ y_m}G_\K (A_{c\Delta,\Lambda}^+, X,- Y,- t)|G_\K (A_{4c \Delta ,\Lambda}^+, X, Y, t)\, \d   X \d  Y\d   t.
\end{align*}
Using the energy estimate of Lemma \ref{lem1en}, and that
$$|G_\K (A_{c\Delta,\Lambda}^+, X,- Y,- t)|+|G_\K (A_{4c \Delta ,\Lambda}^+, X, Y, t)|\lesssim r^{2-{\bf q}}$$
whenever $(X,Y,t)\in E$,
we deduce that
\begin{eqnarray*}%\label{G4h}
I_1+I_2\lesssim \sigma_\K(\Delta )^{-1}.
\end{eqnarray*}
Similarly, using a slightly more involved argument, a Whitney decomposition, Lemma \ref{lem1en} and that the fact that $A$ is independent of $x_m$, {we can proceed in a similar manner as in the proof of Lemma 2.6 in \cite{N1} to also deduce that}
\begin{eqnarray*}%\label{G4h}
I_3+I_4\lesssim \sigma_\K(\Delta )^{-1}.
\end{eqnarray*}
Putting these estimates together we can conclude that
\begin{align*}
\iiint_\Delta |K_\K(A_{c\Delta,\Lambda}^+ , X, Y, t)|^2\, \d\sigma_\K( X, Y, t)=I\lesssim \sigma_\K(\Delta )^{-1},
\end{align*}
whenever $\Delta \subset\partial\Omega\times\R^m\times\R$. Furthermore, as $1\lesssim \omega_\K(A_{c\Delta,\Lambda}^+ ,\Delta )$, we have
\begin{eqnarray*}
\biggl (\bariiint_\Delta |K_\K(A_{c\Delta,\Lambda}^+ , X, Y, t)|^2\, \d\sigma_\K\biggr )^{1/2}\lesssim \biggl (\bariiint_{\Delta }|K_\K(A_{c\Delta,\Lambda}^+ ,X,Y,t)|\, \d\sigma_\K\biggr ),
\end{eqnarray*}
which is  \eqref{Bqrepeat} with $\tilde\Delta=\Delta$. This completes the proof.

\section{Applications to homogenization: Theorem \ref{th:homogenization}}\label{sec4}

By making the change of variables
$(X,Y,t)\mapsto(\tilde X,\tilde  Y,\tilde  t)$, $(X,Y,t)=(\eps \tilde X,\eps^3\tilde Y,\eps^2 \tilde t)$, the boundary
$$\partial \Omega \times \mathbb R^m\times \mathbb R =\{(X,Y,t)\in \mathbb R^m\times \mathbb R^m\times \mathbb R \mid x_m=\psi(x)\}$$
is transformed into
$$\partial \Omega_\eps\times \mathbb R^m\times \mathbb R := \{(\tilde X,\tilde Y,\tilde t)\in \mathbb R^m\times \mathbb R^m\times \mathbb R\mid \tilde x_m=\psi_\eps(\tilde x)\},$$ where
$\psi_\eps(x):=\eps^{-1}\psi(\eps x)$.
Note that $\psi$ and $\psi_\eps$ have the same Lipschitz constant. Let
$$v_\eps(\tilde X,\tilde Y,\tilde t) := u_\eps(X,Y,t),\
f_\eps(\tilde x,\psi_\eps(\tilde x),\tilde Y,\tilde t) := f(x,\psi(x), Y, t).$$
Then,
\begin{equation}\label{epsscalea}
\left\{\begin{aligned}
\L_\K^\eps u_\eps & = 0 \quad \text{in }\Omega\times \mathbb R^m\times \mathbb R,\\
u_\eps &= f \quad \text{n.t. on }\partial \Omega\times \mathbb R^m\times \mathbb R,
\end{aligned}\right.
\end{equation}
where $\L_\K^\eps$ is as in \eqref{Lepsilon}, if and only if
\begin{equation}\label{rescale}
\left\{\begin{aligned}
\L_\K^1 v_\eps & = 0 \quad \text{in }\Omega_\eps\times \mathbb R^m\times \mathbb R,\\
v_\eps &=f_\eps \ \  \text{n.t. on }\partial \Omega_\eps\times \mathbb R^m\times \mathbb R.
\end{aligned}\right.
\end{equation}
 By Theorem \ref{per1} we see that  \eqref{rescale} has a unique weak solution which satisfies
 \[
\|N(v_\eps)\|_{L^p(\partial \Omega_\eps\times \mathbb R^m\times \mathbb R,\d\sigma_\K)}
	 \lesssim \|f_\eps\|_{L^p(\partial \Omega_\eps\times \mathbb R^m\times \mathbb R,\d\sigma_\K)}.
\]
Changing back to the $(X,Y,t)$ coordinates, we get that \eqref{epsscalea} has a unique weak solution verifying the estimate
\begin{equation}\label{ubdedWuurell}
\|N(u_\eps)\|_{L^p(\partial \Omega\times \mathbb R^m\times \mathbb R,\d\sigma_\K)}
	 \lesssim \|f\|_{L^p(\partial \Omega\times \mathbb R^m\times \mathbb R,\d\sigma_\K)},
\end{equation}
and in the last two displays the implicit constants are also allowed to depend on $p$, but are independent of $\epsilon$ and $f$. This settles the proof of the first part of  Theorem \ref{th:homogenization}.

To settle the proof of the second part of Theorem \ref{th:homogenization} we want to let $\epsilon\to 0$ and prove, given $f\in L^p(\partial \Omega\times \mathbb R^m\times \mathbb R,\d\sigma_\K)$, that $u_\eps\to\bar u$ and  that $\bar u$ is a weak solution to the Dirichlet problem
\begin{equation}\label{dpweakhomo}
\begin{cases}
	\bar \L_{\K}\bar u = 0  &\text{in} \ \Omega\times \mathbb R^m\times \mathbb R, \\
     \bar u = f  & \text{n.t. on} \ \partial \Omega\times \mathbb R^m\times \mathbb R,
\end{cases}
\end{equation}
and that
\begin{equation}\label{ubdedWuurell+}
\|N(\bar u)\|_{L^p(\partial \Omega\times \mathbb R^m\times \mathbb R,\d\sigma_\K)}
	 \lesssim \|f\|_{L^p(\partial \Omega\times \mathbb R^m\times \mathbb R,\d\sigma_\K)},
\end{equation}
where the implicit constant also is allowed to depend on $p$. Note that $\bar{A}$ is a constant matrix and once existence is established  uniqueness for the problem stated follows from the  uniqueness part of Theorem \ref{DPequiv}. We also note that in the following it suffices to consider the case $p=2$, again by the classical arguments in \cite{CF}.

 Consider $U_X\times U_Y\times J\subset\mathbb R^{N+1}$, where $U_X\subset\mathbb R^{m}$ and $U_Y\subset\mathbb R^{m}$ are bounded domains and $J=(a,b)$ with $-\infty<a<b<\infty$. Assume that $\overline{U_X\times U_Y\times J}$ is contained in  $\Omega\times\mathbb R^m\times\mathbb R$ and that the distance from
$\overline{U_X\times U_Y\times J}$ to $\partial\Omega\times\R^m\times\R$ is $r>0$. By a covering argument with cubes of size say $r/2$,  Lemma \ref{lem1}, and \eqref{ubdedWuurell}, it follows that
$u_\eps$ is uniformly bounded, with respect to $\eps$, in $L^2(U_X\times U_Y\times J)$ whenever $\overline{U_X\times U_Y\times J}\subset \Omega\times\mathbb R^m\times\mathbb R$. Using this, and the energy estimate of Lemma \ref{lem1en}, we can conclude that
\begin{equation}\label{ubdedWuu}
\mbox{$\|\nabla_X u_\eps\|_{L^2(U_X\times U_Y\times J)}$ is uniformly bounded in $\eps$}.
\end{equation}

Next, using \eqref{ubdedWuu} and the weak formulation of the equation
$\L_\K^\eps u_\eps=0$  it follows that $(X\cdot\nabla_Y-\partial_t)u_\eps$ is uniformly bounded, with respect to $\epsilon$, in  $L_{Y,t}^2(U_Y\times J,H_X^{-1}(U_X))$.
Let $ W(U_X\times U_Y\times J)$  be defined as in \eqref{weak1-+}. By the above argument we can conclude, whenever $\overline{U_X\times U_Y\times J}$ is compactly contained in $\Omega\times\mathbb R^m\times\mathbb R$, that
\begin{equation}\label{ubdedW}
    \mbox{$\|u_\eps\|_{W(U_X\times U_Y\times J)}$ is uniformly bounded in $\eps$},
\end{equation}
and, by ellipticity of $A^\eps$, that
\begin{equation}\label{Aebded}
   \|A^\epsilon \nabla_X u_\epsilon\|_{(L^2_{X,Y,t}(U_X\times U_Y \times J))^m} \text{ is uniformly bounded in }\epsilon.
\end{equation}
Using the Sobolev embedding theorem one can prove that there exists a compact injection
\begin{equation}
    W(U_X\times U_Y\times J) \longrightarrow L^2(U_X\times U_Y \times J).
\end{equation}
Using this, \eqref{ubdedW} and \eqref{Aebded} we see that there exists a subsequence of $\{ u_\epsilon\}$, still denoted $\{ u_\epsilon\}$, such that
\begin{equation}
    \begin{split}\label{conva0}
        u_\epsilon \rightarrow \bar u,\:&\text{ in } L^2(U_X\times U_Y \times J),\\
        A^\epsilon\nabla_X u_\epsilon \rightarrow \xi, \:&\text{ weakly in } (L^2(U_X\times U_Y \times J))^m,\\
        (X\cdot\nabla_Y-\partial_t)u_\epsilon \rightarrow (X\cdot\nabla_Y-\partial_t)\bar u, \:&\text{ weakly in } L^2_{Y,t}(U_Y\times J,H_X^{-1}(U_X)).
    \end{split}
\end{equation}
In particular,
\[
u_\eps \rightarrow \bar u,\text{ weakly in } W(U_X\times U_Y\times J).
\]
Furthermore, using this and the local regularity estimate in Lemma \ref{lem1+} we also have that
\begin{align*}
\mbox{$u_\eps \to \bar u$,  locally uniformly in $\Omega\times \mathbb R^m\times \mathbb R$ as $\eps\to 0$}.
\end{align*}

We now have sufficient information to pass to the limit in the weak formulation of the equation $\L_\mathcal{K}^\eps u_\eps=0$ and doing so we obtain
\begin{equation}\label{eq:xieq}
    \begin{split}
     0 =&\iiint_{U_X\times U_{Y}\times J}\ \xi\cdot \nabla_X\phi\, \d X \d Y \d t\\
    &+\iint_{U_{Y}\times J}\ \langle (-X\cdot\nabla_Y+\partial_t)\bar u(\cdot,Y,t),\phi(\cdot,Y,t)\rangle\, \d Y \d t,
\end{split}
\end{equation}
for all $ \phi\in L_{Y,t}^2(U_{Y}\times J,H_{X,0}^1(U_X))$.
We need to show that $\xi = \bar{A}\nabla_X\bar u$. To this end, we consider the functions
\begin{equation}
    w_\alpha^\epsilon(X) := \epsilon w_\alpha({X}/{\epsilon}),
\end{equation}
with $w_\alpha$ defined as in \eqref{barA}.
Following \cite{CD}, we see that
\begin{equation}\label{conva1}
    \begin{split}
        w_\alpha^\epsilon \rightarrow \alpha\cdot X,&\:\text{ weakly in } H_X^1(U_X),\\
        w_\alpha^\epsilon \rightarrow \alpha\cdot X,&\:\text{ in } L^2(U_X).
    \end{split}
\end{equation}
In particular
\begin{equation}\label{conva2}
    A^\epsilon\nabla_X w_\alpha^\epsilon \rightarrow \bar{A}\alpha,\:\text{ weakly in }(L^2(U_X))^m,
\end{equation}
and
\begin{equation}\label{eq:osceq}
    \int{A^\epsilon(X)\nabla_X w_\alpha^\epsilon\cdot\nabla_X\phi \d X} = 0,
\end{equation}
for all $\phi\in C^\infty_0(U_X)$ (see \cite{CD}, Section 8.1).

Pick $\varphi\in C^\infty_0(U_X)$, $\psi\in C^\infty_0(U_Y\times J)$. We choose $\phi=\varphi u_\epsilon \psi$ in \eqref{eq:osceq}, and integrate with respect to $Y$ and $t$,
\begin{equation}\label{eq:testrelation}
\begin{split}
   0=& \iiint{(A^\epsilon(X)\nabla_X w_\alpha^\epsilon\cdot\nabla_X u_\epsilon) \varphi\psi \, \d X \d Y \d t}\\
   &+
    \iiint{(A^\epsilon(X)\nabla_X w_\alpha^\epsilon\cdot\nabla_X \varphi) u_\epsilon \psi \, \d X \d Y \d t}.
\end{split}
\end{equation}
Picking $\varphi w_\alpha^\epsilon\psi$ as a test function in the weak formulation of $\L_\mathcal{K}^\eps u_\eps=0$ yields
\begin{align*}
    0=&\iiint\ ((A^\epsilon(X)\nabla_Xu_\eps\cdot \nabla_Xw_\alpha^\epsilon)\varphi\psi + (A^\epsilon(X)\nabla_Xu_\eps\cdot \nabla_X\varphi)w_\alpha^\epsilon\psi)\, \d X \d Y \d t\notag\\
    &+\iiint\ (X\cdot \nabla_Y\psi-\partial_t\psi)\varphi w_\alpha^\epsilon u_\eps\, \d X \d Y \d t,
\end{align*}
where we have used that $\varphi$ and $w_\alpha^\epsilon$ only depend on $X$ and that $\psi$ only depends on $Y$ and $t$.
Subtracting the expression in the last display  from \eqref{eq:testrelation} yields
\begin{align}\label{reella}
    0=&\iiint((A^\epsilon(X)\nabla_X w_\alpha^\epsilon \cdot \nabla_X \varphi)u_\epsilon \psi - (A^\epsilon(X)\nabla_X u_\epsilon \cdot \nabla_X\varphi)w_\alpha^\epsilon\psi)\, \d X \d Y\d t\notag\\
    &-\iiint(X\cdot\nabla_Y\psi - \partial_t\psi)\varphi w_\alpha^\epsilon u_\epsilon\, \d X \d Y\d t.
\end{align}
Using \eqref{conva0}, \eqref{conva1}, and \eqref{conva2}, we see that
\begin{align*}
\iiint{\left((A^\epsilon(X)\nabla_X w_\alpha^\epsilon \cdot \nabla_X \varphi)u_\epsilon \psi\right) \d X \d Y\d t}&\to
\iiint\left((\bar{A}\alpha \cdot \nabla_X \varphi)\bar u \psi\right) \d X \d Y\d t,\notag\\
\iiint(A^\epsilon(X)\nabla_X u_\epsilon \cdot \nabla_X\varphi)w_\alpha^\epsilon\psi \d X \d Y\d t&\to  \iiint(\xi \cdot \nabla_X\varphi)(\alpha\cdot X)\psi \d X \d Y\d t,\notag\\
\iiint{(X\cdot\nabla_Y\psi - \partial_t\psi)\varphi w_\alpha^\epsilon u_\epsilon\, \d X \d Y\d t}&\to \iiint(X\cdot\nabla_Y\psi - \partial_t\psi) (\alpha\cdot X) \varphi \bar u \d X \d Y\d t,
\end{align*}
as $\epsilon\to 0$. I.e.,  passing to the limit in \eqref{reella} we obtain
\begin{equation*}
    \iiint{\left((\bar{A}\alpha \cdot \nabla_X \varphi)\bar u \psi - (\xi \cdot \nabla_X\varphi)(\alpha\cdot X)\psi - (X\cdot\nabla_Y\psi - \partial_t\psi) (\alpha\cdot X) \varphi \bar u\right) \d X \d Y\d t}=0.
\end{equation*}
Using that
\[
(\nabla_X\varphi)(\alpha\cdot X)\psi = \nabla_X(\varphi(\alpha\cdot X)\psi) - \alpha\varphi\psi,
\]
and  \eqref{eq:xieq}, now with $\phi=(\alpha\cdot X)\varphi\psi$ as test function, we get
\begin{equation}
    \iiint{\left((\bar{A}\alpha \cdot \nabla_X \varphi)\bar u \psi - (\xi\cdot\alpha)\varphi\psi \right) \d X \d Y\d t}=0.
\end{equation}
Since $\bar{A}$ is constant, this implies that
\[
\xi\cdot\alpha = (\bar{A}\nabla_X \bar u)\cdot \alpha,\:\text{ for all }\alpha\in\R^m,
\]
consequently, $\xi = \bar{A}\nabla_X \bar u$. In particular, $\{u_\eps\}_{\eps>0}$ has a subsequence that converges weakly to $\bar  u$  and $\bar  u$ is a weak solution to  $\bar\L_{\K}\bar u=0$ in $\Omega\times\mathbb R^m\times\mathbb R$.

Next, assume  that $f\in C_0(\partial \Omega\times\R^m\times\R)$. Then
\begin{equation}\label{repo}
u_\eps(X,Y,t)=\iiint K_\eps(X,Y,t,\tilde X,\tilde Y,\tilde t)f(\tilde X,\tilde Y,\tilde t)\d\sigma_\K(\tilde X,\tilde Y,\tilde t),
\end{equation}
and we  need to extract a convergent subsequence from the sequence of kernels $\{K_\eps\}$.  Using the representation in \eqref{repo} we see that if  \begin{equation}\label{xtlambda}(X,Y,t)\in U_X\times U_Y\times J\text{ and }\text{dist}(U_X,\partial\Omega\times\R^m\times\R)\ge 2r,\end{equation}
then as above, i.e. again using a covering argument, Lemma \ref{lem1} and \eqref{ubdedWuurell}, we deduce that
\[
\left|\iiint K_\eps(X,Y,t,\tilde X,\tilde Y,\tilde t)f(\tilde X,\tilde Y,\tilde t)\d\sigma_\K(\tilde X,\tilde Y,\tilde t)\right| = |u_\eps(X,Y,t)|\le
c\|f\|_{L^2(\partial \Omega\times \mathbb R^m\times \mathbb R,\d\sigma_\K)},
\]
for some positive constant $c<\infty$ independent of $\epsilon$. It thus follows by duality that $$\|K_\eps(X,Y,t,\cdot,\cdot,\cdot)\|_{L^2(\partial \Omega\times \mathbb R^m\times \mathbb R,\d\sigma_\K)}$$ is bounded uniformly in $\eps$, for $(X,Y,t)$ as in \eqref{xtlambda}. This clearly implies that \[\|K_\eps\|_{L^2( U_X\times U_Y\times J\times \partial \Omega\times \mathbb R^m\times \mathbb R,\d X\d Y\d t\d\sigma_\K)}\] is bounded uniformly in $\eps$. Thus, for a subsequence,
$$ K_\eps \longrightarrow \bar K\text{ as } \eps \to 0, \ \text{weakly in } L^2( U_X\times U_Y\times J\times\partial \Omega\times \mathbb R^m\times \mathbb R,\d X\d Y\d t\d\sigma_\K).$$

Suppose now that $\{u_{\eps_j}\}$ converges weakly in $W(U_X\times U_Y\times J)$ to $\bar u$. Then, by the above argument there exists a subsequence  $\{\eps_{j'}\}$ of $\{\eps_j\}$
such that $K_{\eps_{j'}}$ converges weakly to $\bar K$ in $L^2( U_X\times U_Y\times J\times\partial \Omega\times \mathbb R^m\times \mathbb R,\d X\d Y\d t\d\sigma_\K)$. This implies, as $u_\eps(X,Y,t)\to \bar u(X,Y,t)$, and by continuity for all $(X,Y,t)$ as in \eqref{xtlambda}, that
\begin{align*}
u_\eps(X,Y,t)&=\iiint K_\eps(X,Y,t,\tilde X,\tilde Y,\tilde t)f(\tilde X,\tilde Y,\tilde t)\d\sigma_\K(\tilde X,\tilde Y,\tilde t)\\
&\to \iiint \bar K(X,Y,t,\tilde X,\tilde Y,\tilde t)f(\tilde X,\tilde Y,\tilde t)\d\sigma_\K(\tilde X,\tilde Y,\tilde t)=\bar u(X,Y,t),
\end{align*}
as $\epsilon\to 0$ and for all $(X,Y,t)$ as in \eqref{xtlambda}. As $U_X\times U_Y\times J$ is arbitrary in this argument, we conclude that for a certain subsequence of $\{u_\eps\}_{\eps>0}$,
$$u_\eps \longrightarrow \bar u \text{ weakly in } W_{\text{loc}}(\Omega\times\mathbb R^m\times \mathbb R),$$
and
\begin{equation}\label{eq:Lloc}
    K_\eps \longrightarrow \bar K \text{ weakly in } L^2_{\text{loc}}(\Omega\times\mathbb R^m\times \R\times\partial \Omega\times \mathbb R^m\times \mathbb R,\d X\d Y\d t\d\sigma_\K).
\end{equation}
Furthermore,
\begin{equation*}
\bar\L_\K\bar u=0\quad\text{in }\Omega\times\mathbb R^m\times\mathbb R,
\end{equation*}
and
\begin{align*}
\bar u(X,Y,t)= \iiint \bar K(X,Y,t,\tilde X,\tilde Y,\tilde t)f(\tilde X,\tilde Y,\tilde t)\d\sigma_\K(\tilde X,\tilde Y,\tilde t),
\end{align*}
whenever $(X,Y,t)\in\Omega\times\mathbb R^m\times \mathbb R$.  Note that the space $$L^2_{\text{loc}}(\Omega\times\mathbb R^m\times \R\times\partial \Omega\times \mathbb R^m\times \mathbb R,\d X\d Y\d t\d\sigma_\K)$$ in \eqref{eq:Lloc} should be interpreted as local only in the first three variables $X$, $Y$ and $t$. As $\bar{A}$ is a constant matrix, the Kolmogorov measure $\omega_{\bar\L_{\K}}$ is absolutely continuous with respect to $\sigma_\K$ and this can be seen as a consequence of  Theorem \ref{per1-}. In particular, the problem $D_{\K}^2(\partial \Omega\times \mathbb R^m\times \mathbb R,\d\sigma_\K)$ is  uniquely solvable for the operator $\bar\L_{\K}$ and $\bar K(X,Y,t,\tilde X,\tilde Y,\tilde t)$
is the Radon-Nikodym derivative of the Kolmogorov measure $\omega_{\bar\L_{\K}}(X,Y,t,\cdot)$ with respect to $\sigma_\K$ at $(\tilde X,\tilde Y,\tilde t)\in \partial \Omega\times \mathbb R^m\times \mathbb R$. As a consequence, using Theorem \ref{DPequiv} we can conclude that for $f\in C_0(\partial \Omega\times\R^m\times\R)$ given, $\bar u$ is the unique solution to
the problem in \eqref{dpweakhomo} which satisfies \eqref{ubdedWuurell+}. For $f\in L^2(\partial \Omega\times \mathbb R^m\times \mathbb R,\d\sigma_\K)$ the same conclusion
follows from the density of $C_0(\partial \Omega\times\R^m\times\R)$ in $L^2(\partial \Omega\times \mathbb R^m\times \mathbb R,\d\sigma_\K)$, see the final part in the proof of $(i)$ implies $(ii)$ in Theorem \ref{DPequiv} for reference. Summing up, the  proof of Theorem \ref{th:homogenization} is complete.\\

\noindent{\bf Acknowledgement.} The authors like to thank two anonymous referees for valuable comments and suggestions.

\end{document}